%% file: bivariant.tex
\documentclass[12pt]{amsart}

\usepackage[all]{xy}
\usepackage{amssymb,amsmath,bbm,url,xr,a4wide}
\usepackage[mathscr]{euscript}
\usepackage{mathrsfs}
\input cyracc.def
\DeclareFontFamily{U}{russian}{}
\DeclareFontShape{U}{russian}{m}{n}
        { <5><6> wncyr5
        <7><8><9> wncyr7
        <10><10.95><12><14.4><17.28><20.74><24.88> wncyr10 }{}
\DeclareSymbolFont{Russian}{U}{russian}{m}{n}
\DeclareSymbolFontAlphabet{\mathcyr}{Russian}
\makeatletter
\let\@math@cyr\mathcyr
\renewcommand{\mathcyr}[1]{\@math@cyr{\cyracc #1}}
\makeatother

\newtheorem{thm}{Theorem}[subsection]
\newtheorem{thmi}{Theorem}
\newtheorem*{pthm}{Theorem}
\newtheorem{prop}[thm]{Proposition}
\newtheorem{cor}[thm]{Corollary}
\newtheorem{lm}[thm]{Lemma}

\theoremstyle{definition}
\newtheorem{df}[thm]{Definition}
\newtheorem{ex}[thm]{Example}
\newtheorem{num}[thm]{}

\theoremstyle{remark}
\newtheorem{rem}[thm]{Remark}

\numberwithin{equation}{thm}

\title{Bivariant theories in motivic stable homotopy}
\author{Fr\'ed\'eric D\'eglise}
\thanks{Partially supported by the ANR (grant No. ANR-07-BLAN-042)}
\date{June 2018.}

\input{command}

\begin{document}

\begin{abstract}
The purpose of this work is to
 study the notion of bivariant theory introduced by
 Fulton and MacPherson
 in the context of motivic stable homotopy theory,
 and more generally in the broader framework
 of the Grothendieck six functors formalism.
 We introduce several kinds of bivariant theories associated
 with a suitable ring spectrum
 and we construct 
 a system of orientations (called fundamental classes)
 for global complete intersection morphisms between arbitrary
 schemes. These fundamental classes satisfy all the expected properties
 from classical intersection theory
 and lead to Gysin morphisms, Riemann-Roch formulas as well as
 duality statements, valid for general schemes,
 including singular ones and without need of a base field.
 Applications are numerous, ranging from classical theories
 (Betti homology)
 to generalized theories (algebraic K-theory, algebraic cobordism)
 and more abstractly to \'etale sheaves (torsion and $\ell$-adic)
 and mixed motives.
\end{abstract}

\setcounter{tocdepth} 2

\maketitle

\tableofcontents

\section*{Introduction}

\input{intro}

\section*{Notations and conventions}

All schemes in this paper are assumed to be noetherian
 of finite dimension.\footnote{Note however that one could work with
 arbitrary quasi-compact and quasi-separated schemes.
 Indeed, the stable homotopy category $\SH(S)$ and its six operations,
 which provides many of our examples,
 has been recently extended in \cite[Appendix C]{HoyoisLefschetz} or
 \cite{KhanThesis} to the case $S$ quasi-compact and quasi-separated.} 

We will say that an $S$-scheme $X$, or equivalently its structure morphism,
 is \emph{quasi-projective} (resp. \emph{projective})
 if it admits an $S$-embedding (resp. a closed $S$-embedding)
 into $\PP^n_S$  for a suitable integer $n$;\footnote{For example,
 if one works with
 quasi-projective schemes over a noetherian affine scheme
 (or more generally a noetherian scheme which admits an ample line bundle), 
 then a morphism is proper if and only if it is projective with our
 convention -- use \cite[Cor. 5.3.3]{EGA2}.}

We will also fix a class of morphisms,
 called \emph{gci}\footnote{A short for \emph{global complete intersection}.}
 such that the following properties hold:
\begin{itemize}
\item any gci morphism admits
 a factorization into a regular closed immersion followed by
 a smooth morphism.
\item gci morphisms are stable under base change.
\item For any composable gci morphisms $f$ and $g$,
 one can find a commutative diagram:
$$
\xymatrix@R=6pt@C=30pt{
Z\ar^g[rr]\ar_k[rd] & & Y\ar^g[rr]\ar_i[rd] & & X \\
 & Q\ar_l[rd]\ar_q[ru] && P\ar_p[ru] & \\
 & & R\ar_r[ru] & & 
}
$$
such that $i$, $k$, $l$ are regular closed immersions
 and $p$, $q$, $r$ are smooth morphisms.
\end{itemize}
The main example of such a class is provided by
 quasi-projective local complete intersection morphisms.
 One can also fix some base scheme $S$,
 restrict to $S$-schemes which admits an embedding into
 a smooth $S$-scheme and work with local complete intersection $S$-morphisms
 of such $S$-schemes.

For short, we will use the term {\it s-morphism} for separated morphism.

Given a closed (resp. regular closed) subscheme $Z$ of a scheme $X$,
 we will denote by $B_ZX$ (resp. $N_ZX$) the blow-up (resp. normal bundle)
 of $Z$ in $X$.

In the whole text, 
 $\base$ stands for a sub-category of the category of
 (noetherian finite dimensional) schemes
 such that:
\begin{itemize}
\item $\base$ is closed under finite sums and pullback along
 morphisms of finite type.
\item For any scheme $S$ in $\base$, any quasi-projective
 (resp. smooth) $S$-scheme belongs to $\base$.
\end{itemize}
The main examples we have in mind are either the category of all
 schemes or the category of $F$-schemes for a prime field $F$.

\bigskip

We will use the axiomatic of Grothendieck six functors
 formalism and more specifically the richer axioms of 
 motivic triangulated categories introduced in \cite{CD3}.\footnote{Recall
 this axiomatic amounts, up to some minor changes, to the axioms
 of crossed functors of Ayoub-Voevodsky, \cite{Ayo1}.}
 All motivic triangulated categories introduced here will be assumed
 to be defined over the above fixed category $\base$.

From Section \ref{sec:orientation} to the end of the paper,
 we will make for simplification 
 the assumption that all motivic triangulated 
 categories $\T$ are equipped with a premotivic adjunction:
$$
\can^*:\SH \leftrightarrows \T:\can_*
$$
where $\SH$ is Morel-Voevodsky's stable homotopy category.
 In fact, this assumption occurs in satisfied in all our main examples
 and is justified by a suitable universal property
 (see Remark \ref{rem:SH_initial&ring_sp}).

\bigskip

When we will consider the codimension of a regular closed immersion $Z \rightarrow X$,
 the rank of a virtual vector bundle over $X$
 or the relative dimension of a gci morphism $X \rightarrow Y$,
 we will understand it as a Zariski locally constant function $d:X \rightarrow \ZZ$.
 In other words, $d$ is a function which to a connected component $X_i$ of $X$
 associates an integer $d_i \in \ZZ$.
 To such a function $d$ and to any motivic triangulated category $\T$,
 we can associate a twist $-(d)$ (resp. shift $-[d]$) on the triangulated category $\T(X)$
 by taking twist $-(d_i)$ (resp. shift $_[d_i]$) over the component $\T(X_i)$.
 In that way, we avoid to artificially assume codimensions, ranks or relative dimensions
 are constant.

\section*{Thanks}

I wish to warmly thank Denis-Charles Cisinski, Ofer Gabber,
 Adeel Khan, Alberto Navarro and 
 Fangzhou Jin for discussions, help and interest
 that greatly helped me in writing this paper.
 The referee of the present version has made many helpful
 comments and suggestions which have helped me to correct mistakes,
 to better account for references around the subject dealt with here
 and to improve the writing of this article.
 This work received support from the French "Investissements d'Avenir"
 program, project ISITE-BFC (contract ANR-lS-IDEX-OOOB).

\section{Absolute spectra and associated bivariant theories}

\input{absolute}

\section{Fundamental classes}

\input{fundamental}

\section{Intersection and generalized Riemann-Roch formulas}

\input{intersection}

\section{Absolute purity and duality}\label{sec:duality}

\input{purity}

\bibliographystyle{amsalpha}
\bibliography{bivariant}

\end{document}

%% file: command.tex
\newcommand{\ZZ} {\mathbb Z}
\newcommand{\QQ} {\mathbb Q}
\newcommand{\CC} {\mathbb C}
\renewcommand{\AA} {\mathbb A}
\newcommand{\GG} {\mathbb G_m}
\newcommand{\GGx}[1] {\mathbb G_{m,#1}}
\newcommand{\PP} {\mathbb P}

\newcommand{\T}{\mathscr T} 
\newcommand{\A}{\mathscr A}
\newcommand{\C}{\mathscr C}

\newcommand{\cL}{\mathcal L}
\newcommand{\cO}{\mathcal O}
\newcommand{\E}{\mathbb E} 
\newcommand{\F}{\mathbb F} 
\newcommand{\HH}{\mathbb H} 

\newcommand{\un}{\mathbbm 1} 

\newcommand{\base}{\mathscr S} 
\newcommand{\sm}{\mathscr Sm}
\newcommand{\can}{\tau} 

\DeclareMathOperator{\rk}{rk}

\newcommand{\SH}{S\mathscr{H}}
\newcommand{\DM}{\mathrm{DM}}

\newcommand{\DMB}{\mathrm{DM}_\mathcyr B}
\newcommand{\smod}[1] {#1\mbox{-}\operatorname{mod}}

\newcommand{\ilim}[1] { \underset{#1}{\varinjlim} \ }

\DeclareMathOperator{\Th}{Th}
\DeclareMathOperator{\MTh}{\mathit M Th}
\DeclareMathOperator{\thom}{\mathfrak t}
\DeclareMathOperator{\rthom}{\bar{\mathfrak t}}
\DeclareMathOperator{\fdl}{\bar{\eta}}
\DeclareMathOperator{\tfdl}{\tilde{\eta}}
\DeclareMathOperator{\cfdl}{\eta}
\DeclareMathOperator{\td}{Td}

\newcommand{\biv}[2]{#1^{BM}_{#2}}

\newcommand{\cupp}{{\scriptstyle \cup}}
\newcommand{\refp}{\cdot}

\newcommand{\Sus}{\Sigma^\infty}
\newcommand{\MGL}{\mathbf{MGL}}
\newcommand{\KGL}{\mathbf{KGL}}
\newcommand{\HB}{\mathbf H_{\mathcyr{B}}}

\DeclareMathOperator\Pic{Pic}
\newcommand{\spec}[1] {\operatorname{\mathrm{Spec}}(#1)}

\DeclareMathOperator{\Hom}{Hom}
\DeclareMathOperator{\uHom}{\underline{Hom}}
\DeclareMathOperator{\End}{End}

\newcommand{\pur} { \mathfrak p } 

\DeclareMathOperator{\ch} {ch}
\DeclareMathOperator{\reg} {\mathscr Reg}

\newcommand{\et}{\text{\'et}}
\newcommand{\h}{\mathrm{h}}
\newcommand{\cdh}{{\mathrm{cdh}}}

%% file: intro.tex
\noindent\textbf{\textit{Genealogy} of ideas.}
As this work comes after a long line of ideas about cohomology,
 may be it is worthwhile to draw the following genealogical tree
 of mathematicians and concepts.
$$\label{tree:history}
\hspace{-10pt}\xymatrix@C=-50pt{
&& *+[F]\txt{Riemann, \\ {\it genus, Riemann's inequality} (1857)}
 \ar^/-20pt/{\txt{\it Riemann-Roch formula (1865)}}[dd]
  \ar_/28pt/{\text{\it Betti numbers (1871)}\quad}[ld] & \\
& *+[F]\txt{Poincar\'e, {\it duality} (1895)}
 \ar^/-26pt/{\!\!\!\txt{{\it homology groups}, Noether (1925)}}|/-4pt/{\qquad}[ld]
 \ar|/-4pt/{\qquad}[rd] & &  \\
*+[F]\txt{Borel-Moore homology (1960)} & &
 *+[F]\txt{Grothendieck (ICM 1957) \\ K-theory, GRR formula, duality}\ar[d]
 & \\
&&  *+[F]\txt{Grothendieck, 6 functors formalism}\ar[ld]\ar[rd] & \\
& *+[F]\txt{Bloch-Ogus axioms (1974)}
  && *+[F]\txt{Fulton-MacPherson, \\ bivariant theories (1981)}
}
$$
So our starting point is the bivariant formalism of Fulton
 and MacPherson (\cite{FMP}), which appears
 in one of the ending point of the above tree.
 The ambition of this formalism is to unify homology
 and cohomology into a single theory. More than that:
 it was slowly observed that Poincar\'e duality gives a
 wrong-way functoriality of singular cohomology. The study
 of this phenomena, which was later discovered to occur also in
 homology, was developed from many perspectives all along the
 century. The striking success of the bivariant formalism is to explain
 all these works by the construction of an element,
 generically called \emph{orientation} by Fulton and MacPherson,
 of a suitable bivariant group. 
 
In fact, this theory can be seen as a
 culminating point of the classical theory of characteristic classes.
 A brilliant illustration is given by Fulton and MacPherson's
 interpretation of each of the known Riemann-Roch formulas as a
 comparison between two given orientations.
 Let us now detail these general principles from our point of view,
 based on motivic homotopy theory.

\bigskip

\noindent\textbf{Motivic homotopy theory.}
Our previous works on Gysin morphisms (\cite{Deg8,Deg12})
 naturally lead to the bivariant language. As we will see,
 it allows to treat both the fundamental class of an algebraic cycle
 and the cobordism class of a projective morphism within a single 
 framework. But it was only when the six functors formalism
 was fully available in motivic homotopy theory, after the work
 of Ayoub (\cite{Ayo1}), that we became aware of a plain incorporation
 of bivariant theories into motivic homotopy theory.

One can already trace back this fact in the work of Fulton and
 MacPherson as one of their examples of a bivariant theory,
 in the \'etale setting, already uses the six functors formalism.
 In this work we go further, showing that any representable cohomology
 in $\AA^1$-homotopy admits a canonical extension to a bivariant
 theory. Basically, it applies to any known cohomology theory (in
 algebraic geometry) which is homotopy invariant.

This result is obtained as a by-product of Morel-Voevodsky's motivic
 homotopy theory, but more generally, we use the axioms of
 Ayoub-Voevodsky's cross functors, fundamentally developed by
 Ayoub in \cite{Ayo1}. This theory was amplified later by Cisinski 
 and the author in \cite{CD3} as a general axiomatic of
 Grothendieck' six functors formalism. Such an axiomatic theory,
 called a \emph{triangulated motivic category}, is in the first
 place a triangulated category $\T$ fibered over a 
 fixed suitable category of schemes $\base$ and equipped with
 the classical six operations, $f^*, f_*, f_!, f^!, \otimes, \uHom$.
 There are many concrete realizations of this formalism in the literature
 so we will only recall here the specific axioms added by Ayoub 
 and Voevodsky:
\begin{itemize}
\item \textit{$\AA^1$-homotopy}.-- for any scheme $X$ in $\base$,
 $p:\AA^1_X \rightarrow X$ being the projection of the affine line,
 the adjunction map $1 \rightarrow p_*p^*$ is an isomorphism;
\item \textit{$\PP^1$-stability}.-- for any scheme $X$ in $\base$,
 $p:\PP^1_X \rightarrow X$ being the projection of the projective line
 and
 $s:X \rightarrow \PP^1_X$ the infinite-section,
 the functor $s^!p^*$ is an equivalence of categories.
\end{itemize}
Recall the first property corresponds to the contractibility
 of the affine line and the second one to the invertibility of
 the Tate twist.\footnote{While the first property is often remembered,
 the second one was slightly overlooked at the beginning of the theory
 but appears to be fundamental in the establishment of the six functors
 formalism.} Note already that this axiomatic is satisfied
 in the \'etale setting (torsion and $\ell$-adic coefficients). Besides,
 thanks to
 the work of the motivic homotopy community, there are now many examples
 of such triangulated categories.\footnote{Stable homotopy, mixed motives,
 modules over ring spectra such as K-theory, algebraic cobordism.
 These examples will appear naturally in the course of the text; see
 Example \ref{ex:spectra} in the first place.}
 
\bigskip
 
\noindent\textbf{Absolute ring spectra and bivariant theories.}
From classical and motivic homotopy theories,
 we retain the notion of a ring spectrum but use
 a version adapted to our theoretical context.
 An \emph{absolute ring $\T$-spectrum} will be a cartesian
 section of the category of commutative monoids in $\T$,
 seen as a fibered category over $\base$; concretely,
 the data of a commutative monoid $\E_S$ of the monoidal
 category $\T(S)$ for any scheme $S$ in $\base$,
 with suitable base change isomorphisms $f^*(\E_S) \xrightarrow \sim \E_T$
 associated with any morphism $f:T \rightarrow S$ in $\base$
 (see Definition \ref{df:absolute_sp}).\footnote{The terminology
 is inspired by the terminology used by Beilinson in his formulation
 of the Beilinson's conjectures. In particular, the various absolute 
 cohomologies
 considered by Beilinson are representable by absolute spectra in our sense,
 with $\base$ being the category of all (noetherian, finite dimensional)
 schemes, or that of schemes over a field (for example in the case of Deligne
 cohomology).}
 Our main observation is that to such an object is associated
 not only the classical cohomology theory but also a bigraded 
 bivariant theory: to a separated morphism
 $f:X \rightarrow S$ of finite type and integers $(n,m)  \in \ZZ^2$,
 one associates the abelian group:
$$
\biv \E {n,m}\big(X \xrightarrow f S\big)
  =\underset{(*)}{\underbrace{\Hom_{\T(S)}(f_!(\un_X)(n)[m],\E_S)}}
  \simeq \underset{(**)}{\underbrace{\Hom_{\T(S)}(\un_X(n)[m],f^!\E_S)}}
$$
It usually does not lead to confusions to denote these groups
 $\biv \E {n,m}(X/S)$. Now the word bivariant roughly means
 the following two fundamental properties:
\begin{itemize}
\item \textit{Functoriality}.-- the group $\biv \E {n,m}(X/S)$
 is covariant in $X$ with respect to proper $S$-morphisms
 and covariant in $S$ with respect to any morphism;
\item \textit{Product}.-- given $Y/X$ and $X/S$, there exists
 a product:
$$
\biv \E {n,m}(Y/X) \otimes \biv \E {s,t}(X/S)
 \rightarrow \biv \E {n+s,m+t}(Y/S).
$$
\end{itemize}
Of course, these two structures satisfy various properties;
 we refer the reader to Section \ref{sec:ass_biv} for a complete
 description.\footnote{Note also that we get a particular instance of
 what Fulton and MacPherson call a bivariant theory. However,
 in algebraic geometry, this instance is the most common one.}

Several paths lead naturally to our definition.
 First, our initial motivation comes from the case
 where $f=i:Z \rightarrow S$ is a closed immersion.
 In this case, from the localization triangle attached to $i$ by
 the six functors formalism,
 one realizes that the abelian group $\biv \E {**}(Z/X)$
 is nothing else than the $\E$-cohomology group of $X$ with support
 in $Z$, which naturally receives the refined fundamental classes
 we had built earlier (\cite[2.3.1]{Deg10}).
 Secondly, when $\T(X)=D^b_c(X_\et,\Lambda)$, 
 $\Lambda=\ZZ/\ell^n\ZZ, \ZZ_\ell$ or $\QQ_\ell$, for $\ell$ a prime 
 invertible on $X$, and $\E_X=\Lambda_X$ is the constant sheaf,
 formulas $(*)$ and $(**)$ agree with that considered by Fulton and
 MacPherson in \cite[7.4.1]{FMP} for defining a bivariant \'etale theory.
 Finally, when $S=\spec k$ is the spectrum of a field
 and $\E_X=\un_X$, formula $(**)$ gives an interpretation of
 $\biv \E {**}(X/k)$ as the cohomology with coefficients in the object
 $f^!(\un_k)$, which is Grothendieck's formula for the dualizing
 complex.\footnote{Recall the object $f^!(\un_k)$ is indeed a dualizing
 object in $\T(X)$ if $\T$ is $\QQ$-linear (see \cite{CD3})
 or under suitable assumptions of resolution of singularities
 (see \cite{Ayo1}).} In other words, $\biv \E {**}(X/k)$ is
 the analogue of Borel-Moore homology defined in \cite{BM}.\footnote{In
 fact we prove in Example \ref{ex:coh_support}(3)
 that when $\T$ is Morel-Voevodsky's stable homotopy
 category and $\E$ is the spectrum representing Betti cohomology,
 this abelian group is precisely Borel-Moore homology -- the twists
 in that case do not change the group up to isomorphism.}
 This last example justifies our notation: the letters
 ``BM'' stands for Borel-Moore and we use the terminology
 \emph{Borel-Moore homology} for the bivariant theory $\biv \E {**}$.

Besides, we can define other bivariant theories from 
 the absolute ring spectrum $\E$.
 In fact, we remark that one can attach to $\E$
 four theories: cohomology as usual, Borel-Moore homology
 defined by the above formula $(**)$, but also
 cohomology with compact support and (plain) homology.
 The two later theories are in fact bivariant theories:
 they can be defined not only for $k$-schemes but for
 any morphism of schemes (in $\base$). We refer the reader
 to Definition \ref{df:compact_support_theories} but
 one can also guess the formulas: they
 are the variants of formula $(**)$ obtained from the various
 possibility of combining the functors $f_*,f_!, f^*, f^!$.
 Note our main result, which will be stated just below, will give
 new structures for each of the four theories.

\bigskip
 
\noindent\textbf{Orientations and fundamental classes.}
The key idea of this work is that one can recast previously
 known constructions of Gysin morphisms, based on the orientation
 theory of motivic ring spectra, in the bivariant framework
 provided by formula $(*)$.
 Recall for the sake of notations that an orientation $c$
 of an absolute ring $\T$-spectrum $\E$ is roughly a family of classes
 $c_S$ in $\E^{2,1}(\PP^\infty_S)$, cohomology of the infinite
 projective space, for schemes $S$ in $\base$,
 satisfying suitable conditions (see Definition
 \ref{df:orientation} for the precise formulation). 
 
One can attach to the orientation
 $c$ a complete formalism of Chern classes, and even Chern classes
 with supports (see Proposition \ref{prop:chern_classes}
 and \ref{df:Chern_support}). Our main result is the following
 construction of characteristic classes of morphisms, whose
 main advantage against previous constructions is that it works
 for singular schemes and without a base field.
\begin{thmi}[see Theorem \ref{thm:BM_fundamental_classes}]
\label{thmi:BM_fundamental_classes}
Consider an absolute ring $\T$-spectrum $\E$
 with a given orientation $c$.

Then for any global complete intersection\footnote{We say $f$ is a global
 complete intersection if it admits a factorization $f=p \circ i$
 where $p$ is smooth separated of finite type and $i$ is a regular
 closed immersion;} morphism $f:X \rightarrow S$ of relative dimension $d$,
 there exists a class $\fdl_f$ in $\E_{2d,d}^{BM}(X/S)$,
 called the \emph{fundamental class} of $f$ associated with $(\E,c)$,
 with the following properties:
\begin{enumerate}
\item \emph{Normalization}.-- If $f=i:D \rightarrow X$
 is the immersion of a regular divisor then $\fdl_i=c_1^D(\mathcal O(-D))$
 (where $\mathcal O(-D)$ is the dual of the invertible sheaf parametrizing
  $i$; see Example \ref{ex:Navarro}).
\item \emph{Associativity}.-- For any composable morphisms 
 $Y \xrightarrow g X \xrightarrow f S$,
 $\fdl_{f \circ g}=\fdl_g.\fdl_f$.
\item \emph{Base change}.-- Given a morphism $p:T \rightarrow S$
 which is transversal to the map $f$, then $p^*(\fdl_f)=\fdl_{f \times_S T}$
 (see Example \ref{ex:transversal}).
\item \emph{Excess intersection}.-- Given a morphism $p:T \rightarrow S$
 such that the base change $f \times_S T$ is a local complete
 intersection, $p^*(\fdl_f)=e(\xi).\fdl_{f \times_S T}$
 (the class $e(\xi)$ stands for the Euler class of the excess intersection
 bundle, as in Fulton's classical formula for Chow groups;
 see Proposition \ref{prop:excess_generalized}).
\item \emph{Ramification formula}.-- Let $p:Y \rightarrow X$ be a dominant
 morphism of normal schemes, $i:D \rightarrow X$ be the immersion
 of a regular divisor, $(E_j)_{j=1,\hdots,r}$ the family of irreducible
 components of $p^{-1}(D)$ and $m_j$ the ramification index\footnote{or,
 in other words, the intersection multiplicity of $E_j$ in the pullback
 of $D$ along $p$, see \cite[4.3.7]{Ful}; note the integer
 $m_j$ is nothing else than the ramification index in the classical sense
 of the extension $\mathcal O_{Y,E_j}/\mathcal O_{X,D}$
 of discrete valuation rings;}
 of $f$ along $E_j$. Then one has:
$$
p^*(\fdl_X(D))=[m_1]_F.\fdl_Y(E_1)+_F \hdots +_F [m_r]_F.\fdl_Y(E_r);
$$
here $([m]_F.-)$ stands for the power series corresponding to the $m$-th self
 addition in the sense of the formal group law $F$ attached with the 
 orientation $c$ (see Corollary \ref{cor:ramification}).
\item \emph{Duality}.-- When $f:X \rightarrow S$ is smooth
 or is the section of a smooth morphism, the multiplication map:
$$
\biv \E {**}(Y/X) \rightarrow \biv \E {**}(Y/S), y \mapsto y.\fdl_f
$$
is an isomorphism (see Remark \ref{rem:formal_duality}).
\end{enumerate}
\end{thmi}
Before discussing uniqueness statements of our construction,
 let us explain the existence part. As mentioned before the statement
 of the theorem, we use previous methods of construction of Gysin
 maps. The first one, in the case where $f$ is a regular closed immersion,
 is a construction whose motivic homotopy formulation is due to
 Navarro (cf. \cite{Nav}), based on a method of Gabber 
 (cf. \cite[chap. 7]{Gabber}) who treated the case of \'etale
 cohomology.\footnote{The generalization of the method of Gabber to motivic 
 homotopy is non trivial because the Chern classes associated to
 an oriented ring spectrum are in general non additive:
 $c_1(L \otimes L') \neq c_1(L)+c_1(L')$. See Prop.
 \ref{prop:chern_classes}.}
 The second method, for a smooth quasi-projective morphism $f$,
 is given as a corollary of Ayoub's fundamental work on the six
 functors formalism in motivic homotopy (cf. \cite{Ayo1}).
 Actually, the fundamental class $\fdl_f$ is essentially induced
 by the purity isomorphism associated with $f$ by the six functors
 formalism (see Paragraph \ref{num:fdl_smooth} for the actual construction).

After realizing that these two methods can actually be formulated
 in the bivariant language, the main point in the proof of the preceding
 theorem is to show that ``they glue''. This can be expressed
 in a simple equation that we let the reader discover in 
 the key lemma \ref{lm:key}.

From the description of our construction, the uniqueness of the
 family of classes $\fdl_f$ is clear: they are the unique
 formalism that extends the fundamental classes obtained 
 respectively from Navarro's and Ayoub's methods. A more satisfactory
 statement is obtained if one restricts to quasi-projective
 local complete intersection morphisms $f$: then the family of fundamental
 classes $\fdl_f$ is uniquely characterized by properties
 (1), (2), (3), (4) and (6). Actually, we can even replace (4)
 by the particular case where $p$ is a blow-up: see Theorem
 \ref{thm:uniqueness}.

\bigskip

\noindent\textbf{Riemann-Roch formulas.}
According to Fulton and MacPherson,
 one of the motivations for developing bivariant theory was the aim
 to
 synthesize several Riemann-Roch formulas (respectively
 by Grothendieck, by Baum, Fulton and MacPherson,
 and by Verdier; see \cite[section 0.1]{FMP}).
 The underlying principle is that the classical Chern character
 corresponds in fact to a natural transformation of bivariant
 theories, suitably compatible with the structure of a bivariant
 theory (called a \emph{Grothendieck transformation} in \cite{FMP}).
 Then the general Riemann-Roch formula essentially comes from the 
 effect of a Grothendieck transformation to a given theory
 of fundamental classes, as the one given in the above theorem.

The same goal has mainly contributed to our choice of framework.
 Indeed we show how to produce Grothendieck transformations 
 by considering
 a suitable functor $\varphi^*:\T \rightarrow \T'$
 --- more precisely a premotivic adjunction of motivic
 triangulated categories in the sense of \cite{CD3} ---
 an absolute ring $\T$-spectrum $\E$ (resp. $\T'$-spectrum $\F$)
 and a morphism of absolute ring $\T'$-spectra:
$$
\phi:\varphi^*(\E) \rightarrow \F;
$$
see Definition \ref{df:morph_abs_sp} for details.
 It is clear from our choice of definition that there
 is an induced Grothendieck transformation:
$$
\phi_*:\biv \E {**}(X/S) \rightarrow \biv \F {**}(X/S).
$$
Then we can prove the following general Riemann-Roch theorem.
\begin{thmi}[see Theorem \ref{thm:RR}]
Consider a morphism $(\varphi^*,\phi)$ as above
 and respective orientations $c$ of $\E$ and $d$ of $\F$.

Then there exists a canonical \emph{Todd class} morphism:
$$
\td_\phi:K_0(X) \rightarrow \F^{00}(X)^\times
$$
from the Grothendieck group of vector bundles on $X$
 to the group of units of $\F$-cohomology classes on $X$
 with degree $(0,0)$, natural in $X$ and such that for
 any global complete intersection morphism $f:X \rightarrow S$
 with virtual tangent bundle $\tau_f$, the following relation
 holds:
$$
\phi_*(\fdl_f^c)=\td_\phi(\tau_f).\fdl_f^d
$$
where $\fdl_f^c$ (resp. $\fdl_f^d$) is the fundamental
 class associated with $f$ and $(\E,c)$ (resp. $(F,d)$)
 in the above theorem.
\end{thmi}
At this point of the theory, the proof is straightforward
 and actually essentially works as the original proof of
 Grothendieck. But the beauty of our theorem is that it
 contains all previously known Riemann-Roch formulas
 as a particular case. We will illustrate this by 
 the concrete applications to come.
 
\bigskip

\noindent\textbf{Gysin morphisms.}
Let us quit the realm of general principles now
 and show the new results that our theory brings.
 As devised by Fulton and MacPherson, the interest
 of fundamental classes\footnote{Recall they call them
 ``orientations''.} is that they induce wrong-way morphisms
 in cohomology, and in fact as we remarked in this paper,
 in the four theories associated to $\E$ above.
 We generically call
 these morphisms \emph{Gysin morphisms}.
 Here are our main examples of applications.
\begin{itemize}
\item \emph{\'Etale cohomology} (resp. \emph{Borel-Moore
\'etale homology})
 with coefficients in
 a ring $\Lambda$ is \emph{covariant with respect
 to proper global complete intersection morphisms}
 (resp. \emph{contravariant with respect
 to global complete intersection morphisms})
 $f:X \rightarrow S$
 provided $\Lambda$ is a torsion ring with exponent invertible
 on $S$ or $\Lambda=R_\ell, Q_\ell$
 where $R_\ell$ is a complete discrete valuation ring
 over $\ZZ_\ell$, $Q_\ell=\mathrm{Frac}(R_\ell)$ and
 we assume $\ell$ invertible on $S$
 (see respectively Examples \ref{ex:cov_Gysin_etale&motivic} and
  \ref{ex:Gysin_BM_hlg}).
 This was previously only known for regular closed immersions
 by \cite[chap. 7]{Gabber} or for flat proper morphisms
 by \cite[XVIII, 2.9]{SGA4}.
\item \emph{Higher Chow groups} are \emph{contravariant 
 with respect to global complete intersection morphisms}
 $f:X \rightarrow S$ provided the residue fields of $S$ have
 all the same characteristic exponent, say $p$, and we invert
 $p$ (see Example \ref{ex:Gysin_BM_hlg}).
  Besides the trivial case of flat morphisms,
 only the case where $X$ and $S$ are smooth was known.
\item \emph{Integral motivic cohomology} in the sense of Spitzweck
 (cf. \cite{Spi}) --- which implies the case of rational
 motivic cohomology as defined in \cite{CD3} ---
 is \emph{covariant with respect to proper global complete
 intersection morphisms}
 (see Example \ref{ex:cov_Gysin_etale&motivic}). The case of 
 projective morphisms between regular schemes (resp. any scheme)
 was obtained in \cite{Deg10} (resp. \cite{Nav}).
\item \emph{Betti homology} of complex schemes
 and its analogue \emph{\'etale homology} with coefficients in a ring
 $\Lambda$
 as above (not to be mistaken with Borel-Moore \'etale homology)
 are \emph{contravariant with respect to proper global complete 
 intersection morphisms}. This example uses a construction due 
 to A. Khan (see Paragraph \ref{num:Gysin_support}). 
\end{itemize}
The general constructions of these Gysin type morphisms
 are given in Definition \ref{df:Gysin} and
 Paragraph \ref{num:Gysin_support}.
Note that more than a mere existence theorem we also obtain,
 as a consequence of the properties of fundamental classes
 stated in the preceding theorem, all of the expected properties
 of these Gysin morphisms (see Section \ref{sec:Gysin}).
 This also includes Grothendieck-Riemann-Roch formulas for
 Gysin morphisms (see in particular Proposition \ref{prop:GRR}).
 May be it is worth to formulate in this introduction
 the following new  formula,
 analogue to Verdier's Riemann-Roch formula for homology.
\begin{thmi}[See Example \ref{eq:GRR_gysin}]
Let $k$ be a field.

Let $f:Y \rightarrow X$ be a global complete intersection
 morphisms of $k$-schemes of finite type.
 Then we get the following commutative diagram:
$$
\xymatrix@=24pt@C=54pt{
G_n(X)\ar^{f^*}[r]\ar_{\ch_X}[d]
 & G_n(Y)\ar^{\ch_Y}[d] \\
\bigoplus_{i \in \ZZ} CH_i(X,n)_\QQ\ar^{\td(\tau_f).f^*}[r]
 & \bigoplus_{i \in \ZZ} CH_i(Y,n)_\QQ
}
$$
where $\ch$ is a Chern character isomorphism,
 $\td$ is the Todd class in rational motivic cohomology
 (which is acting on higher Chow groups),
 the upper (resp. lower) map $f^*$ is the Gysin morphism associated
 with $f$ on Thomason's $G$-theory, or equivalently
 Quillen's K'-theory (resp. higher Chow groups).
\end{thmi}
Note this theorem makes use of a recent result of Jin (cf. \cite{Jin2})
 about the representability of $G$-theory.

\bigskip

\noindent\textbf{Comparison with previous works in motivic homotopy theory.}

Orientation theory in motivic homotopy theory has been developed
 from two different points of view.

In the first point of view, one works with suitable axioms on certain
 functors, contravariant (cohomology theories)
 or covariant (homology theories). Among the axioms are the homotopy
 invariance property and a suitable orientation property (analogue to
 the one considered here). Then one deduces properties and constructions
 from the axioms, among which the construction of Gysin morphisms.
 This is the approach of Panin (\cite{Panin1, PaninRR, Panin2})
 on the one hand and Morel and Levine (\cite{ML}) on the other hand.
 Note the axiomatic of Panin is cohomological while that of Morel and
 Levine is homological. Besides, Morel and Levine construct the universal
 oriented (Borel-Moore) homology theory, the \emph{algebraic cobordism}.
 Both axiomatic are considered for schemes over a given base
 field.\footnote{More recently, this kind of axiomatic as well as the
 universal property of algebraic cobordism, has been extended to the framework
 of derived algebraic geometry by Lowrey and Sch\"urg in \cite{LS1}.}

In the second point of view, one studies oriented ring spectra
 in the motivic stable homotopy category,
 following the classical approach of Adams in algebraic topology.
 This was suggested by Voevodsky, and first worked out by Vezzosi (\cite{Vez})
 and independently Borghesi (\cite{Bor}). The construction
 of Gysin morphisms in this context was done in \cite{Deg8}
 (see also \cite[8.3, 8.4]{Deg1}). 
 The construction of \emph{loc. cit.} works other an arbitrary base
 and is internal: for example, it gives Gysin morphisms on the level
 of Voevodsky motives and allows one to get duality for motives of smooth
 projective schemes over an arbitrary base (\cite[5.23]{Deg8}),
 a result that previously uses resolution of singularity
 (see \cite[chap. 5, 4.3.2]{FSV}). It was later generalized
 in the works of the author (\cite{Deg12}) and then 
 in the work of Navarro (\cite{Nav}).

The two approaches are closely related. On the one hand,
 the results of Panin, and Levine-Morel, can be applied to the
 corresponding functor associated with a ring spectrum, giving back
 Gysin morphisms obtained from the second point of view. They are
 therefore more general, though most of the known homotopy invariant
 cohomology (homology) theories in algebraic geometry are known
 to be representable in the motivic stable homotopy category.
 A notable exception is the algebraic cobordism of Lowrey and Sch\"urg,
 a version of Levine-Morel's algebraic cobordism
 defined even over a field of positive characteristic. In this later case,
 the corresponding cohomology theory is not known to be representable.
 On the other hand, the methods used for oriented ring spectra 
 can be adapted to the (cohomological) functorial point of view:
 see \cite[\textsection 6]{Deg12}. Besides the methods are stronger
 as they produce Gysin morphisms internally, for example on motives
 or on modules over an oriented ring spectrum.

Compared to these previous works, the contribution of this work is two-fold.
 Firstly, we associate to any ring spectrum a canonical bivariant theory,
 called here the Borel-Moore homology --- as well as variants with compact support.
 Secondly, when the ring spectrum is oriented, we associate a system of 
 fundamental classes in the corresponding bivariant theory
 to a suitable class of local complete intersection
 morphisms and we show this construction recovers all the previously
 known Gysin morphisms, both in cohomology and in Borel-Moore homology
 --- it also gives Gysin morphisms for the compactly supported versions.
 These new Gysin maps
 are defined for more theories and/or for a larger class
 of morphisms than before.\footnote{In Panin's work,
 it was defined in cohomology for projective morphisms between smooth schemes
 over a field. In Levine's work, it was defined in Borel-Moore homology
 for smooth morphisms between quasi-projective schemes over a field.
 In Navarro's work, it was defined in cohomology for projective local
 complete intersection morphisms.}

Note finally that the consideration of bivariant theories in motivic homotopy theory,
 from the functorial point of view, was also considered by Yokura in \cite{Yok}.
 We note here that
 the bivariant theories associated with an oriented ring spectrum
 defined in this work
 (Definition \ref{df:bivariant_no_prod}) do satisfies the axioms
 of \cite{Yok}. Several definitions of a bivariant algebraic cobordism theory,
 corresponding to Levine and Morel algebraic cobordism, have been introduced
 in the literature (see \cite{GK, Sar, Yok2}). A comparison of our definition
 in the case of the ring spectrum $\MGL$ (see Example \ref{ex:biv_theories})
 with the definitions of these authors,
 when restricting to fields of characteristic $0$
 and to the graded parts $(2n,n)$, is an interesting problem
 (based on the representability of algebraic cobordism in characteristic $0$
  proved by Levine in \cite{Lev}).
	
\bigskip

\noindent\textbf{Further applications and future works.} 
Almost by definition of the bivariant theory $\biv \E {**}$,
 there is a categorical incarnation of the fundamental
 class associated with a global complete intersection morphism
 $f:X \rightarrow S$
 in the above Theorem \ref{thmi:BM_fundamental_classes},
 closer to the spirit of Grothendieck six functors formalism; it
 corresponds to a map:
$$
\tfdl_f:\E_X(d)[2d] \rightarrow f^!(\E_S),
$$
where $d$ is the relative dimension of $f$. In the end of this paper,
 we study conditions under which this map is an isomorphism.\footnote{Note
 this condition is stronger than the property of being strong, as defined
 in \cite{FMP}; see Definition \ref{df:strong_orientation}.}
 In brief, this will always be the case when $f$ is an $S$-morphism
 between smooth $S$-schemes, and in general, it is related to the
 absolute purity property. Let us also indicate that this property
 implies several duality statements, in the style of Bloch-Ogus duality
 giving us finally the link of our work with the last ending point
 of the historical tree on page \pageref{tree:history}. We refer the reader
 to Section \ref{sec:duality} for details.

\bigskip

Our main motivation for developing a general theory of fundamental classes,
 without base field, is the following theorem 
 which we state here using the definitions of \cite{BD1} and \cite{Ros}:
\begin{pthm}
Let $S$ be a base scheme with a dimension function $\delta$
 and consider further the following assumptions:
\begin{itemize}
\item $S$ is any noetherian finite dimensional scheme
 and $\Lambda$ is a $\QQ$-algebra;
\item $S$ is a scheme defined over a field of characteristic
 exponent $p$ and $\Lambda$ is a $\ZZ[1/p]$-algebra.
\end{itemize}
Then for any motive $M$ in the heart of the $\delta$-homotopy
 $t$-structure of $\DM(S,\Lambda)$, the functor $\hat H_0^\delta(M)$
 defined in \cite{BD1} admits a canonical structure, functorial in $M$,
 of a cycle module over $S$ in the sense of Rost \cite{Ros}.
\end{pthm}
Indeed, the Gysin morphisms (resp. fundamental classes)
 constructed here are the essential tool to obtain
 the corestriction and residue operations
 of the structure of a cycle module. This work is in progress,
 and is part of a general strategy to prove an original conjecture of
 Ayoub (see \cite{Deg17}). More generally, this theorem will help us
 to analyse the Leray-type spectral sequence derived from the
 $\delta$-homotopy $t$-structure of \cite{BD1}.

Let us finally mention that the techniques of this paper
 will be exploited in a future work (\cite{DJK}) whose aim is to
 define Gysin morphisms in motivic stable homotopy, without requiring
 the orientation used throughout the present paper. This result
 is motivated and supported by the fundamental work of Morel which
 analyses the structure of $\AA^1$-homotopy groups (\cite{MorLNM}).

\subsection*{Outline of the paper}

As a general guideline for the reader, let us mention that Section 1
 contains the main definitions and notations used in the paper
 as well as the examples.
 We advise the reader to use this section as a reference part.
 The main part of this work is Section 2 which contains
 our main results while Section 3 deals with applications.
 
Let us review the content in more detail.
 In section 1.1, we settle our framework by introducing
 absolute ring $\T$-spectra as explained in the above introduction,
 as well as morphisms between them.
 In the style of algebraic topology, we also consider
 modules over ring spectra.
 Then (section 1.2), we associate to an absolute ring $\T$-spectrum
 its canonical bivariant theory, called {\it Borel-Moore homology}
 as explained in the introduction and show various of its
 basic properties --- note that a variant of the theory is explained
 for modules over ring spectra. Finally (section 1.3),
 we define the other bivariant theories associated to absolute
 ring spectra, cohomology with compact support and homology.
 Examples are given throughout Section 1.

Our main theorem is developed in Section 2.
 In sections 2.1 and 2.2, we recall the basic theories
 of \emph{orientations} for bivariant theories, called here 
 \emph{fundamental classes},
 and that of \emph{orientation} for ring spectra, modelled on algebraic
 topology and giving the first characteristic classes
 (Chern class and Thom classes).
 In sections 2.3 and 2.4, we settled the particular cases of fundamental
 classes that will be used in our main theorem, respectively
 the case of smooth morphisms and that of regular closed
 immersions. Our main result (Theorem \ref{thmi:BM_fundamental_classes}
 above) is proved in section 2.5. Section 2.6 deals with finer uniqueness
 results when restricted to quasi-projective morphisms.

Then Sections 3 and 4 gives applications and properties of the
 fundamental classes obtained in Section 2.
 Sections 3.1 and 3.2 are concerned with the main properties of
 fundamental classes such as behaviour with respects to pullbacks
 and compatibility with morphisms of ring spectra.
 Sections 3.3 applies the theory to the construction of Gysin
 morphisms with an emphasis on concrete examples.
 Finally in Section 4, we treat the questions of purity
 (or equivalently \emph{absolute purity} following the classical
 terminology  of the \'etale formalism) and its relation with duality
 statements. Again, many examples where purity, and therefore duality,
 holds are given all along.

%% file: absolute.tex
\subsection{Definition of absolute spectra}

The following notion is a simple extension of \cite[1.2.1]{Deg12}.
\begin{df}\label{df:absolute_sp}
An \emph{absolute spectrum} over $\base$ is a pair
 $(\T,\E)$ where $\T$ is a triangulated motivic category over $\base$
 and $\E$ is a cartesian section of the fibered category $\T$
 \emph{i.e.} the data:
\begin{itemize}
\item for any scheme $S$ in $\base$, of an object $\E_S$ of $\T(S)$,
\item for any morphism $f:T \rightarrow S$, of an isomorphism
 $\tau_f:f^*(\E_S) \rightarrow \E_T$, called \emph{base change
 isomorphism},
\end{itemize}
and we require that base change isomorphisms are compatible with
 composition in $f$ as usual. We will also say that $\E$ is an \emph{absolute
 $\T$-spectrum} and sometimes just absolute spectrum when this
 does not lead to confusion.

A \emph{ring structure} on $(\T,\E)$ will be 
on each $\E_S$ such that the base change isomorphisms
 are isomorphisms of monoids. We will say $\E$ is an
 \emph{absolute ring $\T$-spectrum}.

Given a triangulated motivic category $\T$,
 the pair $(\T,\un)$ where $\un$ is the cartesian section
 corresponding to the unit $\un_S$ for all schemes $S$
 will be called the \emph{canonical absolute $\T$-spectrum}.
 It obviously admits a ring structure. We will sometime denote it by $\HH_\T$.
\end{df}
When $\T$ is a $\Lambda$-linear category, we will also say
 \emph{absolute $\Lambda$-spectrum}.

\begin{ex}\label{ex:spectra}
\begin{enumerate}
\item Let $\base$ be the category of $\ZZ[P^{-1}]$-schemes
 for a set of primes $P$. Assume $\Lambda=\ZZ/n\ZZ$ for $n$ a product
 of primes in $P$,
 or $\Lambda=\ZZ_\ell, \QQ_\ell$ and $P=\{\ell\}$.
 Then we get the \emph{\'etale absolute $\Lambda$-spectrum} as the
 canonical absolute spectrum associated with the motivic triangulated
 category $S \mapsto D^b_c(S_\et,\Lambda)$ ---
 the bounded derived category of $\Lambda$-sheaves on the small \'etale
 site of $S$ (cf. \cite{SGA4}, \cite{Eke}).
\end{enumerate}
Let now $\base$ be the category of all schemes.
\begin{enumerate}\setcounter{enumi}{1}
\item Let $\Lambda$ be a $\QQ$-algebra. The \emph{motivic
 (Eilenberg-MacLane) absolute $\Lambda$-spectrum} $\HH\Lambda$
 can be defined as
 the canonical absolute spectrum associated with one
 of the equivalent version of the triangulated
 category of rational motives (see \cite{CD3}).
 Following a notation of Riou, when $\Lambda=\QQ$,
 this ring spectrum is sometimes denoted by $\HB$
 and called the Beilinson motivic ring spectrum.\footnote{According
 to \cite[14.2.14]{CD3},
 it represents, over regular schemes, rational motivic cohomology
 as first defined by Beilinson in terms of rational Quillen K-theory.}
\item Let $\Lambda$ be any ring. The \emph{\'etale motivic
 absolute $\Lambda$-spectrum} $\HH_\et\Lambda$ can be defined as the
 canonical absolute
 spectrum associated with the triangulated motivic category
 of $\h$-motives of Voevodsky (see \cite{CD5}). When $2$ is invertible
 in $\Lambda$, one can also use the \'etale-local $\AA^1$-derived category
 as defined in \cite{Ayo2}.

When $\Lambda=\ZZ_\ell$ (resp. $\QQ_\ell$) the ring of $\ell$-adic
  integers
 (resp. rational integers), we will adopt the usual abuse of notations
 and denote by $\HH_\et\Lambda$ the canonical absolute
 spectrum associated with the homotopy $\ell$-completion (resp. rational part
 of the homotopy $l$-completion) of the triangulated motivic category
 of $\h$-motives of Voevodsky (see \cite[7.2.1]{CD5}).
\item Let $\SH$ be Morel-Voevodsky's stable homotopy category.
 Then an absolute spectrum $\E$ in the sense of \cite{Deg12}
 is an absolute spectrum of the form $(\SH,\E)$.
 This includes in particular the following absolute spectra:
\begin{itemize}
\item \emph{algebraic cobordism} $\MGL$,
\item \emph{Weibel K-theory} $\KGL$,
\item when $\Lambda$ is a localization of $\ZZ$,
 \emph{motivic cohomology $\HH\Lambda$ with $\Lambda$-coefficients}
 as defined by Spitzweck (cf. \cite{Spi}).
\end{itemize}
We refer the reader to \cite[Ex. 1.2.3 (4,5)]{Deg12} for more details.
\end {enumerate}

When $\base$ is the category of $S$-schemes for a given scheme $S$,
 and $\T=\SH$, any spectrum (resp. ring spectrum) $\E_S$ of $\SH(S)$ gives rise to
 an absolute spectrum (resp. ring spectrum) $(\T,\E)$ by putting
 for any $f:T \rightarrow S$, $\E_T=f^*(\E_S)$.
 This gives the following classical examples of absolute ring spectra
 over $S$-schemes:
\begin{enumerate}\setcounter{enumi}{4}
\item $S=\spec k$ for a field $k$,
 any mixed Weil cohomology $E$ over $k$, in the sense of
 \cite{CD2};
\item $S=\spec K$ for a $p$-adic field $K$, the syntomic cohomology
 with coefficients in $K$ (cf. \cite{DN1});
\item $S=\spec V$ for a complete discrete valuation ring $V$,
 the rigid syntomic cohomology, with coefficients in $K$ (cf. \cite{DM1}).
\end{enumerate}
\end{ex}

\begin{rem}
It is not absolutely clear from the literature that the bounded
 derived category of mixed Hodge modules as defined by Saito
 satisfies the complete set of axioms of a motivic triangulated
 category.
However, the faithful reader can then consider the associated
 absolute ring spectrum, an object that should be called the 
 \emph{Deligne absolute ring spectrum}.\footnote{A short for: the
 \emph{absolute ring spectrum representing Deligne cohomology}.}
\end{rem}

Recall that an \emph{adjunction} 
 of  motivic triangulated categories 
 (or equivalently a \emph{premotivic adjunction})
 is a functor
$$
\varphi^*:\T \rightarrow \T'
$$
of triangulated categories such that $\varphi^*_S:\T(S) \rightarrow \T'(S)$
 is monoidal and 
 commutes with pullback functors (see \cite[1.4.2]{CD3}).
 In particular, given a cartesian section $\E$ of $\T$,
 we get a cartesian section $\F:=\varphi^*(\E)$ of $\T'$
 by putting $\F_S=\varphi_S^*(\E_S)$.
\begin{df}\label{df:morph_abs_sp}
A morphism of absolute spectra
$$
(\varphi,\phi):(\T,\E) \rightarrow (\T',\F)
$$
is a premotivic adjunction $\varphi^*:\T \rightarrow \T'$
 together with a morphism of cartesian sections
 $\phi:\varphi^*(\E) \rightarrow \F$,
 \emph{i.e.} a family of morphisms $\phi_S:\varphi^*(\E_S) \rightarrow \F_S$
 compatible with the base change isomorphisms.
 
A \emph{morphism of absolute ring spectra} is a pair $(\varphi,\phi)$ as
 above such that for any scheme $S$ in $\base$, $\phi_S$ is
 a morphism of commutative monoids. In that case,
 we also say that $\F$ is an \emph{$\E$-algebra}.

Obviously, these morphisms can be composed. Moreover we will say that
 $(\varphi,\phi)$ is an \emph{isomorphism} if $\varphi^*$ is fully faithful
 and for all schemes $S$ in $\base$, $\phi_S$ is an isomorphism.
 Finally we will say that the isomorphism $(\varphi,\phi)$ is strong
 if the functor $\varphi^*$ commutes with $f^!$ for any s-morphism
 $f$ between excellent schemes.
\end{df}

\begin{rem}
Note that it usually happens that a motivic triangulated category $\T$
 admits a distinguished motivic triangulated subcategory $\T_c$
 of constructible objects (see \cite[Def. 4.2.1]{CD3}).
 According to our definitions, the absolute spectrum associated
 with $\T$ is then canonically isomorphic to that associated with $\T_c$
 because for any scheme $S$, $\un_S$ is constructible, 
 thus belongs to $\T_c$. It frequently happens that the corresponding
 isomorphism is strong (see for example \cite[4.2.28]{CD3},
 \cite[6.2.14]{CD4}, \cite[6.4]{CD5}).
\end{rem}

\begin{ex}\label{ex:morph_sp}
\begin{enumerate}
\item It follows from the previous remark and the rigidity theorems
 of \cite{Ayo3, CD5} that when $\base$ is the category of $\ZZ[P^{-1}]$-schemes
 for a set of primes $P$  and $\Lambda=\ZZ/n\ZZ$ where $n$ is a product of
 primes in $P$,
 the \'etale absolute $\Lambda$-spectrum and the \'etale motivic 
 absolute $\Lambda$-spectrum are isomorphic.
\item Any adjunction of motivic triangulated categories
 $\varphi^*:\T \rightarrow \T'$ trivially induces a morphism of
 absolute ring spectra
$$
(\varphi,Id):\HH_\T \rightarrow \HH_{\T'}
$$
because by definition, $\varphi^*_S$ is monoidal.

This immediately gives several examples of morphisms of
 absolute spectra:
\begin{itemize}
\item when $\base$ is the category of all schemes,
 for a prime $\ell$, we get 
$$
\HH \QQ \rightarrow \HH_\et \QQ_\ell
$$
associated to the \'etale realization functor
 $\rho_\ell:\DM_\QQ \rightarrow \DM_\h(-,\QQ_\ell)$ defined in \cite[7.2.24]{CD4}.
\item when $\base$ is the category of $k$-schemes for a field $k$
 of characteristic $p$, we get
 $$
\HH \ZZ[1/p] \rightarrow \HH_\et \ZZ_\ell
$$
associated to the integral \'etale realization functors
$$
\rho_\ell:\DM_\cdh(-,\ZZ[1/p]) \rightarrow D(-_\et,\ZZ_\ell)
$$
defined in \cite[Rem. 9.6]{CD5}.
\end{itemize}
\item Following Riou, we have the Chern character:
$$
\ch:\KGL_\QQ \xrightarrow \simeq \bigoplus_{i \in \ZZ} \HH \QQ(i)[2i], 
$$
which is an isomorphism of absolute ring $\SH$-spectra
 (cf. \cite[5.3.3]{Deg12}).
\end{enumerate}
\end {ex}

\begin{rem}\label{rem:SH_initial&ring_sp}
The triangulated motivic category $\SH$ is almost initial.
 In fact, as soon as a triangulated motivic category is
 the homotopy category of a combinatorial model stable category,
 there exists an essentially unique premotivic adjunction\footnote{The
 same result holds in the framework of $\infty$-category 
 according to \cite[1.2]{Rob}.}:
$$
\nu^*:\SH \leftrightarrows \T:\nu_*
$$
In this case, for any scheme $S$, we get a ring spectrum\footnote{Indeed,
 recall that $\nu_*$, as the right adjoint of a monoidal functor,
 is weakly monoidal.}
 $\HH_S^\T:=\nu_*(\un_S)$ in $\SH(S)$
 which represents the cohomology associated with the canonical absolute
 $\T$-spectrum.
 The collection $\HH_S^\T$ indexed by schemes $S$ in $\base$
 define a section of the fibered category $\SH$ as
 for any morphism $f:T \rightarrow S$, we have natural maps:
$$
\tau_f^\T:f^*(\HH_S^\T)=f^*\nu_*(\un_S)
 \rightarrow \nu_*(f^*\un_S)\simeq \nu_*(\un_S)=\HH_S^\T
$$
compatible with the monoid structure.
 In general, these maps are not isomorphisms \emph{i.e.}
 $\HH^\T$ does not form an absolute $\SH$-spectrum.

Note however that $\tau_f^\T$ is an isomorphism when the functor $\varphi_*$
 commutes with pullback functors $f^*$.
 Most of the examples given in \ref{ex:spectra} 
 will go into this case except for one example, that
 of the (motivic)  \'etale $\Lambda$-spectrum for
 $\Lambda=\ZZ_\ell,\QQ_\ell$. In fact, in this case
 we do not know whether the right adjoint of the
 $\ell$-adic realization functor commutes with
 $f^*$. This latter case justifies the generality chosen in this paper.
\end{rem}

\begin{df}\label{df:modules}
Let $(\E,\T)$ be an absolute ring spectrum.

A \emph{module over $(\E,\T)$} will be an absolute spectrum
 $(\F,\T)$  and for any scheme $S$ in $\base$ an associative and unital
 action
$$
\phi_S:\E_S \otimes \F_S \rightarrow \F_S
$$
which is compatible with the structural base change isomorphisms.

Given a premotivic adjunction $\varphi^*:\T \rightarrow \T'$,
 the cartesian section $\varphi^*(\E)$ is again an absolute ring $\T'$-spectrum.
 A $\varphi$-module over $(\E,\T)$ will be a module over
 $(\varphi^*(\E),\T)$.
\end{df}
In both cases, when the context is clear,
 we will simply say $\F$ is an $\E$-module.
 
\begin{rem}
Obvious examples of modules over an absolute ring spectrum $\E$
 are given by $\E$-algebras, defined in \ref{df:absolute_sp}.
 We will see many examples of that kind in Paragraph \ref{num:univ_MGL}.
\end{rem}

\subsection{Associated bivariant theory} \label{sec:ass_biv}

\begin{num}\label{num:prebivariant}
We now give the the basic definitions of bivariant theories
 suitable to our needs taken with some small variants from \cite{FMP}.

Let us fix $\mathcal F\base$ the subcategory of the category
 of arrows in $\base$ whose objects are the s-morphisms and
 maps are cartesian squares.
 Usually, an object $f:X \rightarrow S$ of $\mathcal F\base$ 
 will be denoted by $X/S$ when no confusion can arise.
 Similarly, a morphism $\Delta$:
$$
\xymatrix@=10pt{
Y\ar[r]\ar[d]\ar@{}|{\stackrel \Delta \Rightarrow}[rd] & X\ar[d] \\
T\ar_f[r] & S
}
$$
will be indicated by the map $f:T \rightarrow S$.
 Let $\A$ be the category of bigraded abelian groups with morphisms the
 homogeneous ones of degree $(0,0)$.

A \emph{bivariant theory without products}
\footnote{Products will be introduced in the second part of this introduction.
 Apart from the absence of products, this notion corresponds
 to a bivariant theory as in \cite{FMP} where independent squares are
 cartesian squares and confined maps are proper morphisms.}
 will be a contravariant functor
$$
\E:\mathcal F\base \rightarrow \A, X/S \mapsto \E_{**}(X/S)
$$
which is also a covariant functor in $X$
 with respect to proper morphisms of $S$-schemes,
 and satisfies the following projection formula:
 for any cartesian squares
$$
\xymatrix@R=10pt@C=14pt{
Y'\ar[r]\ar_{p'}[d] & Y\ar^p[d] \\
X'\ar[r]\ar[d] & X\ar[d] \\
S'\ar^f[r] & S,
}
$$ 
and any element $x \in \E_{**}(X/S)$ one has: $f^*p_*(x)=p'_*f^*(x)$,
 as soon as all the maps exist.
 The structural map $f^*:\E_{**}(X/S)\rightarrow \E_{**}(X'/S')$ will be
 referred to as the \emph{base change map}
 associated with $f$.

Given any absolute spectrum $(\T,\E)$, any s-morphism $p:X \rightarrow S$
 and any pair $(n,m) \in \ZZ^2$, we put:
$$
\biv \E {n,m}(X/S):=\Hom_{\T(X)}(\un_X(m)[n],p^!(\E_S)).
$$
The functoriality of $p^!$ then allows us to defines
 a contravariant functor from $\mathcal F\base$
 to bigraded abelian groups and the base change map is given
 by the pullback functor $f^*$ for a given morphism $f:T \rightarrow S$.
 The covariance with respect to a proper $S$-morphism $f:Y \rightarrow X$
 can be defined using the map:
$$
\E_X \xrightarrow{ad(f^*,f_*)} f_*f^*(\E_X) \simeq f_*(\E_Y) \simeq f_!(\E_Y).
$$
where $ad(f^*,f_*)$ is the unit of the relevant adjunction,
 the first isomorphism uses the structural isomorphism
 of the absolute spectrum $\E$ and the last isomorphism follows
 from the fact $f$ is proper (see \cite[2.4.50(2)]{CD2}).

It is now a formal exercise to check the axioms of a bivariant theory without
 products are fulfilled for the bifunctor $\E_{**}^{BM}$.
\end{num}
\begin{df}\label{df:bivariant_no_prod}
Under the assumptions above, the bifunctor $\biv \E {**}$ will be
 called the \emph{Borel-Moore homology}\footnote{This terminology
 extends the classical terminology in motivic homotopy theory, usually
 applied in the case where $S$ is the spectrum of a field. Note that we will
 see other bivariant theories associated with $\E$ so we have chosen to
 use that terminology following the tradition of our field.}
 associated with the absolute spectrum $\E$.

When $\E=\HH^\T$ is the canonical absolute ring spectrum
 associated with $\T$, we will denote the corresponding
 Borel-Moore homology by $\biv H {n,m}(X/S,\T)$.
\end{df}
Recall that one associates to a bivariant theory a cohomology theory; in our case, we have:
$$
\E^{n,m}(X)=\biv \E {-n,-m}\left(X \xrightarrow{1_X} X\right)
 =\Hom_{\T(X)}\big(\un_X,\E_X(m)[n]\big)
$$
which is the usual formula for the cohomology represented
 by the spectrum $\E_X$. 

\begin{num}\label{num:duality_etale}
According to the six functors
 formalism, for any \'etale s-morphism $f:X \rightarrow S$, we
 have a canonical isomorphism of functors $\pur_f:f^! \simeq f^*$
 (cf. for example \cite[2.4.50(3)]{CD3}).
 Therefore, we also get a canonical isomorphism:
$$
\biv \E {n,m}(X/S) \simeq \E^{-n,-m}(X).
$$
As the isomorphism $\pur_f$ is compatible with composition
 (see Proposition \ref{thm:ayoub}),
 we obtain that $\biv \E {n,m}(X/S)$ is functorial
 in $X$ with respect to \'etale morphisms.

Recall now a classical terminology in motivic homotopy theory.
 We say a cartesian square:
\begin{equation}\label{eq:distinguished}
\xymatrix@=10pt{
Y'\ar^k[r]\ar_v[d] & X'\ar^u[d] \\
Y\ar^i[r] & X,
}
\end{equation}
is Nisnevich (resp. cdh) distinguished if
 $i$ is an open (resp. closed) immersion, $u$ is an \'etale
 (resp. proper) morphism and the induced map
 $(X'-Y') \rightarrow (X-Y)$ on the underlying reduced
 subschemes is an isomorphism.

Now, the following properties are direct consequences
 of the Grothendieck six functors formalism.
\end{num}
\begin{prop}\label{prop:basic_bivariant}
Let $\E$ be an absolute spectrum. The following properties hold:
\begin{enumerate}
\item \emph{Homotopy invariance}.-- For any vector bundle
 $p:E \rightarrow S$, and any s-scheme $X/S$, the base change map:
$$
p^*:\biv \E {n,m}(X/S) \rightarrow \biv \E {n,m}(X \times_S E/E)
$$
is an isomorphism.
\item \emph{\'Etale invariance}.-- Given any $s$-schemes
 $X/T/S$ such that $u:T \rightarrow S$ is \'etale,
 there exists a canonical isomorphism:
$$
u^*:\biv \E {**}(X/S) \xrightarrow{\sim} \biv \E {**}(X/T)
$$
which is natural with respect to base change in $S$
 and the covariance in $X/T$ for proper morphisms.
\item \emph{Localization}.--  For and s-scheme $X/S$ and
 any closed immersion $i:Z \rightarrow X$
 with complementary open immersion $j:U \rightarrow X$,
 there exists a canonical \emph{localization long exact sequence} of
 the form:
\begin{align*}
\biv \E {n,m}(Z/S) & \xrightarrow{i_*}
 \biv \E {n,m}(X/S) \xrightarrow{j^*}
 \biv \E {n,m}(U/S) \xrightarrow{\partial_i}
 \biv \E {n-1,m}(Z/S)
\end{align*}
which is natural with respect to the contravariance in $S$,
 the contravariance in $X/S$ for \'etale morphisms
 and the covariance in $X/S$ for proper morphisms.
\item \emph{Descent property}.-- for any square \eqref{eq:distinguished}
 of s-schemes over $S$
 which is either Nisnevich or cdh distinguished, there exists
 a canonical long exact sequence:
$$
\biv \E {n,m}(X/S)
 \xrightarrow{i^*+u^*} \biv \E {n,m}(Y/S) \oplus \biv \E {n,m}(X'/S)
 \xrightarrow{v^*-k^*} \biv \E {n,m}(Y'/S)
 \rightarrow \biv \E {n-1,m}(X/S)
$$
natural with respect to the contravariance in $S$,
 the contravariance in $X/S$ for \'etale morphisms
 and the covariance in $X/S$ for proper morphisms.
\end{enumerate}
\end{prop}
The proof is again an exercise using  the properties of the motivic triangulated
 category $\T$. More precisely: (1) follows from the homotopy property,
 (2) from the isomorphism $\pur_f:f^! \simeq f^*$ for an \'etale
 s-morphism $f:T \rightarrow S$, 
 (3) from the localization property (and for the functoriality,
 from the uniqueness of the boundary operator at the triangulated level,
 see \cite[2.3.3]{CD3}), (4) from the Nisnevich and cdh descent properties
 of $\T$ (see \cite[3.3.4 and 3.3.10]{CD3}).

\begin{rem}\label{rem:biv&support}
An important remark for this work  is the fact that the Borel-Moore
 homology associated with an absolute spectrum $\E$, restricted
 to the subcategory of $\mathcal F\base$ whose objects are
 {\it closed immersions},
 coincides with the cohomology with support and
 coefficients in $\E$.
 Indeed, this can be seen from the localization property
 in the case of a closed immersion $i:Z \rightarrow S$.
 And in fact, using  the definition of cohomology
 with support introduced in \cite[1.2.5]{Deg12}, we get an equality:
$$
\biv \E {n,m}(i:Z \rightarrow S)=\E^{-n,-m}_Z(X).
$$
This explains why the properties used in \emph{op. cit.}
 are exactly the same than the ones of bivariant theories
 (a fact the author became aware after writing
 \emph{op. cit.}).
\end{rem}

\begin{num}\label{num:functoriality1}
The Borel-Moore homology associated with an absolute spectrum is functorial:
 given a morphism $(\varphi,\phi):(\T,\E) \rightarrow (\T',F)$
 of absolute spectra, and an s-morphism $p:X \rightarrow S$,
 we define $(\varphi,\phi)_*$ --- often simply denoted by $\phi_*$:
\begin{align*}
\biv \E {n,m}(X/S)=\Hom_{\T(X)}(f_!(\un_X)(m)[n],\E_S)
 \xrightarrow{\varphi^*}
 & \Hom_{\T'(X)}(\varphi^*f_!(\un_X)(m)[n],\varphi^*\E_S) \\
 \simeq & \Hom_{\T'(X)}(f_!(\un'_X)(m)[n],\varphi^*\E_S) \\
 \xrightarrow{\ \phi_*\ } &\Hom_{\T'(X)}(f_!(\un'_X)(m)[n],\F_S)
 =\biv \F {n,m}(X/S)
\end{align*}
where the isomorphism comes from the exchange isomorphism 
$$
\varphi^*f_! \xrightarrow \sim f_!\varphi^*
$$
associated with the premotivic adjunction $(\varphi^*,\varphi_*)$
 (cf. \cite[2.4.53]{CD3}).

It is not difficult (using the compatibility of the various 
 exchange transformations involved) to prove $(\varphi,\phi)_*$
 is compatible with the base change maps, the covariant functoriality
 in $X/S$ with respect to proper maps and the contravariant functoriality
 in $X/S$ with respect to \'etale maps. In a word,
 $(\varphi,\phi)_*$ is a natural transformation of bivariant theories
 without products.

Note moreover that $(\varphi,\phi)_*$ is compatible with composition of
 morphisms of absolute spectra and an isomorphism of absolute spectra 
 induces an isomorphism of bivariant theories.

Interesting examples will be given later, after the consideration
 of products.
\end{num}

\begin{rem}\label{rem:SH_initial&ring_sp_bis}
Consider the setting of Remark \ref{rem:SH_initial&ring_sp}.
 Assume in addition that the section $\HH^\T$ is cartesian.
 Then one gets a morphism of absolute ring spectra (cf. Definition
 \ref{df:morph_abs_sp}) 
 $(\HH^\T,\SH) \rightarrow (\un,\T)$ as for any scheme $S$,
 we have morphisms of monoids:
$$
\can^*(\HH^\T_S)=\can^*(\can_*(\un_S)) \xrightarrow{ad'(\can^*,\can_*)}
 \un_S.
$$
Note this morphism is not an isomorphism in the sense of
 \ref{df:morph_abs_sp}.
 However, the map induced on Borel-Moore homologies of
 an s-morphism $f:X \rightarrow S$
$$
\biv{(\HH^\T)}{**}(X/S) \rightarrow \biv H {**}(X/S,\T)
$$
is an isomorphism as the functor $\varphi^*$ commutes with
 direct images $f_!$ (see \cite[2.4.53]{CD3}).
\end{rem}

\begin{num}\label{num:product}
Recall now that a \emph{bivariant theory} $\E$ in the sense
 of Fulton and MacPherson\footnote{More precisely,
 when independent squares are cartesian squares, confined map are proper
 morphisms and the category of values is that of bigraded abelian groups;}
 is a bivariant
 theory without products as introduced in \ref{num:prebivariant}
 such that for any s-schemes $Y \rightarrow X \rightarrow S$,
 there is given a product:
$$
\E_{n,m}(Y/X)\otimes \E_{s,t}(X/S) \rightarrow \E_{n+s,m+t}(Y/S),
 (y,x) \mapsto y.x
$$
satisfying the following axioms:
\begin{itemize}
\item \emph{Associativity}.-- given s-morphisms $Z/Y/X/S$,
 for any triple $(z,y,x)$, we have: $$(z.y).x=z.(y.x).$$
\item \emph{Compatibility with pullbacks}.-- Given s-morphisms $Y/X/S$ and a morphism $f:S' \rightarrow S$
 inducing $g:X' \rightarrow X$ after pullback along $X/S$,
 for any pair $(y,x)$, we have: $f^*(y.x)=g^*(y).f^*(x).$
\item \emph{Compatibility with pushforwards}.-- Given s-morphisms  $Z \xrightarrow f Y \rightarrow X \rightarrow S$
 such that $f$ is proper, for any pair $(z,y)$, one has:
 $f_*(z.y)=f_*(z).y.$
\item \emph{Projection formula (second)}.-- Given a cartesian square of
 s-schemes over $S$:
$$
\xymatrix@=10pt{
Y'\ar^g[r]\ar[d] & Y\ar[d] & \\
X'\ar^f[r] & X\ar[r] & S,
}
$$
such that $f$ is proper, for any pair $(y,x)$, one has: 
 $g_*(f^*(y).x')=y.f_*(x')$. 
\end{itemize}
Consider now an absolute ring spectrum $(\T,\E)$. Then one can define
 a product of the above form on the associated Borel-Moore homology
 defined in \ref{df:bivariant_no_prod}. Consider indeed
 s-morphisms $Y \xrightarrow f X \xrightarrow p S$ and classes
$$
y:\un_Y(m)[n] \rightarrow f^!(\E_X),
 x:\un_X(s)[t] \rightarrow p^!(\E_S).
$$
Let us first recall that one gets a canonical pairing\footnote{This is
 classical: see also \cite[IV, 1.2.3]{SGA45}.}
\begin{equation}\label{eq:pairing!*}
Ex^{!*}_\otimes:p^*(M) \otimes p^!(N) \rightarrow p^!(M \otimes N)
\end{equation}
obtained by adjunction from the following map:
$$
p_!(p^*(M) \otimes p^!(N))
 \xrightarrow{\sim} M \otimes p_!p^!(N)
 \xrightarrow{1 \otimes ad'(p_!,p^!)} M \otimes N.
$$
Then one associates to $x$:
$$
\tilde x:\E_X(s)[t] \xrightarrow{1_{\E_X} \otimes x} \E_X \otimes p^!(\E_S)
 \simeq p^*(\E_S) \otimes p^!(\E_S)
 \xrightarrow{Ex^{*!}_\otimes} p^!(\E_S \otimes \E_S)
 \xrightarrow{p^!(\mu)} p^!(\E_S)
$$
where the map $\mu$ is the multiplication map of the ring spectrum $\E_S$.
Then, one defines the product as the following composite map:
$$
y.x:\un_Y(m+s)[n+t] \xrightarrow{y(s)[t]} f^!(\E_X)(s)(t]
 \xrightarrow{f^!(\tilde x)} f^!p^!(\E_S)=(pf)^!(\E_S).
$$
This is now a lengthy exercise to prove that the axioms stated
 previously are satisfied for the product just defined and
 the bifunctor $\E_{**}^{BM}$. We refer the reader to
 \cite[proof of 1.2.10]{Deg12} for details\footnote{the corresponding
 axioms are proved in \emph{loc. cit.} when confined maps are closed immersions (see Remark \ref{rem:biv&support}).
 In fact the proof does not change
 when confined maps are only assumed to be proper rather than being
 closed immersions.}.
\end{num}

\begin{rem}\label{rem:cup_biv&act_coh&supp_prod}
\begin{enumerate}
\item Products on bivariant theories obviously
 induce products on the associated cohomology theory.
 In the case of the bivariant theory of the above definition
 these induced products are nothing else than the usual cup-products.
\item If one considers the case $Y=X$, the product of the preceding
 paragraph gives an action of the cohomology of $X$ on the Borel-Moore
 homology of $X/S$.
\item We have seen in Remark \ref{rem:biv&support} that
 cohomology with support is a particular instance of Borel-Moore homology.
 In fact,
 the product introduced in \cite[1.2.8]{Deg12} for cohomology with supports
 coincides with the one defined here restricted to Borel-Moore homology
 of closed immersions.
 This is obvious as the formulas in each case are exactly the same.
\end{enumerate}
\end{rem}

\begin{ex}\label{ex:biv_theories}
From the examples of \ref{ex:spectra}, we
 get respectively the Borel-Moore motivic $\Lambda$-homology,
 the Borel-Moore \'etale motivic $\Lambda$-homology,
 the \emph{Borel-Moore homotopy invariant K-theory} and
 the \emph{Borel-Moore algebraic cobordism}:
$$
\HH^{BM}_{**}(X/S,\Lambda),
 \HH_{**}^{BM,\et}(X/S,\Lambda),
 \KGL^{BM}_{**}(X/S),
 \MGL^{BM}_{**}(X/S).
$$
\begin{enumerate}
\item The Borel-Moore motivic homology $\HH^{BM}_{**}(X/S,\Lambda)$
 could also be called \emph{bivariant higher Chow groups}. In fact, it
 follows from \cite[8.13]{CD5} that for any s-scheme $X/k$
 where $k$ is a field of characteristic exponent $p$, one
 has a canonical isomorphism:
$$
\varphi:\HH^{BM}_{n,m}(X/k,\ZZ[1/p])
 \xrightarrow \sim CH_m(X,n-2m)[1/p]
$$
where the right hand side is the relevant Bloch's higher Chow
 group.\footnote{In fact, the assumption that $X$ is equidimensional
 in \emph{loc. cit.} can be avoided as follows: coming back to the
 proof of Voevodsky in \cite[chap. 5, 4.2.9]{FSV}, one sees that in
 any case the map $\varphi$ exists --- it is induced by and inclusion
 of groups of cycles. Therefore, to prove it is an isomorphism,
 we reduce to the equidimensional case by noetherian induction,
 as the map $\varphi$ is compatible with the localization sequence
 (see the proof of \emph{loc. cit.}).}
\item It follows from \ref{ex:morph_sp}(1) that when $\base$ is the category
 of $\ZZ[P^{-1}]$-schemes and $\Lambda=\ZZ/n\ZZ$ with $n$ a product of
 primes in $P$, the Borel-Moore \'etale motivic $\Lambda$-homology
 $\HH_{**}^{\et}(X/S,\Lambda)$ coincides with the bivariant \'etale
 theory with $\Lambda$-coefficients as considered in \cite[7.4]{FMP}.
\item Recall that the absolute ring spectrum $\KGL$ satisfies
 Bott periodicity. In particular we get a canonical
 isomorphism:
$$
\biv \KGL {n,m}(X/S) \simeq \biv \KGL {n+2,m+1}(X/S).
$$
In particular, the double indexing is superfluous and
 we sometime consider bivariant K-theory as $\ZZ$-graded according
 to the formula:
$$
\biv \KGL {n,m}(X/S) \simeq \biv \KGL {n-2m}(X/S).
$$

In general, this bivariant K-theory does not coincide with
 the bivariant K-theory $\mathrm K_{\mathrm{alg}}$ of \cite[1.1]{FMP}. 
 Indeed the theory of \emph{loc. cit.} does not satisfies
 the homotopy property (Prop. \ref{prop:basic_bivariant})
 when considering non regular schemes (as algebraic K-theory).

Note however that according to \cite{Jin2},
 one gets a canonical isomorphism:
$$
\biv \KGL n(X/S) \simeq G_n(X)
$$
for a quasi-projective morphism $f:X \rightarrow S$ with $S$ regular,
 where $G_*$ is Thomason's G-theory, or equivalently
 Quillen's K'-theory as we work with noetherian schemes
 (see \cite[3.13]{TT}). The isomorphism of Jin is functorial
 with respect to proper covariance and \'etale contravariance.
\end{enumerate}
\end{ex}

\begin{rem}\label{rem:G-theory}
In general, there should exist a natural transformation
 of bivariant theories:
$$
K_{\mathrm{alg},n}(X \rightarrow S) \rightarrow \biv \KGL n(X/S)
$$
which extends the known natural transformations on
 associated cohomologies and which is compatible with the Chern
 character with values in the motivic bivariant
 rational theory (see below).
\end{rem}

\begin{num}\label{num:functoriality2}
Let $(\varphi,\phi):(\T,\E) \rightarrow (\T',\F)$ be a morphism
 of ring spectra (Definition \ref{df:morph_abs_sp}).
 Then one checks that the associated natural transformation
 of bivariant theories defined in Paragraph \ref{num:functoriality1}
$$
\phi_*:\biv \E {n,m}(X/S) \rightarrow \biv \F {n,m}(X/S)
$$
is compatible with the product structures on each Borel-Moore homology
 (Paragraph \ref{num:product}).\footnote{This comes again
 as the exchange transformations involved in the functoriality and
 in the products are compatible.} So in fact, $\phi_*$ is a
 \emph{Grothendieck transformation} in the sense of \cite[I. 1.2]{FMP}.

Note that this natural transformation then formally induces a natural
 transformation on cohomology theories, compatible with cup-products:
$$
\phi_*:\E^{n,m}(X) \rightarrow \F^{n,m}(X),
$$
as usual. This construction gives many interesting examples.
\end{num}

\begin{ex}
\begin{enumerate}
\item Let $\ell$ be a prime number.
 Assume one of the following settings:
\begin{itemize} 
\item $\base$ is the category of all schemes, $\Lambda=\QQ$,
  $\Lambda_\ell=\QQ_\ell$;
\item  $\base$ is the category of schemes over a prime field $F$ with
 characteristic exponent $p$ such that $\ell \neq p$, $\Lambda=\ZZ[1/p]$
 and $\Lambda_\ell=\ZZ_\ell$.
\end{itemize}
Then, one gets from Example \ref{ex:morph_sp}(2) 
 a natural transformation of bivariant theories:
$$
\HH_{**}^{BM}(X/S,\Lambda) \xrightarrow{\ \sim\ } \HH_{**}^{BM,\et}(X/S,\Lambda_\ell)
$$
whose associated natural transformation on cohomology is the
 (higher) cycle class map in \'etale $\ell$-adic cohomology.
\item Assume $\base$ is the category of all schemes.
 Then one gets from \ref{ex:morph_sp}(3) a higher bivariant Chern
 character:
$$
ch_n:\KGL_n^{BM}(X/S)_\QQ \xrightarrow{\ \sim\ }
 \bigoplus_{i \in \ZZ} \HH^{BM}_{2i+n,i}(X/S,\QQ)
$$
which is in fact a Grothendieck transformation in the
 sense of \cite[I, 1.2]{FMP}.
 From \cite[5.3.3]{Deg12}, it coincides with Gillet's higher Chern
 character on the associated cohomology theories.
 Therefore, it extends Fulton and MacPherson Chern character
 \cite[II. 1.5]{FMP}, denoted in \emph{loc. cit.} by $\tau$.

Suppose $S$ is a regular scheme and $X/S$ is an s-scheme.
 Given the result of Jin (Remark \ref{rem:G-theory})
 and Riou (\cite{Riou}), one gets Adams operations $\psi^i$ on
 Thomason's G-theory $G_n(X)$ and the above isomorphism
 identifies $\HH^{BM}_{2i+n,i}(X/S,\QQ)$ with the eigenvector space
 of $G_n(X)_\QQ$ for the eigenvalue $r^i$ of $\psi^i$,
 $r\neq 0$ being a fixed integer.

Finally, when $S=\spec k$ is the spectrum of a field, 
 from Example \ref{ex:biv_theories}(2) the above higher Chern character
 can be written as an isomorphism:
$$
ch_n:G_n(X)_\QQ \xrightarrow{\sim} \bigoplus_{i \in \ZZ} CH_i(X,n) \otimes \QQ
$$
for any s-scheme $X/k$.
\end{enumerate}
More examples will be given in the section dealing with orientations.
\end{ex}

\begin{rem}\label{rem:modules}
Let us observe finally that the product structure on the Borel-Moore
 homology associated with an absolute ring spectrum $(\E,\T)$
 can be extended to the setting of modules over ring spectra.
 Indeed, given a premotivic adjunction
 $\varphi^*:\T \rightarrow \T'$ and a $\varphi$-module
 $\F$ over $(\E,\T)$ with structural maps $\phi_S$ as in Definition
 \ref{df:modules}, we get a product:
$$
\biv \E {n,m}(Y/X) \otimes \biv \F {s,t}(X/S)
 \rightarrow \biv \F {n+s,m+t}(Y/S), (y,x) \mapsto y.x,
$$
using the construction of Paragraph \ref{num:product}.
 Let us be more explicit. First we remark
 that we have a Grothendieck transformation from the Borel-Moore homology
 represented by $(\E,\T)$ to that represented by $(\varphi^*(\E),\T')$,
 according to \ref{num:functoriality2}. Thus, we can replace
 $\E$ by $\varphi^*(\E)$ to describe the above product.
 In other words, we can assume $\T=\T'$, $\varphi^*=Id$.
 Then, given classes:
$$
y:\un_Y(m)[n] \rightarrow f^!(\E_X),
 x:\un_X(t)[s] \rightarrow p^!(\F_S)
$$
one associates to $x$ the following map:
$$
\tilde x:\E_X(t)[s] \xrightarrow{1_{\E_X} \otimes x}
 \E_X \otimes p^!(\F_S)
 \simeq p^*(\E_S) \otimes p^!(\F_S)
 \xrightarrow{Ex^{*!}_\otimes} p^!(\E_S \otimes \F_S)
 \xrightarrow{p^!(\nu_S)} p^!(\F_S)
$$
using the pairing \eqref{eq:pairing!*} and the structural
 map $\nu_S$ of the module $\F_S$ over $\E_S$.
Then one defines the product as the following composite map:
$$
y.x:\un_Y(m+t)[n+s] \xrightarrow{y(t)[s]} f^!(\E_X)(t)[s]
 \simeq f^!(\E_X(t)[s])
 \xrightarrow{f^!(\tilde x)} f^!p^!(\E_S)=(pf)^!(\F_S).
$$
Similarly, we also define a right action:
$$
\biv \F {n,m}(Y/X) \otimes \biv \E {s,t}(X/S)
 \rightarrow \biv \F {n+s,m+t}(Y/S).
$$
It is straightforward to check these two products
 satisfy the associativity, compatibility with pullbacks and pushforwards,
 and projection formula like the products in bivariant theories
 (cf. Paragraph \ref{num:product}).
\end{rem}

\subsection{Proper support}

\begin{num}\label{num:supp_bivariant}
Let $(\E,\T)$ be an absolute spectrum and $p:X \rightarrow S$ be an s-morphism.
 The six functors formalism gives us two other theories which depend on $X/S$
 as follows:
\begin{align*}
\E_c^{n,m}(X/S)&=\Hom_{\T(S)}\big(\un_S,p_!(\E_X)(m)[n]\big), \\
\E_{n,m}(X/S)&=\Hom_{\T(S)}\big(\un_S(m)[n],p_!p^!(\E_X)\big).
\end{align*}
Using the same techniques as in Paragraph \ref{num:prebivariant},
 one gets the following functoriality:
\begin{itemize}
\item $\E_c^{n,m}(X/S)$ is contravariant in $X/S$
 with respect to cartesian squares, contravariant in $X$
 with respect to proper $S$-morphisms and covariant in $X$
 with respect to \'etale $S$-morphisms;
\item $\E_{n,m}(X/S)$ is contravariant in $X/S$
 with respect to cartesian squares, covariant in $X$
 with respect to all $S$-morphisms
 and contravariant in $X$ with respect to finite $S$-morphisms.
\end{itemize}
So in each cases, $\E_c^{**}$ and $E_{**}$ are contravariant functors
 from the category $\mathcal F\base$ (see \ref{num:prebivariant}) to
 the category of bigraded abelian groups. In fact, they are bivariant
 theories without products where independent squares are the cartesian
 squares and confined maps are respectively the proper morphisms
 and the \'etale morphisms.
\end{num}
\begin{df}\label{df:compact_support_theories}
Given the notations above, the functor $\E^{**}_c$ (resp. $\E_{**}$)
 will be called the \emph{cohomology with compact support}
  (resp. \emph{homology}) associated with $\E$.
\end{df}

\begin{ex}\label{ex:coh_support}
 These notions were classically considered when the base $S$
 is the spectrum of a field $k$. Let $p$ be the characteristic exponent of
 $k$.
\begin{enumerate}
\item When $\E$ is the absolute $\Lambda$-spectrum of \'etale cohomology
 as in \ref{ex:spectra}(1), our formulas for $X/k$ gives the classical
 \'etale cohomology with support.
\item More generally, when $\E$ is the spectrum associated with
 a mixed Weil theory over $k$ as in \ref{ex:spectra}(5), one recovers
 the classical notion of the corresponding cohomology with compact support
 (eg. Betti, De Rham, rigid). See also Corollary \ref{cor:caract_c&h-theory}.
\item When $k=\CC$ (or more generally, one has a given
 embedding of $k$ into $\CC$), and $\E=\HH_B$ is the spectrum
 representing Betti cohomology with integral coefficients,\footnote{Either
 as the one obtained through the corresponding Mixed Weil theory
 or as the canonical spectrum associated with the motivic triangulated
 category $X \mapsto D(X(\CC),\ZZ)$ of \cite[section 1]{Ayo4}.
 For the fact these two versions give the same answer, see \cite[17.1.7]{CD2}.}
 one gets an isomorphism:
$$
(\HH_B)_{n,m}(X/k)=H_n^{sing}(X(\CC),\ZZ)
$$
which is canonical if $m=0$ and only depends on the choice of
 a trivialization of $\HH_B^{1,1}(\GG)$ if $m>0$.

Indeed, one obtains using Grothendieck-Verdier duality for the triangulated
 motivic category  $D(-,\ZZ)$ that when $p:X \rightarrow k$ is the canonical
 projection, the complex
$$
p_!p^!(\un_k)
$$
is the dual of the complex $p_*p^*(\un_k)$ which
 is quasi-isomorphic to $C_*^{sing}(X(\CC))$ by definition.
 So the result follows from the classical definition of singular
 homology.\footnote{One could also use Corollary
 \ref{cor:caract_c&h-theory} to conclude here.}
\item When $\E$ is the absolute motivic $\ZZ[1/p]$-spectrum,
 it follows from \cite[8.7]{CD5} that for any s-scheme $X/k$
  and any integer $n \in \ZZ$, one gets a canonical isomorphism:
$$
\HH_{n,0}(X/k,\ZZ[1/p]) \simeq H_n^{Sus}(X)[1/p]
$$
where the left hand side is the homology in the above sense
 associated with the motivic absolute spectrum $\HH \Lambda$
 and the right hand side is Suslin
 homology (cf. \cite{SV}).\footnote{Note that
 though Suslin homology is defined for s-morphisms $X/S$, it does
 not seem that the above identification extends to cases where
 $S$ is of positive dimension.}
\end{enumerate}
\end{ex}

Let us collect some properties of these two new types of bivariant
 theories.
\begin{prop}\label{prop:basic_bivariant_proper}
Let $\E$ be an absolute spectrum. The following properties hold:
\begin{enumerate}
\item \emph{Homotopy invariance}.-- For any s-scheme $X/S$
 and any vector bundle $p:E \rightarrow X$, the push-forward map
 in bivariant homology:
$$
p_*:\E_{**}(E/S) \rightarrow \E_{**}(X/S)
$$
is an isomorphism.
\item \emph{Proper invariance}.-- Given any $s$-schemes
 $X/T/S$ such that $T/S$ is proper, there exists a canonical
 isomorphism:
$$
\E_c^{**}(X/S) \xrightarrow{\sim} \E_c^{**}(X/T)
$$
which is natural with respect to the functorialities of
 compactly supported cohomology (cf. \ref{num:supp_bivariant}).
\item \emph{Comparison}.-- For any s-scheme $X/S$
 one has natural transformations:
\begin{align*}
\E_c^{n,m}(X/S) & \rightarrow \E^{n,m}(X), \\
\E_{n,m}(X/S) & \rightarrow \biv \E {n,m}(X/S)
\end{align*}
which are isomorphisms when $X/S$ is proper.
\item \emph{Localisation}.--  For any s-scheme $X/S$
 and any closed immersion $i:Z \rightarrow X$
 with complementary open immersion $j:U \rightarrow X$,
 there exists a canonical \emph{localization long exact sequence} of
 the form:
\begin{align*}
\E_c^{n,m}(U/S) & \xrightarrow{j_*}
 \E_c^{n,m}(X/S) \xrightarrow{i^*}
 \E_c^{n,m}(Z/S) \xrightarrow{}
 \E_c^{n+1,m}(U/S)
\end{align*}
which is natural with respect to the functorialities of
 compactly supported cohomology (cf. \ref{num:supp_bivariant}).
\end{enumerate}
\end{prop}
Property (1) follows from the homotopy invariance of the category $\T$,
 which implies that the adjunction map $p_!p^! \rightarrow 1$ is an isomorphism.
 Property (2) and (3) follows from the existence, for $p:X \rightarrow S$,
 of the natural transformation of functors
$$
\alpha_p:p_! \rightarrow p_*
$$
which is an isomorphism when $p$ is proper.
Property (4) is a direct translation of the existence of the localization
 triangle $j_!j^* \rightarrow 1 \rightarrow i_!i^* \xrightarrow{+1}$.

The following corollary justifies our terminology.
\begin{cor}\label{cor:caract_c&h-theory}
Consider an $S$-scheme $X$ with an open $S$-immersion
 $j:X \rightarrow \bar X$ such that $\bar X$ is proper over $S$.
 Let $X_\infty$ be the reduced complement of $j$, $i:X_\infty \rightarrow \bar X$ the corresponding
 immersion.

Then one has canonical long exact sequences:
\begin{align*}
\E^{n-1,m}(\bar X) & \xrightarrow{i^*} \E^{n-1,m}(X_\infty) \rightarrow
 \E_c^{n,m}(X/S) \rightarrow
 \E^{n,m}(\bar X) \xrightarrow{i^*}
 \E^{n,m}(X_\infty), \\
\E_{n,m}(X_\infty/S) & \xrightarrow{i_*}
 \E_{n,m}(\bar X/S) \rightarrow
 \biv \E {n,m}(X/S) \rightarrow
 \E_{n-1,m}(X_\infty/S) \xrightarrow{i_*}
 \E_{n-1,m}(\bar X/S).
\end{align*}
\end{cor}
Indeed, the first long exact sequence is obtained from
 points (3) and (4) of the preceding proposition
 and the second one from Proposition \ref{prop:basic_bivariant}(4)
 and point (3) of the previous proposition.

\begin{rem}
\begin{enumerate}
\item The first long exact sequence gives us the usual way
 to get compactly supported cohomology out of a compactification,
 which can also be interpreted as a canonical isomorphism:
$$
\E_c^{n,m}(X/S) \simeq \E^{n,m}(\bar X,X_\infty)
$$
where the right hand side is the cohomology
 of the pair $(\bar X,X_\infty)$ (as classically considered
 in algebraic topology).
 The second long exact sequence is less usual and gives a way
 to get back Borel-Moore homology from homology.
 In fact, it gives us an interpretation of Borel-Moore homology
 as the compactly supported theory associated with homology.
\item Homology and cohomology with compact support also
 admit \emph{descent long exact sequences} with respect
 to Nisnevich and cdh distinguished squares as in
 \ref{prop:basic_bivariant}(4). We leave the formulation
 to the reader.
\end{enumerate}
\end{rem}

\begin{num}
Assume finally that $(\E,\T)$ has a ring structure.

Then one can define a product, for s-morphisms
 $Y \xrightarrow f X \xrightarrow p S$,
$$
\E_c^{n,m}(Y/X) \otimes \E_c^{s,t}(X/S) \rightarrow \E_c^{n+s,m+t}(Y/S),
 (y,x) \mapsto y.x,
$$
so that the functor $\E_c^{**}$ becomes a bivariant theory in the sense
 of Fulton and MacPherson (cf. \ref{num:product}).
 Indeed, given classes
$$
y:\un_X(m)[n] \rightarrow f_!(\E_Y),
 x:\un_S(t)[s] \rightarrow p_!(\E_S)
$$
we define
$$
y':\E_X(t)[s] \xrightarrow{1_{\E_X} \otimes y}
 \E_X \otimes f_!(\E_Y) \xrightarrow{PF} f_!(f^*\E_X \otimes \E_Y)
 \simeq f_!(\E_Y \otimes \E_Y) \xrightarrow{\mu} f_!(\E_Y)
$$
(where PF stands for projection formula) and then
$$
y.x:\un_S(m+t)[n+s] \xrightarrow x p_!(\E_S)(t)[s]
 \simeq p_!(\E_S(t)[s])
 \xrightarrow{p_!(y')} p_!f_!(\E_Y)=(pf)_!(\E_Y).
$$
Again the formulas required for the product
 of a bivariant theory (cf. \ref{num:product}) follow
 from the six functors formalism.

Such a product does not exist on the bivariant theory
 without products $\E_{**}$ defined above. Instead
 one can define an exterior product:
$$
\E_{**}(X/S) \otimes \E_{**}(Y/S)
 \rightarrow \E_{**}(X \times_S Y/S).
$$
using the following pairing:
\begin{align*}
p_!p^!(\E_S) \otimes q_!q^!(\E_S)
 & \xrightarrow{PF} p_!\big(p^!(\E_S) \otimes p^*q_!q^!(\E_S)\big)
 \xrightarrow{BC} p_!\big(p^!(\E_S) \otimes q'_!p^{\prime*}q^!(\E_S)\big) \\
 & \xrightarrow{PF}
  p_!q'_!\big(q^{\prime*}p^!(\E_S) \otimes p^{\prime*}q^!(\E_S)\big)
	\simeq a_!\big(q^{\prime*}p^!(\E_S) \otimes p^{\prime*}q^!(\E_S)\big) \\
& \xrightarrow{Ex^{*!}}
  a_!\big(q^{\prime*}p^!(\E_S) \otimes q^{\prime!}p^*(\E_S)\big)
  \xrightarrow{\eqref{eq:pairing!*}}
  a!a^!(\E_S \otimes \E_S) \xrightarrow{\mu} a!a^!(\E_S)
\end{align*}
(where BC stands for base change formula) for a cartesian square of s-morphisms
$$\xymatrix@R=16pt@C=30pt{
X \times_S Y\ar|-{p'}[r]\ar_{q'}[d]\ar|{a}[rd] & Y\ar^q[d] \\
X\ar|p[r] & S.}$$
One can check this product is associative, compatible with pushforwards
 and base changes (we will not use these properties).

In a more original way, one can define the following \emph{cap-product},
 pairing of bivariant theories:
\begin{equation}\label{eq:cap}
\E^{**}_c(X/S) \otimes \E_{**}^{BM}(X/S)
 \rightarrow \E_{**}(X/S), (a,b) \mapsto a \cap b
\end{equation}
which, using the preceding notations, is induced by the following
 pairing of functors:
\begin{align*}
p_*p^!(\E_S) \otimes p_!p^*(\E_S)
 &\xrightarrow{PF} p_!\big(p^*p_*p^*(\E_S) \otimes p^*(\E_S)\big)
 \xrightarrow{ad'(p^*,p_*)} p_!\big(p^*(\E_S) \otimes p^*(\E_S)\big) \\
 &\xrightarrow{\eqref{eq:pairing!*}} p_!p^!(\E_S \otimes \E_S)
 \xrightarrow{\mu} p_!p^!(\E_S).
\end{align*}
\end{num}

\begin{rem}
Our definition of cap-product is an extension of the classical definition
 of cap-product, between cohomology and homology, as defined
 for classical spectra (\cite[III, \textsection 9]{Ada}).
 In fact, when $S$ is the spectrum of a field $k$
 and $X/k$ is projective, it also coincides with the cap-product
 appearing in Bloch-Ogus axioms \cite{BO}.
\end{rem}

%% file: fundamental.tex
\subsection{Abstract fundamental classes}

Let us recall the following basic definitions from \cite{FMP}.
\begin{df}\label{df:fdl_orientation}
Let $\E$ be an absolute ring spectrum
 and $f:X \rightarrow S$ be an s-morphism.

An \emph{orientation} for the morphism $f$ with coefficients in $\biv \E {**}$
 will be the choice of an element $\eta_f \in \biv \E {**}(X/S)$
 in the bivariant theory associated with $\E$
 (Def. \ref{df:bivariant_no_prod}).

Given a locally constant function $d:X \rightarrow \ZZ$,
 with values $d(i)$ on the connected components $X_i$ of $X$,
 for $i \in I$, we will say that $\eta_f$ has dimension $d$ if
 it belongs to the group:
$$
\bigoplus_{i \in I} \biv \E {2d(i),d(i)}(X_i/S).
$$
\end{df}
Accordingly, we introduce the following notation for any $\T$-spectrum
 $\F$ over $X$:
$$
\F(d)[2d]=\bigoplus_{i \in I} \F|_{X_i}(d(i))[2d(i)].
$$
Then one defines an \emph{orientation of degree $d$} as a map:
$$
\eta_f:\E_X(d)[2d] \rightarrow f^!\E_S.
$$

\begin{rem}
The word orientation in the context of the previous definition
 has been chosen by Fulton and MacPherson in \cite{FMP}.
 We will also use the
 terminology \emph{fundamental class} for such an orientation when
 it is part
 of a coherent system of orientations: see Definition
 \ref{df:fdl_classes_abstract}. In our main example, the choice
 of an orientation in the sense of $\AA^1$-homotopy theory
 (see Definition \ref{df:orientation}) will indeed canonically
 determine such a coherent system.
\end{rem}

\begin{ex}\label{ex:fdl_et}
Consider any absolute ring spectrum $\E$.
 Then any \'etale s-morphism $f:X \rightarrow S$
 admits a canonical orientation $\fdl_f$ of degree $0$.
 Take:
\begin{equation}\label{eq:fdl_etale}
\fdl_f:\un_X \xrightarrow{\eta_X} \E_X \xrightarrow{\tau^{-1}_f} f^*(\E_S)
 \xrightarrow{\pur_f^{-1}} f^!(\E_S)
\end{equation}
where $\eta_X$ is the unit of the ring spectrum $\E_S$,
 $\tau_f$ is the base change isomorphism (Def. \ref{df:absolute_sp})
 and $\pur_f$ is the purity isomorphism of the six functors formalism
 (see \ref{num:duality_etale}).
\end{ex}

\begin{rem}\label{num:basic_fdl_etale}
As the previous example is a basic piece of our main result,
 Theorem \ref{thm:BM_fundamental_classes},
 we recall the definition of the isomorphism $\pur_f$ of the
 above example. We consider the pullback square:
$$
\xymatrix@=18pt{
X\ar^-\delta[r]\ar@{=}[rd] & X \times_S X\ar^-{f''}[r]\ar^{f'}[d]
 & X\ar^f[d] \\
 & X\ar|f[r] & S
}
$$
where $\delta$ is the diagonal immersion, which is both open
 and closed according to our assumptions on $f$
 (\'etale and separated).
Then we define $\pur_f$ as follows:
$$
f^* \simeq \delta^!f^{\prime\prime!}f^*
 \xrightarrow{Ex^{!*}} \delta^!f^{\prime*}f^!
 \xrightarrow{(1)} \delta^*f^{\prime*}f^! \simeq f^!.
$$
To get the isomorphism (1), we come back to the construction
 of exceptional functors following Deligne (see \cite[2.2]{CD3}).
 Indeed,
 as $\delta$ is an open immersion, we get a canonical identification
 $\delta_! \simeq \delta_\sharp$ of functors so that we get a canonical
 isomorphism $\delta^! \simeq \delta^*$ of their respective right adjoints 
 as required.
\end{rem}

\begin{num}\label{num:tilde_fdl}
Consider an absolute ring spectrum $\E$,
 and an orientation $\eta_f$ of an s-morphism $f:X \rightarrow S$.

Given an s-scheme $Y/X$ and using the product of the Borel-Moore
 $\E$-homology,
 one can associate to $\eta_f$ a map:
$$
\delta(Y/X,\eta_f):\biv \E {**}(Y/X) \rightarrow \biv \E {**}(Y/S), y \mapsto y.\eta_f.
$$
Note that going back to the definition of this product (Par. \ref{num:product}),
 this map can be described up to shift and twist as the composition on the left
 with the following morphism of $\T(X)$:
\begin{equation}\label{eq:associated_pur_iso}
\tfdl_f:\E_X(*)[*] \xrightarrow{1_{\E_X} \otimes \eta_f} \E_X \otimes f^!(\E_S)
 \simeq f^*(\E_S) \otimes f^!(\E_S)
 \xrightarrow{Ex^{*!}_\otimes} f^!(\E_S \otimes \E_S)
 \xrightarrow \mu f^!(\E_S).
\end{equation}
\end{num}
\begin{df}\label{df:strong_orientation}
Consider the above assumptions.
One says that the orientation $\eta_f$ is
\begin{itemize}
\item \emph{strong} if for any s-scheme $Y/X$, the map $\delta(Y/X,\eta_f)$ is
 an isomorphism.
\item \emph{universally strong} if the morphism
 $\tfdl_f$ is an isomorphism in $\T(X)$.
\end{itemize}
\end{df}
As remarked in \cite{FMP}, a strong orientation of $X/S$ is unique
 up to multiplication by an invertible element in $\E^{0,0}(X)$.
 The notion of universally strong is new, as it makes sense only in
 our context. Obviously, universally strong implies strong
 according to Paragraph \ref{num:tilde_fdl}.

\begin{rem}\label{rem:strong_orientation}
Consider the notations of the above definition.
\begin{enumerate}
\item The property of being universally strong
 for an orientation $\eta_f$  as above implies
 that for any smooth morphism $p:T \rightarrow S$,
 the orientation $p^*(\eta_f)$ of $f \times_S T$ is strong
  --- this motivates the name. We will see more implications
 of this property in Section \ref{sec:duality}.
\item The data of the orientation $\fdl_f$ is equivalent to the data of the map
 $\tfdl_f$ as the map $\fdl_f$ is equal to the following composite:
$$
\un_X(*)[*] \xrightarrow{\eta_X} \E_X(*)[*] \xrightarrow{\tfdl_f} f^!\E_S
$$
where $\eta_X$ is the unit of the ring spectrum $\E_X$.
\end{enumerate}
\end{rem}

\begin{ex}\label{ex:fdl_etale}
Consider the notations of the previous definition 
 and assume that $f$ is \'etale as in Example \ref{ex:fdl_et}.

It follows from point (2) of the preceding remark that the map $\tfdl_f$
 associated with the orientation $\fdl_f$ of the latter example
 is equal to the following composite morphism:
$$
\tfdl_f:\E_X
 \xrightarrow{\tau^{-1}_f} f^*(\E_S)
 \xrightarrow{\pur_f^{-1}} f^!(\E_S),
$$
using the notations of the example. Thus, $\tfdl_f$ is an isomorphism:
 the canonical orientation $\eta_f$ of an \'etale morphism $f$ is 
 universally strong.
\end{ex}

\begin{df}\label{df:fdl_classes_abstract}
Given a class $\C$ of morphisms of schemes closed under composition,
 a \emph{system of fundamental classes} for $\C$ with coefficients in $\E$ 
 will be the datum for
 any $f \in \C$ of an orientation $\eta^\C_f$ such that for
 any composable maps $Y \xrightarrow g X \xrightarrow f S$ in $\C$
 one has the relation:
$$
\eta^\C_g.\eta^\C_f=\eta^\C_{f \circ g}
$$
using the product of the bivariant theory $\biv \E {**}$. 
 This relation will be referred to as the \emph{associativity formula}.
\end{df}
Recall the aim of this paper is to construct
 a system of fundamental classes for a class of morphisms
 as large as achievable under the minimal possible choices.
 
\subsection{Global orientations}\label{sec:orientation}

\begin{num}\label{num:orientation}
Recall from our convention that
 we assume from now on that any motivic triangulated category $\T$
 is equipped with a premotivic adjunction:
$$
\can^*:\SH \rightarrow \T.
$$

Consider an absolute ring $\T$-spectrum $\E$.
 Let us fix a scheme $S$ in $\base$. Then $\can_*(\E_S)$
 is a motivic ring spectrum and for any smooth scheme $X/S$,
 for any pair $(n,m) \in \ZZ^2$,
 one gets an isomorphism:
\begin{align*}
\Hom_{\SH(S)}(\Sus X_+,\can_*(\E_S)(m)[n])
 & \xrightarrow{\sim} \Hom_{\T(S)}(\can^*(\Sus X_+),\E_S(m)[n]) \\
 & \simeq \Hom_{\T(S)}(f_\sharp(\un_X),\E_S(m)[n])
 \simeq \E^{n,m}(X)
\end{align*}
In other words, the ring spectrum $\can_*(\E_S)$ in $\SH(S)$
 represents the cohomology $\E^{**}$ restricted to smooth
 $S$-schemes and the above isomorphism is also compatible 
 with cup-products.\footnote{Beware however that
 $\can_*(\E_S)$ for various schemes $S$ only gives a section
 of $\SH$, not necessarily a \emph{cartesian} one.}

Therefore, one can apply all the definitions and results of
 orientation theory of motivic homotopy theory for
 which we refer to \cite{Deg12}. In the remainder of this section,
  we recall these results, applied more specifically to our situation.

As usual $\tilde \E^{**}$ denotes the reduced cohomology
 with coefficients in $\E$.
 As, by definition, $M_S(\PP^1_S)=\un_S \oplus \un_S(1)[2]$
 and because $\un_S(1)$ is $\otimes$-invertible, we
 get a canonical isomorphism:
$$
\tilde \E^{2,1}(\PP^1_S) \xrightarrow \psi \E^{0,0}(S)
$$
where $\PP^1_S$ is pointed by $\infty$. Therefore, the unit $\eta_S$
 of the ring spectrum $\E$ induces a canonical cohomology class
 $\sigma^\E_S=\psi^{-1}(\eta_S) \in \tilde \E^{2,1}(\PP^1_S)$ --- classically
 called the \emph{stability class}.

As in \cite[Def. 2.1.2]{Deg12},
 we let $\PP^\infty_S$ be the colimit, in the category
 of Nisnevich sheaves of sets over the category of smooth $S$-schemes,
 of the inclusions
 $\PP^n_S \rightarrow \PP^{n+1}_S$ by means of the first coordinates.
\end{num}
\begin{df}\label{df:orientation}
Consider the above notations.
An \emph{orientation} of the absolute ring $\T$-spectrum $\E$
 will be the datum, for any scheme $S$ in $\base$, of a class 
 $c_S \in \tilde \E^{2,1}(\PP^\infty_S)$ such that:
\begin{itemize}
\item the restriction of $c_S$ to $\PP^1_S$ equals the stability class
  $\sigma^\E_S$ defined above;
\item for any morphism $f:T \rightarrow S$, one has: $f^*(c_S)=c_T$.
\end{itemize}
For short, we will say that $(\E,c)$ is an absolute oriented
 ring $\T$-spectrum (or simply spectrum).

A morphism of absolute oriented ring spectra
 $(\T,\E,c) \rightarrow (\T',\F,d)$
 will be a morphism of absolute ring spectra (Def. \ref{df:morph_abs_sp})
 $(\varphi,\psi)$ such that for any scheme $S$,
 the map induced on cohomology 
$$
\psi_*:\tilde \E^{2,1}(\PP^\infty_S) \rightarrow \tilde \F^{2,1}(\PP^\infty_S)
$$
sends $c_S$ to $d_S$.
\end{df}

\begin{rem}
We will show later (Example \ref{ex:orientations&orientations})
 that an orientation of the ring spectrum $\E_S$
 does correspond to a family of orientations of the associated
 Borel-Moore homology in the sense of Definition \ref{df:fdl_orientation}.
\end{rem}

\begin{ex}\label{ex:orient_spectra}
Each of the absolute ring spectra of Example \ref{ex:spectra} admits a canonical
 orientation; see \cite[2.1.4]{Deg12}.
\end{ex}

\begin{num}
Consider the previous
  assumptions and notations.
 Recall one can build out of the orientation $c$ a complete theory of
 characteristic classes.
 The first building block is the first Chern class which follows rightly
 from the class  $c$, seen as a morphism. Indeed,
 from \cite[2.1.8]{Deg12}:
\begin{align*}
c_1:\Pic(S) & \rightarrow  \Hom_{\mathscr H(S)}(S_+,\PP^\infty_S)
 \xrightarrow{\Sus} \Hom_{\SH(S)}(\Sus S_+,\Sus \PP^\infty_S) \\
 & \xrightarrow{\tau^*} \Hom_{\T(S)}(\un_S,M_s(\PP^\infty_S)) 
 \xrightarrow{(c_{S})_*} \Hom_{\T(S)}(\un_S,\E_S(1)[2])=\E^{2,1}(S),
\end{align*}
--- $\mathscr H(S)$ is Morel and Voevodsky's pointed unstable
 homotopy category.

One then deduces from \cite[2.1.13]{Deg12} that the cohomology theory
 $\E^{**}$ satisfies the classical projective bundle formula,
 which will be freely
 used in the rest of the text. Further, one gets higher Chern classes
 (see \cite[Def. 2.1.16]{Deg12}) satisfying the following properties:
\end{num}
\begin{prop}\label{prop:chern_classes}
 Considering the above notations and a given base scheme $X$ in $\base$,
 the following assertions hold:
\begin{enumerate}
\item For any vector bundle $E$ over $X$, there exist
 \emph{Chern classes} $c_i(E) \in \E^{2i,i}(X)$ uniquely defined
 by the formula:
\begin{equation} \label{eq:Chern}
\sum_{i=0}^n p^*(c_i(E)).\big(-c_1(\lambda)\big)^{n-i}=0,
\end{equation}
$c_0(E)=1$ and $c_i(E)=0$ for $i \notin [0,n]$. As usual,
we define the \emph{total Chern class} in the polynomial ring $\E^{**}(X)[t]$:
$$
c_t(E)=\sum_i c_i(E).t^i.
$$
\item Chern classes are nilpotent, compatible with pullbacks in $X$,
 invariant under isomorphisms of vector bundles
 and satisfy the Whitney sum formula: for any vector bundles $E$, $F$ over $X$,
$$
c_t(E \oplus F)=c_t(E).c_t(F).
$$
\item There exists a (commutative) formal group law $F_X(x,y)$
 with coefficients in the ring $\E^{**}(X)$
 such that for any line bundles $L_1$, $L_2$ over $X$, the following
 relation holds:
$$
c_1(L_1 \otimes L_2)=F_X\big(c_1(L_1),c_1(L_2)\big) \in \E^{2,1}(X),
$$
--- which is well defined as the cohomology class $c_1(L_i)$ is nilpotent.
 Moreover, for any morphism $f:Y \rightarrow X$, one gets the relation:
 $f^*(F_X(x,y))=F_Y(x,y)$, in other words,
 the morphism of rings  $f^*:\E^{**}(X) \rightarrow \E^{**}(Y)$ 
  induces a morphism of formal group laws.
\end{enumerate}
\end{prop}
This is the content of \cite[2.1.17, 2.1.22]{Deg12}.

\begin{rem}\label{rem:morph_oriented&Chern}
Consider a morphism of oriented ring spectra:
$$
(\varphi,\psi):(\T,\E,c) \rightarrow (\T',\F,d)
$$
as in Definition \ref{df:orientation},
 and $\psi_*:\E^{**}(X) \rightarrow \F^{**}(X)$ the map induced in cohomology.
Let us denote by $c_n(E)$
 (resp. $d_n(E)$) the $n$-th Chern class in $\E^{**}(X)$ (resp. $\F^{**}(X)$)
 associated with a vector bundle $E/X$ using the previous proposition.
 It follows from the construction of Chern classes and the fact $\psi_*$
 respects the orientation that we get the relation:
$$
\psi_*(c_n(E))=d_n(E).
$$
\end{rem}

Before going down the path of characteristic classes,
 let us recall that, according to Morel, the existence of an orientation
 on an absolute ring spectrum implies the associated cohomology
 is graded commutative. Actually, this property holds for the associated
 bivariant theory in the following terms.
\begin{prop}\label{prop:orient_comm}
Let $(\E,c)$ be an absolute oriented ring $\T$-spectrum.
 Then for any cartesian square of s-morphisms
$$
\xymatrix@=10pt{
Y\ar[r]\ar[d] & X\ar^p[d] \\
T\ar_f[r] & S,
}
$$
and any pair $(x,t) \in \biv \E {n,i} (X/S) \times \biv \E {m,j} (T/S)$,
 the following relation holds in $\biv \E {n+m,i+j} (Y/S)$:
$$
p^*(t).x=(-1)^{nm}f^*(x).t.
$$
\end{prop}
\begin{proof}
For cohomology with support, this was proved in \cite[2.1.15]{Deg12}.
 The proof here is essentially the same.
 Recall that $M(\GGx S)=\un_S \oplus \un_S(1)[1]$.
 The map permuting the factors of $\GG \times \GG$
 therefore induces an endomorphism of $\un_S(1)[1]$ which after
 desuspension and untwisting gives an element\footnote{In fact,
 this element is the image of Morel's element
 $\epsilon \in \pi_0(S^0_S)$
 by the functor $\tau^*:\SH(S) \rightarrow \T(S)$, justifying our notation.}
$$
\epsilon \in \End_{\T(S)}(\un_S)=H_{0,0}(S,\T).
$$
 Formally, we get (under the assumptions of the proposition)
 the following relation:
\begin{equation}\label{eq:basic_comm}
f^*(t).x=(-1)^{nm-ij}\epsilon^{ij}.p^*(x).t
\end{equation}
where the multiplication by $(-1)^{nm-ij}\epsilon^{ij}$ is seen
 via the action of $\E_{0,0}(S)$ on $\biv \E {**}(X,S)$
 --- see Remark \ref{rem:cup_biv&act_coh&supp_prod}(2).
 Therefore, we are done as $\epsilon=-1$ in $\E_{0,0}(S)$.
 Indeed, according to relation \eqref{eq:basic_comm} applied
 with $Y=X=T=S$, $(n,i)=(m,j)=(2,1)$, we have
 $c^2=-\epsilon.c^2 \in \E^{4,2}(\PP^2)$ 
 and the projective bundle theorem for $\E^{**}$ thus concludes.
\end{proof} 

\begin{num}\label{num:thom}
Suppose given a motivic triangulated category $\T$.
Recall that given a vector bundle $p:E \rightarrow X$
 with zero section $s$, one defines the Thom space attached with $E/X$
 as:
$$
\MTh(E/X):=p_\sharp s_*(\un_X).
$$
This is also the image under the right adjoint $\can^*:\SH(X) \rightarrow \T(X)$
 of the classical Thom space $E/E-X$.
 Given an absolute $\T$-spectrum $\E$,
 we define the $\E$-cohomology of the Thom space of $E$ as:
$$
\E^{n,m}(\Th(E))=\Hom_{\T(X)}(\MTh(E),\E_X(m)[n]).
$$
Note that by adjunction, one immediately gets an isomorphism:
\begin{equation}\label{eq:compute_Thom_coh}
\E^{n,m}(\Th(E)) \xrightarrow{\alpha_s^*} \E^{n,m}_X(E)
 =\biv \E {-n,-m}(X \xrightarrow s E)
\end{equation}
where the map $\alpha_s:s_! \rightarrow s_*$ is the canonical 
 isomorphism obtained from the six functors formalism, as $s$ is proper.
Finally, one gets the classical short exact sequence:
\begin{equation}\label{eq:thom&proj}
0 \rightarrow \E^{n,m}(\Th(E))
 \xrightarrow{\partial} \E^{n,m}(\PP(E \oplus 1))
 \xrightarrow{\nu^*} \E^{n,m}(\PP(E)) \rightarrow 0
\end{equation}
where $\nu:\PP(E) \rightarrow \PP(E \oplus 1)$ is the canonical
 immersion of the projective bundle associated with $E/X$
 into its projective completion --- cf. the construction
 of \cite[2.2.1]{Deg12}. 
\end{num}
\begin{prop}\label{prop:thom_class}
Let $(\E,c)$ be an absolute oriented ring $\T$-spectrum
 and $E/X$ be a vector bundle of rank $r$.

One defines the \emph{Thom class} of $E$ in $\E^{2r,r}(\PP(E\oplus 1))$ as:
\begin{equation}\label{eq:thom_class}
\thom(E)=\sum_{i=0}^r p^*(c_i(E)).\big(-c_1(\lambda)\big)^{r-i}.
\end{equation}
Then $\thom(E)$ induces a unique class $\rthom(E) \in \E^{2r,r}(\Th(E))$,
 called the \emph{refined Thom class},
 such that $\partial(\rthom(E))=\thom(E)$.
 
Moreover, $\E^{**}(\Th(E))$ is
 a free graded $\E^{**}(X)$-module of rank $1$ with base $\rthom(E)$.
 In other words, the sequence \eqref{eq:thom&proj} is split and we
 get a canonical isomorphism:
\begin{equation}\label{eq:thom_iso}
\tau_E:\E^{**}(X) \rightarrow \E^{**}(Th(E)), x \mapsto x.\rthom(E).
\end{equation}
\end{prop}
For the proof, see \cite[2.2.1, 2.2.2]{Deg12}.
 The preceding isomorphism is traditionally called the \emph{Thom isomorphism}
 associated with the vector bundle $E/X$.
 It follows from Remark \ref{rem:morph_oriented&Chern} that morphisms
 of oriented absolute ring spectra respect Thom classes as well as refined Thom classes.

\begin{rem}
Note that the proposition makes sense even when the rank of $E/X$
 is not constant. Indeed, in any case, the rank is locally constant on $X$, \emph{i.e.}
 constant over each connected component $X_i$ of $X$ and we just take
 direct sums of the Thom classes restricted to each connected component,
 in the canonical decomposition:
$$
\E^{**}(\Th(E))=\bigoplus_i \E^{**}\left(\Th\big(E|_{X_i}\big)\right).
$$
\end{rem}

\begin{num}\label{num:univ_MGL}
The natural functor $\can^*:\SH(S) \rightarrow \T(S)$ is monoidal.
 In particular, we get a canonical
 absolute ring $\T$-spectrum $\can^*(\MGL)$, the avatar of algebraic
 cobordism in $\T$. Note that by definition, it satisfies the following
 formula for any base scheme $S$:
$$
\can^*(\MGL_S)=\mathrm{hocolim}_{n \geq 0} \MTh_S(\gamma_n)(-n)[-2n]
$$
where $\gamma_n$ is the tautological vector bundle on the infinite Grassmannian
 of $n$-planes over $S$.
 Note that, by adjunction,
 a structure of a $\can^*(\MGL_S)$-module (resp. $\can^*(\MGL_S)$-algebra)
 over a $\T$-spectrum $\E_S$
 is the same thing as a structure of $\MGL_S$-module (resp. $\MGL_S$-algebra)
 over $\can_*(\E_S)$.
 Thus, according to \cite[4.3]{Vez} (see also \cite[2.2.6]{Deg12}),
 there is a bijection between the following sets:
\begin{enumerate}
\item the orientations $c$ on $\E$ as defined in \ref{df:orientation};
\item the structures of an $\MGL$-algebra on $\E$ as defined in
 \ref{df:morph_abs_sp}.
\end{enumerate}
In particular, the class $c$ induces a unique morphism of
 absolute ring spectra:
$$
(\can,\phi^c):\MGL \rightarrow \E.
$$
This morphism induces a natural transformation of cohomology theories,
 compatible with cup-products (see \ref{num:functoriality2}):
$$
\phi^c_*:\MGL^{n,m}(X) \rightarrow \E^{n,m}(X)
$$
 which by definition satisfies the property that
  $\tau^c_*(c^\MGL)=c$ in $\tilde \E^{2,1}(\PP^\infty_S)$.
 In other words, $(\can,\phi^c)$ is a morphism of oriented ring spectra.
 
This fact suggest the following definition.
\end{num}
\begin{df}\label{df:weak_orient}
A \emph{weak orientation} of an absolute $\T$-spectrum $\E$
 is a structure of a $\can$-module over the absolute ring spectrum $\MGL$
 --- in short, an $\MGL$-module structure.
\end{df}

\begin{rem}
\begin{enumerate}
\item Typically, an $\MGL$-module $\E$ will not possess Chern classes or
 Thom classes but will possess a structural action of the ones
 which naturally exists for $\MGL$.
 As we will see below, this is enough to obtain Gysin morphisms and duality
 results for the cohomology represented by $\E$.
\item When $\E$ is an absolute ring spectrum,
 the difference between a weak orientation and an orientation is the one
 between an $\MGL$-module structure and an $\MGL$-algebra structure.
\end{enumerate}
\end{rem}

\begin{rem} In this example, we consider one of the following
 assumptions on a given ring of coefficients $\Lambda$:
\begin{itemize}
\item the category of schemes $\base$ can be arbitrary
 and $\Lambda=\QQ$;
\item the category $\base$ is a subcategory of the category
 of $k$-schemes for a field $k$ of characteristic exponent $p$
 and $\Lambda=\ZZ[1/p]$.
\end{itemize}
We denote by $\HH\Lambda$ the motivic (Eilenberg-MacLane) absolute spectrum
 with coefficients in $\Lambda$ (Example \ref{ex:spectra}).
 Then given an absolute oriented ring $\T$-spectrum $(\E,c)$
 which is $\Lambda$-linear and  whose associated formal group law is additive,
 there exists a unique morphism of absolute ring spectra
$$(\can,\tilde \phi^c):\HH\Lambda \rightarrow \E
$$
such that $\tilde \phi^c(c^{\HH\Lambda})=x$.

Indeed, in the first case, this follows from \cite[Th. 14.2.6]{Deg12},
 and in the second case from Hoyois-Hopkins-Morel Theorem
 (see \cite{Hoy}).
 See \cite[5.3.1, 5.3.9]{Deg12} for more details.
\end{rem}

\begin{num}\label{num:Thom_virtual}
Consider again the setting of Paragraph \ref{num:thom}.
 Recall from \cite[2.4.18]{Deg12} that the association $E \mapsto \MTh(E)$
 can be uniquely extended to a monoidal functor:
$$
\MTh:\underline K(X) \rightarrow \Pic(\T,\otimes)
$$
where $\underline K(X)$ is the Picard groupoid of \emph{virtual vector bundles} over $X$
 (\cite[4.12]{Del}) and $\Pic(\T,\otimes)$ that of $\otimes$-invertible objects of $\T(X)$,
 morphisms being isomorphisms. 
 Actually, this extension follows from the fact that for any short exact sequence
 of vector bundles over $X$:
\begin{equation}\tag{$\sigma$}
0 \rightarrow E' \rightarrow E \rightarrow E'' \rightarrow 0
\end{equation}
there exists a canonical isomorphism:
\begin{equation}\label{eq:Thom_virtual}
\epsilon_\sigma:\MTh(E') \otimes \MTh(E'')  \rightarrow \MTh(E)
\end{equation}
and the isomorphisms of this form satisfy the coherence conditions
 of \cite[4.3]{Del}.
 
 Then the definition of the refined Thom class can be extended
 to Thom spaces of virtual vector bundles using the following lemma
 (see \cite[2.4.7]{Deg12}).
\end{num}
\begin{lm}\label{lm:product&thom}
Consider as above an exact sequence $(\sigma)$ of vector bundles over a scheme $X$.
Then the following relation holds in $\E^{**}(\Th(E))$:
$$
\rthom(E/X)=\rthom(E'/X).\rthom(E/E')
$$
using the product 
 $\E^{**}_X(E') \otimes \E^{**}_{E'}(E) \rightarrow \E^{**}_X(E)$
 of the corresponding bivariant theories.
\end{lm}

\begin{num}\label{num:Thom_class_vb}
Consider an arbitrary vector bundle $E/X$.
 Then we obviously get a perfect pairing of $\E^{**}(X)$-modules:
$$
\E^{**}(\Th(E)) \otimes \E^{**}(\Th(-E)) \rightarrow \E^{**}(X),
 (a,b) \mapsto a \otimes_X b.
$$
We let $\rthom(-E)$ be the unique element of
 $\E^{**}(\Th(-E))$ such that $\rthom(E) \otimes_X \rthom(-E)=1$
 so that $\rthom(-E)$ is a basis of the $E^{**}(X)$-module
 $\E^{**}(\Th(-E))$.

Let now $v$ be a virtual vector bundle over $X$. 
Then we deduce from the preceding lemma that for
any $X$-vector bundles $E$ and $E'$ such that $v=[E]-[E']$, the class
$$
\rthom(v)=\rthom(E) \otimes_X \rthom(-E')
$$
is independent of the choice of $E$ and $E$'.
\end{num}
\begin{df}\label{df:Thom_class_vb}
Consider the notations above. We define the \emph{Thom class}
 of the virtual vector bundle $v$ over $X$ as the element
 $\rthom(v) \in \E^{**}(\Th(v))$ defined by the preceding
 relation.
\end{df}
Recall $\rthom(v)$ is a basis of $\E^{**}(\Th(v))$ as an
 $\E^{**}(X)$-module. In other words, the map:
\begin{equation}\label{eq:thom_iso_virt}
\tau_v:\E^{**}(X) \rightarrow \E^{**}(Th(v)), x \mapsto x.\rthom(v).
\end{equation}
is an isomorphism, again called the \emph{Thom isomorphism} associated
 with the virtual vector bundle $v$.
 Besides, the preceding lemma shows we have the relation:
\begin{equation}\label{eq:rthom_add}
\rthom(v+v')=\rthom(v) \otimes_X \rthom(v')
\end{equation}
in $\E^{**}(\Th(v+v'))$.

\begin{num}
Consider again the setting of Paragraph \ref{num:thom}.
 According to the projection formulas of the motivic triangulated
 category $\T$, 
 one gets an isomorphism of functors in $M$, object of $\T(X)$,
\begin{equation}\label{eq:functorial_thom_iso1}
p_\sharp s_*(M)=p_\sharp s_*(\un_S \otimes s^*p^*M)
 \simeq p_\sharp s_*(\un_S) \otimes M=\MTh_S(E) \otimes M.
\end{equation}
Therefore, $p_\sharp s_*$ is an equivalence of categories
 with quasi-inverse:
\begin{equation}\label{eq:functorial_thom_iso2}
M \mapsto s^!p^*(M)=\MTh_S(-E) \otimes M.
\end{equation}
The following proposition is a reinforcement of Proposition
 \ref{prop:thom_class}.
\end{num}
\begin{prop}\label{prop:fdl_thom}
Let $(\E,c)$ be an absolute oriented $\T$-spectrum.

Then for a vector bundle $E/X$ with zero section $s$,
 the refined Thom class $\rthom(E)$, seen as an element $\fdl_s$ of
 $\biv \E {**}(X \xrightarrow s E)$ through the identification
 \eqref{eq:compute_Thom_coh}, is a universally strong orientation of $s$,
 with degree equal to the rank $r$ of $E/X$.

Moreover, these orientations form a system of fundamental classes
 with respect to the class of morphisms made by the zero sections of vector bundles
 (over schemes in $\base$).
\end{prop}
\begin{proof}
Let us consider the following map:
$$
\pur'_E:\MTh(E) \otimes \E_X \xrightarrow{\rthom(E) \otimes 1} \E_X \otimes \E_X(r)[2r]
 \xrightarrow{\mu} \E_X(r)[2r]
$$
where $\mu$ is the multiplication map of the ring spectrum $\E_X$.
It follows formally from this construction that the map
\begin{align*}
\E^{n,m}(\Th(-E))=&\Hom_{\T(X)}\left(\MTh(-E),\E(m)[n]\right)
 \simeq \Hom_{\T(X)}\left(\un_X,\MTh(E) \otimes \E(m)[n]\right)
\\
& \xrightarrow{\quad (\pur'_E)_*\quad}
 \Hom_{\T(X)}(\un_X,\E(m+r)[n+2r])
  =\E^{n+2r,m+r}(X)
\end{align*}
induced
 by $\pur'_E$ after applying the functor
 $\Hom_{\T(X)}(\un_X,-(*)[*])$ is equal to the inverse of
 the Thom isomorphism \eqref{eq:thom_iso_virt}. 
 Because Thom classes are stable under pullbacks,
 we further deduce that for any smooth morphism $f:Y \rightarrow X$,
 the map induced after applying the functor $\Hom_{\T(X)}(M_X(Y),-(m)[n])$
 is equal to the inverse of the Thom isomorphism associated with
 the virtual vector bundle $(-f^{-1}(E))$ over $Y$.
 As the objects $M_X(Y)(-m)$ for $Y/X$ smooth and $m \in \ZZ$
 form a family of generators for the triangulated category $\T(X)$
 (according to our conventions on motivic triangulated categories),
 we deduce that $\pur'_E$ is an isomorphism.

Now, one can check going back to definitions that the following isomorphism:
$$
\E_X=s^!p^*(\MTh(E) \otimes \E_X) \xrightarrow{s^!p^*(\pur'_E)}
 s^!p^*(\E_X(r)[2r]) \simeq s^!(\E_X)(r)[2r]
$$
is equal to the map $\tfdl_s(r)[2r]$ associated to $\fdl_s$
 as in \eqref{eq:associated_pur_iso}. This implies the first claim.

Then the second claim is exactly Lemma \ref{lm:product&thom}.
\end{proof}

\begin{rem}
We will remember from the above proof that, in the condition of the proposition,
 given any virtual vector bundle $v$ over $X$ with virtual rank $r$, the following
 map:
\begin{equation}\label{eq:thom_iso_virtual}
\pur'_v:\MTh(v) \otimes \E_X \xrightarrow{\rthom(v) \otimes 1} \E_X \otimes \E_X(r)[2r]
 \xrightarrow{\mu} \E_X(r)[2r]
\end{equation}
is an isomorphism --- this follows from the case where $v=[E]$ explicitly treated in the
 proof, relation \eqref{eq:rthom_add} and the fact $\rthom(0)=1$.
 This map obviously represents the Thom isomorphism \eqref{eq:rthom_add}
 so we will also call it the \emph{Thom isomorphism} when no confusion
 can arise.
\end{rem}

\subsection{The smooth case}

\begin{num}\label{num:rel_Tspectrum}
Let $\T$ be a triangulated motivic category
 with, according to our conventions, a premotivic adjunction $\tau^*:\SH \rightarrow \T$.

As usual we call closed $S$-pair any pair of $S$-schemes $(X,Z)$ such that
 $X/S$ is smooth and $Z$ is a closed immersion. One defines
 the motive of $X$ with support in $Z$ as:
$$
M_S(X/X-Z):=p_\sharp i_*(\un_Z)
$$
where $p$ is the structural morphism of $X/S$, $i$ the immersion of $Z$ in $X$.
 Alternatively, one can equivalently put:
$$
M_S(X/X-Z):=\tau^*(\Sus X/X-Z)
$$
 where $X/X-Z$ is the quotient computed in the category pointed
 Nisnevich sheaves of sets over the category of smooth $S$-schemes,
 seen as an object of the pointed $\AA^1$-homotopy category over $S$.

A particular example that we have already seen is given for a vector bundle $E/X$,
 with $X/S$ smooth. We then put:
$$
\MTh_S(E):=M_S(E/E-X)
$$
extending the definition of Paragraph \ref{num:thom} --- for which we
 had $X=S$.

These objects satisfy a classical formalism which has been summarized
 in \cite[2.1]{Deg8}. In particular,
 they are covariant in the closed $S$-pair $(X,Z)$ --- recall morphisms
 of closed pairs are given by commutative squares
 which are topologically cartesian\footnote{\emph{i.e.} cartesian on
 the underlying topological spaces.}; one says such a morphism is
 \emph{cartesian} if the corresponding square is cartesian.

Given a closed $S$-pair $(X,Z)$, we define the associated deformation
 space as:
$$
D_ZX:=B_{Z \times \{0\}}(\AA^1_X)-B_ZX.
$$
Let us put $D=D_ZX$ for the rest of the discussion.
 This deformation space contains
 as a closed subscheme the scheme $\AA^1_Z \simeq B_Z(\AA^1_Z)$.
 It is flat over $\AA^1$ and the fiber of the closed pair
 $(D,\AA^1_Z)$ over $1$ (resp. $0$) is  $(X,Z)$ (resp. $(N_ZX,Z)$). 
 Therefore one gets cartesian morphisms of closed $S$-pairs:
\begin{equation}\label{eq:deformation}
(X,Z) \xrightarrow{d_1} (D,\AA^1_Z) \xleftarrow{d_0} (N_ZX,Z).
\end{equation}
\end{num}
\begin{thm}[Morel-Voevodsky]\label{thm:purity_MV}
Consider a closed $S$-pair $(X,Z)$ such that $Z/S$ is smooth.

Then the induced maps
$$
M_S(X/X-Z) \xrightarrow{d_{1*}}
 M_S(D/D-\AA^1_Z) \xleftarrow{d_{0*}}
 M\Th_S(N_ZX)
$$
are isomorphisms in $\T(S)$.
\end{thm}
The proof is well known --- see for example \cite[Th. 2.4.35]{CD3}.

\begin{df}\label{df:purity_closed_pairs}
Under the assumptions of the previous theorem,
 we define the \emph{purity isomorphism} associated
 with $(X,Z)$ as the composite isomorphism:
$$
\pur_{(X,Z)}:M_S(X/X-Z)
 \xrightarrow{d_{1*}} M_S(D/D-\AA^1_Z)
 \xrightarrow{d_{0*}^{-1}} M\Th_S(N_ZX).
$$
\end{df}

\begin{num}\label{num:fdl_retraction}
We can derive a functorial version of the preceding purity
 isomorphism. Consider for simplicity the case of a closed immersion $i:S \rightarrow X$
 which admits a smooth retraction $p:X \rightarrow S$.
 We deduce from the preceding construction the following isomorphism functorial
 in a given object $\E$ of $\T(S)$:
$$
\pur_{p,i}:p_\sharp i_*(\E)=p_\sharp i_*(\un \otimes s^*p^*\E)
 \xrightarrow{\quad \sim\quad}  M_S(X/X-S) \otimes \E \\
 \xrightarrow{\pur_{(X,Z)}} \MTh_S(N_SX) \otimes \E,
$$
where $p$ is the structural map of $X/S$,
 and the first isomorphism is given by the projection formulas
 associated with $i_*$ and $p_\sharp$. As the $\T$-spectrum
 $\MTh_S(N_SX)$ is $\otimes$-invertible,
 we deduce that the functor $p_\sharp s_*$ is an equivalence of categories. 
 Then by adjunction, we get a dual isomorphism:
\begin{equation}\label{eq:functorial_thom_iso}
\pur'_{p,i}:i^! p^*(\E) \rightarrow \MTh_S(-N_SX) \otimes \E
\end{equation}
using the notation of Paragraph \ref{num:Thom_virtual}. Note by the way
 this map can be written as the following composite of isomorphisms:
\begin{equation}
\begin{split}\label{eq:functorial_thom_iso_bis}
\pur'_{p,i}:i^! p^*(\E)
 \xrightarrow{\sim} \uHom(M_S(X/X-S),\E)
 \xrightarrow{\big(\pur_{(X,Z)}^{-1}\big)^*} &\uHom(\MTh_S(N_SX),\E) \\
 &\simeq \MTh_S(-N_SX) \otimes \E
\end{split}
\end{equation}

Consider now an absolute oriented spectrum $(\T,\E,c)$.
 Let $n$ be the function on $S$ which measures the local codimension of $S$ in $X$.
 It is locally constant as $i$ is a regular closed immersion --- as it admits
 a smooth retraction.
 Then we deduce from the previous purity isomorphism and from the Thom
 isomorphism \eqref{eq:thom_iso_virtual} the following one:
$$
\tfdl_i:\E_S(n)[2n] \xrightarrow{(\pur'_{-N_SX})^{-1}} \MTh(-N_SX) \otimes \E_S
 \xrightarrow{(\pur'_{p,i})^{-1}} i^! p^*(\E_S) \simeq i^!(\E_X). 
$$
As explained in Remark \ref{rem:strong_orientation}(2),
 we can associate to this isomorphism the following orientation:
$$
\fdl_i:\un_S(n)[2n] \rightarrow \E_S(n)[2n] \xrightarrow{\tfdl_i} i^!(\E_X)
$$
which is therefore universally strong.
\end{num}
\begin{prop}\label{prop:fdl_retraction}
Consider the above notations and assumptions.
 The universally strong orientations $\fdl_i$ constructed above
 form a system of fundamental classes (Definition \ref{df:fdl_classes_abstract})
 with coefficients in $\E$ for closed immersions
 which admit a smooth retraction.
\end{prop}
Indeed, all what remains to be proved is the associativity formula.
 This follows from the use of the double deformation space
 and the associativity formula for Thom classes (Proposition \ref{prop:fdl_thom}).
 The reader can consult \cite[2.4.9]{Deg12} for details.
 
\begin{rem}\label{rem:fdl_retraction&thom}
\begin{enumerate}
\item Note the same result could have been derived replacing closed immersions
 which admits a smooth retraction by closed immersions between smooth
 schemes over some fixed base. We will derive this case later from our
 more general results.
\item We will need the following normalisation property of the fundamental
 classes constructed above. In the assumptions of the proposition,
 we consider the deformation diagram \eqref{eq:deformation}:
$$
\xymatrix@=10pt@C=24pt{
S\ar[r]\ar_i[d] & \AA^1_S\ar^\nu[d] & S\ar^s[d]\ar[l] \\
X\ar^-{d_1}[r] & D & N_SX.\ar_/2pt/{d_0}[l]
}
$$
It induces pullback morphisms on Borel-Moore homology:
$$
\biv \E {**}(i:S \rightarrow X) 
 \xleftarrow{\ d_1^*\ } \biv \E {**}(\nu:\AA^1_S \rightarrow D)
 \xrightarrow{\ d_0^*\ } \biv \E {**}(s:S \rightarrow N_SX) \simeq \E^{**}(\Th(N_SX))
$$
which are isomorphisms according to Theorem \ref{thm:purity_MV}.

It rightly follows from the above construction
 that one has the relation:
$$
d_0^*(d_1^*)^{-1}(\fdl_i)=\rthom(N_SX).
$$
This relation implies that the system of fundamental classes of 
 the preceding proposition,
 restricted to zero sections of vector bundles,
 coincides with that of Proposition
 \ref{prop:fdl_thom}: indeed, when $X/S$ is a vector bundle,
 one obtains that $D$ is isomorphic to $\AA^1_X$, 
 $N_SX$ is isomorphic to $X$ and the maps $d_0$ and $d_1$ corresponds
 respectively to the zero and unit sections of $\AA^1_X$ through
 these isomorphisms.
\end{enumerate}
\end{rem}

\begin{ex}\label{ex:orientations&orientations}
Fixing a base scheme $S$,
 one can interpret the global orientation
 $c=c_S \in \E^{2,1}(\PP^\infty_S)$
 as a sequence of classes $c_n \in \E^{2,1}(\PP^n_S)$, $n>0$.
 Let us consider the immersion:
 $\PP^{n-1}_S \xrightarrow{\nu^{n-1}_S} \PP^n_S$
 of the hyperplane at infinity (say $\PP^{n-1}_S \times \{\infty\}$).
 Then the canonical exact sequence:
$$
0 \rightarrow \E^{2,1}_{\PP^{n-1}_S}(\PP^n_S)
 \rightarrow \E^{2,1}(\PP^n_S)
 \xrightarrow{j^*} \E^{2,1}(\PP^n_S-\PP^{n-1}_S)\simeq \E^{2,1}(S)
 \rightarrow 0
$$
is split exact and the class $c_n$ uniquely lifts to a class
 $\bar c_n$ in 
$$
\E^{2,1}_{\PP^{n-1}_S}(\PP^n_S) \simeq
 \biv \E {-2,-1}(\PP^{n-1}_S \xrightarrow{\nu^{n-1}_S} \PP^n_S);
$$
So the family $(c_n)_{n>0}$ corresponds to a family of orientations
 for the closed immersions $\nu^{n-1}_S$.

According to the previous remark, we get the equality:
$$
\fdl_{\nu^{n-1}_S}=c_1^{\PP^{n -1}}(\cO_{\PP^n}(-1))=\bar c_n,
$$
using the notations of Definition \ref{df:Chern_support} in the
 the middle term.
 In fact one can interpret an orientation of the ring spectrum
 $\E_S$ as a family of orientations of the immersions $\nu^\infty_n$
 satisfying suitable conditions.
\end{ex}

\begin{num}\label{num:basic_fdl_smooth}
Let us recall how the homotopy purity theorem of Morel and Voevodsky
 (stated above as Theorem \ref{thm:purity_MV})
 is used according to the method of Ayoub to define the relative purity
 isomorphism of the six functors formalism.

Consider now an arbitrary smooth s-morphism $f:X \rightarrow S$.
 We look at the following diagram:
$$
\xymatrix@=28pt@R=32pt{
X\ar@{=}[rd]\ar^-\delta[r]
& X\times_S X\ar^-{f_2}[r]\ar_/-3pt/{f_1}[d]\ar@{}|\Delta[rd] & X\ar^f[d] \\
 & X\ar_f[r] & S \\ 
}
$$
where the square $\Delta$ is cartesian and $\delta$ is the diagonal immersion.
Let $T_f$ be the tangent bundle of $X/S$, that is the normal bundle
 of the immersion $\delta$.
 Interpreting the construction of Ayoub (see also \cite[2.4.39]{CD3}),
 we then introduce the following natural transformation:
\begin{equation}\label{eq:f_purity_iso_smooth1}
\pur_f:f^* \simeq \delta^!f_2^!f^* \xrightarrow{\ Ex^{!*}(\Delta)\ }
 \delta^!f_1^*f^!
 \xrightarrow{\ \pur'_{f_1,\delta}\ } \MTh_X(-T_f) \otimes f^!.
\end{equation}
By adjunction, one gets a natural transformation:
\begin{equation}\label{eq:f_purity_iso_smooth2}
\pur'_f:f_\sharp \rightarrow f_!\big(\MTh_X(T_f) \otimes -\big).
\end{equation}
Then one deduces from the axioms of motivic triangulated
 categories that $\pur_f$ and $\pur'_f$ are isomorphisms,
 simply called the \emph{purity isomorphisms associated with
 $f$}.\footnote{This
 is one of the main results of \cite{Ayo1}, though it was proved there
 only in the quasi-projective case. The extension to the general case was
 first made in \cite[2.4.26]{CD3}.} Besides, we will use the following
 functoriality result satisfied by these purity isomorphisms.
\end{num}
\begin{prop}[Ayoub] \label{thm:ayoub}
Consider smooth s-morphisms
 $Y \xrightarrow g X \xrightarrow f S$ with respective
 tangent bundles $T_g$ and $T_f$. Then
 the following diagram of natural transformations is commutative:
$$
\xymatrix@R=12pt@C=90pt{
\MTh_Y(T_g) \otimes g^*\big(\MTh_X(T_f) \otimes f^*\big)\ar^-{\pur_g.\pur_f}[r]\ar_\sim[d]
 & g^! \circ f^!\ar@{=}[ddd] \\
\MTh_Y(T_g) \otimes \MTh_X(g^{-1}T_f) \otimes g^*f^*\ar_{\epsilon_\sigma}[d] & \\
\MTh_Y(T_{fg}) \otimes g^*f^*\ar@{=}[d] & \\
\MTh_Y(T_{fg}) \otimes (fg)^*\ar^-{\pur_{fg}}[r] & (fg)^!
}
$$
where the first isomorphism comes from the fact that $g^*$ is monoidal and
 Thom spaces are compatible with base change while
 the isomorphism $\epsilon_\sigma$ stands for \eqref{eq:Thom_virtual}
 associated with the exact sequence of vector bundles:
\begin{equation}\tag{$\sigma$}
0 \rightarrow T_g \rightarrow T_{fg} \rightarrow g^{-1}(T_f) \rightarrow 0.
\end{equation}
\end{prop}
For the proof, we refer the reader to \cite[1.7.3]{Ayo1} in the quasi-projective case
 -- actually Ayoub proves the assertion for the right adjoints but this is obviously 
 equivalent to our statement. Then the general case is reduced to the quasi-projective
 one using the localization property of motivic triangulated categories.
 
\begin{num}\label{num:fdl_smooth}
It is now easy to deduce  from the preceding results
 canonical new orientations for our bivariant theories using the method
 of Paragraph \ref{num:fdl_retraction}.
 
Let us fix again an absolute oriented ring $\T$-spectrum $(\E,c)$.
 Given a smooth morphism $f:X \rightarrow S$ of relative dimension $d$
 (seen as a locally constant function on $X$), we define the following
 isomorphism:
$$
\tfdl_f:\E_X(d)[2d] \xrightarrow{(\pur'_{T_f})^{-1}} \MTh(T_f) \otimes \E_X
 \simeq \MTh(T_f) \otimes f^*(\E_S)
 \xrightarrow{(\pur_f)^{-1}} f^!(\E_S)
$$
where $\pur'_{T_f}$ is the Thom isomorphism \eqref{eq:thom_iso_virtual}
 associated
 with the tangent bundle $T_f$ of $f$ and $\pur_f$ is Ayoub's purity isomorphism
 \eqref{eq:f_purity_iso_smooth1}. 

Following Remark \ref{rem:strong_orientation}(2), we then define the following orientation
 of $f$:
$$
\fdl_f:\un_S(d)[2d] \rightarrow \E_S(d)[2d] \xrightarrow{\tfdl_f} f^!(\E_X).
$$
Combining this construction together with the preceding proposition, we have obtained:
\end{num}
\begin{prop}\label{prop:fdl_smooth}
The universally strong orientations $\fdl_f$ constructed above
 form a system of fundamental classes (Definition \ref{df:fdl_classes_abstract})
 with coefficients in $\E$ for smooth s-morphisms.
\end{prop}
 
\begin{rem}\label{rem:fdl_smooth}
\begin{enumerate}
\item As the relative dimension $d$ of a smooth s-morphism $f:Y \rightarrow X$ 
 is equal to the rank of its tangent bundle, the fundamental class
 $\fdl_f$ has degree $d$ in the sense of Definition \ref{df:fdl_classes_abstract}.
\item It obviously follows from the constructions of Paragraphs
 \ref{num:basic_fdl_etale} and \ref{num:basic_fdl_smooth}
 that the orientations constructed here for arbitrary smooth
 s-morphisms extend the definition given in Example \ref{ex:fdl_etale}
 for \'etale s-morphisms.
\item For future reference, we will recall the following characterisation
 of the orientation $\fdl_f \in \biv \E {**}(X/S)$ constructed above.
 We have:
$$
\eta_f=\pur_f^{\prime*}(\rthom(-T_f))
$$
where the map $\pur_f^{\prime*}$ is induced by the isomorphism
 \eqref{eq:f_purity_iso_smooth2} as follows:
\begin{equation}\label{eq:f_purity_iso_smooth1_induced}
\E^{**}(\Th(-T_f))\simeq\Hom(f_\sharp(\MTh(-T_f)),\E_S)
 \xrightarrow{\pur_f^{\prime*}} \Hom(f_!(\un_X),\E_S)=\biv \E {**}(X/S)
\end{equation}
--- here Hom are understood with their natural $\ZZ^2$-graduation.
\end{enumerate}
\end{rem}

Let us finally note the following lemma for later use.
\begin{lm}\label{lm:fdl_sm_pullback}
Consider the notations of the previous proposition
 together with a cartesian square:
$$
\xymatrix@=10pt{
Y\ar^q[r]\ar_g[d]\ar@{}|\Theta[rd] & X\ar^f[d] \\
T\ar_p[r] & S
}
$$
such that $f$ is a smooth s-morphism.
Then the following relation holds in $\biv \E {**}(Y/T)$:
 $p^*(\fdl_f)=\fdl_g$.
\end{lm}
\begin{proof}
We reduce to prove the analogous fact for the isomorphism
 constructed in \ref{num:fdl_smooth}. Coming back to definitions
 (see in particular \eqref{eq:f_purity_iso_smooth1}
 and \eqref{eq:functorial_thom_iso_bis}),
 one reduces to prove the following diagram of natural transformations
 is commutative:
$$
\xymatrix@R=16pt@C=30pt{
q^*f^*\ar@{=}[dd]\ar^-\sim[r]\ar@{}_/-28pt/{(1)}[rdd]
 & q^*\uHom(M(X^2/X^2-\Delta_X),f^!)\ar[r]\ar^\sim[d]\ar@{}^/12pt/{(2)}[rdd]
 & q^*(\MTh(-T_f) \otimes f^!)\ar[r]\ar^\sim[d]\ar@{}^/12pt/{(3)}[rdd]
 & q^*f^!((d))\ar@{=}[d] \\
 & \uHom(q^*M(X^2/X^2-\Delta_X),q^*f^!)\ar^{\uHom(\phi,\psi)}[d]
 & q^*\MTh(-T_f) \otimes q^*f^!\ar^{\phi' \otimes \psi}[d]
 & q^*f^!((d))\ar^{\psi}[d] \\
g^*q^*\ar^-\sim[r] & \uHom(M(Y^2/Y^2-\Delta_Y),g^!p^*)\ar[r]
 & \MTh(-T_g) \otimes g^!p^*\ar[r]
 & g^!p^*((d))
}
$$
where:
\begin{itemize}
\item $d$ means the relative dimension of $f$ which we can assume
 to be constant, and which is equal to the relative dimension of $g$,
 and we have denoted $-((d))$ the twist $-(d)[2d]$.
\item $\psi$ stands for the reciprocal isomorphism
 of the exchange transformation: $Ex^{!*}:g^!p^* \rightarrow q^*f!$
 associated with the square $\Theta$
 --- which is an isomorphism as $f$ is smooth;
\item we have put $X^2=X \times_S X$ and
 $Y^2=Y \times_T Y=X \times_S X \times_S T$; $\phi$ is the
 isomorphism associated with the cartesian squares:
$$
\xymatrix@R=16pt@C=26pt{
Y\ar^{\delta_Y}[r]\ar_q[d]\ar@{}[rd]
 & Y^2\ar^{q \times q}[d]\ar^{g_1}[r] & Y\ar^q[d]\\
X\ar_{\delta_X}[r] & X^2\ar_{f_1}[r] & X
}
$$
where $f_1$ (resp. $g_1$) stands for the projection on the first factor.
\item $\phi'$ is the isomorphism induced by the identification
 $q^{-1}(T_f) \simeq T_g$ --- which also expresses that the square $\Theta$
 is transversal.
\end{itemize}
Then, diagram (1) is commutative as it is made of exchange transformations,
 diagram (2) is commutative has the deformation diagram \eqref{eq:deformation}
 is functorial with respect to cartesian morphisms of closed pairs
 --- applied to the cartesian square (*) --- and diagram (3) is commutative
 as the Thom class \eqref{eq:thom_class} is stable under pullbacks.
\end{proof}

\subsection{The regular closed immersion case according to Navarro}

\begin{num}
Let us now recall the construction of Navarro of a system of
 fundamental classes for regular closed immersions which extends
 the one constructed in Proposition \ref{prop:fdl_retraction}.
 The construction can be safely transported to our generalized
 context as we assume the existence of a premotivic
 adjunction $\tau ^*:\SH \rightarrow \T$.

Given a closed pair $(X,Z)$, $U=X-Z$ the open complement,
 we recall that the relative Picard group $\Pic(X,Z)$
 is the group of isomorphisms classes of
 pairs $(L,u)$ where $L$ is a line bundle on $X$
 and $u:L|_U \rightarrow \AA^1_U$ a trivialization of $L$ over $U$.

As remarked by Navarro (\cite[Rem. 3.8]{Nav}), one deduces from
 a classical result of Morel and Voevodsky that 
 there is a natural bijection:
$$
\epsilon_{X,Z}:\Pic(X,Z) \xrightarrow{\sim} [X/U,\PP^\infty_X]
$$
where the right hand-side stands for the unstable $\AA^1$-homotopy classes
 of pointed maps over $X$. Thus, one obtains Chern classes with support
 as in \cite[1.4]{Nav}.
\end{num}
\begin{df}\label{df:Chern_support}
Let $(\E,c)$ be an absolute oriented ring $\T$-spectrum.

Then, given any closed pair $(X,Z)$ with open complement $U$,
 one defines the first Chern class map with support and coefficients
 in $\E$ as the following composite:
\begin{align*}
c_1^Z:\Pic(X,Z) \xrightarrow{\epsilon_{X,Z}} [X/U,\PP^\infty_X]
 & \rightarrow \Hom_{\SH(X)}(\Sus X/U,\Sus \PP^\infty_S) \\
 & \xrightarrow{\tau^*}\Hom_{\T(X)}\big(M(X/X-Z),M(\PP^\infty_X)\big) \\
 & \xrightarrow{(c_X)_*}\Hom_{\T(X)}(M(X/X-Z),\E_X(1)[2])=\E^{2,1}_Z(X).
\end{align*}
\end{df}
These Chern classes satisfy the following properties
 (see \cite[1.39, 1.40]{Nav}):
\begin{prop}\label{prop:chern_supp}
Consider the notations of the above definition.
\begin{enumerate}
\item Given any cartesian morphism $(Y,T) \rightarrow (X,Z)$ of
 closed pairs and any element $(L,u) \in \Pic(X,Z)$, one has:
 $f^*c_1^Z(L,u)=c_1^T\big(f^{-1}(L),f^{-1}(u)\big)$.
\item For any $(L,u) \in \Pic(X,Z)$, the cohomology class
 $c_1^Z(L)$ is nilpotent in the ring $\E^{**}_Z(X)$.
\item Let $F_X$ be the formal group law associated with
 the orientation $c_X$ of $\E_X$. Then, given any
 classes $(L_1,u_1)$, $(L_2,u_2)$ in $\Pic(X,Z)$, one has the
 following relation in $\E^{2,1}_Z(X)$:
$$
c_1^Z(L_1 \otimes L_2,u_1 \otimes u_2)
 =F_X\big(c_1^Z(L_1,u_1),c_1^Z(L_2,u_2)\big),
$$
using the $\E^{**}(X)$-module structure on $\E^{**}_Z(X)$
 and point (2).
\end{enumerate}
\end{prop}

The second tool needed in the method of Navarro is the following
 version of the blow-up formula for oriented theories
 (see \cite[2.6]{Nav}).
\begin{prop}
Consider an absolute oriented ring $\T$-spectrum $(\E,c)$.
 Let $(X,Z)$ be a closed pair of codimension $n$,
 $B$ be the blow up of $X$ in $Z$
 and consider the following cartesian square:
$$
\xymatrix@=14pt{
P\ar^k[r]\ar_q[d] & B\ar^p[d] \\
Z\ar^i[r] & X.
}
$$
Then the following sequence is split exact:
$$
0 \rightarrow \E^{**}_Z(X)
 \xrightarrow{p^*} \E^{**}_P(B)
 \xrightarrow{\overline{k^*}} \E^{**}(P)/\E^{**}(X) \rightarrow 0.
$$
Letting $\cO_P(-1)$ (resp. $\cO_B(-1)$) be the canonical
 line bundle over $P$ (resp. $B$), and putting $c=c_1(\cO(-1))$
 (resp. $b=c_1^P(\cO_B(-1))$), a section is given by
 the following $\E^{**}(X)$-linear morphism:
$$
s:\E^{**}(P)/\E^{**}(X) \simeq \big(\oplus_{i=1}^{n-1}
 \E^{**}(X).c^i\big) \rightarrow \E^{**}_P(B), c^i \mapsto b^i.
$$
\end{prop}

\begin{num}\label{num:fdl_regular}
Consider again the assumption and notations
 of the preceding proposition.
 We can now explain the construction of Navarro
 \cite[2.7]{Nav}.

One defines a canonical class in $\E^{2n,n}_P(B)$
 (where $n$ is seen as a locally constant function on $P$)
 as follows:
$$
\fdl'_i:=-\left(\sum_{i=0}^{n-1} q^*(c_i(N_ZX)).(-c)^{n-i}\right).b
$$
where we have used the $\E^{**}(P)$-module structure
 on $\E^{**}_P(B)$. Then, according to the projective bundle theorem
 for $P=\PP(N_ZX)$, one deduces that $k^*(\fdl'_i)=c_n(N_ZX)$
 so that this class is zero in the quotient $\E^{**}(P)/\E^{**}(X)$.
 Therefore, according to the preceding proposition, there exists
 a unique class $\fdl_i \in \E^{2n,n}_Z(X)$ such that:
$$
p^*\fdl_i=\fdl'_i.
$$
This will be called the \emph{orientation} of $i$
 associated with the orientation $c$ of the ring spectrum $\E$.

According to Navarro's work, we get the following result.
\end{num}
\begin{prop}\label{prop:fdl_regular}
The orientations $\fdl_i$ constructed above
 form a system of fundamental classes
 with coefficients in $\E$ for regular closed immersions.

Besides, this system coincides with that of
 Proposition \ref{prop:fdl_retraction} when restricted
 to the closed immersions which admit a smooth retraction. 
\end{prop}
To prove the first assertion, we need to prove the
 associativity formula; this is \cite[Th. 2.14]{Nav}.
 The second assertion follows from the compatibility
 of $\fdl_i$ with respect to base change along
 transversal squares (\cite[2.12]{Nav})
 and from the fact the orientation defined above
 coincides with the refined Thom class when $i$
 is the zero section of a vector bundle (\cite[2.19]{Nav}).

\begin{ex}\label{ex:Navarro}
Consider the notations of the previous proposition.
\begin{enumerate}
\item Let $i:D \rightarrow X$ be the immersion of
 a regular divisor in a scheme $X$.
 We let $\mathcal O(-D)$ be the line bundle on $X$ corresponding to
 the inverse of the ideal sheaf of $Z$ in $X$
 (following the convention of \cite[21.2.8.1]{EGA4}).
 The sheaf $\mathcal O(-D)$ has support in $D$ and admits a canonical
 trivialization $s$ over $X-D$. As blowing-up a divisor
 does not do anything, it follows from the previous construction
 that we have the relation:
$$
\fdl_i=c_1^D\big(\mathcal O(-D),s\big).
$$
\item Let $i:Z \rightarrow X$ be a regular closed immersion
 of codimension $n$. 
 Recall we have a base change map:
$$
i^*:\E^{**}_Z(X)=\biv \E {**} (Z/X)
 \rightarrow \biv \E {**} (Z/Z)=\E^{**}(Z),
$$
--- equivalently, the map forgetting the support.

Then, from the construction of Paragraph \ref{num:fdl_regular},
 we deduce the relation:
$$
i^*(\fdl_i)=c_n(N_ZX).
$$
\end{enumerate}
\end{ex}

\subsection{The global complete intersection case}

We are finally ready for the main theorem of this work.
 We have so far constructed several systems of fundamental
 classes and we now  show how to glue them. The main result
 in order to do so is the following lemma.
\begin{lm}\label{lm:key}
Consider an absolute oriented ring spectrum $(\E,c)$.
 Let $f:X \rightarrow S$ be a smooth s-morphism and $s:S \rightarrow X$
 a section of $f$.

Then using the notations of Propositions \ref{prop:fdl_retraction}
 and \ref{prop:fdl_smooth}, the following relation holds in $\E^{00}(X)$:
$$
\fdl_s.\fdl_f=1. 
$$
\end{lm}
\begin{proof}
 Let $V=N_S(X)$ be the normal bundle of $S$ in $X$.
 By construction of the deformation space,
 we get a commutative diagram made of cartesian squares:
$$
\xymatrix@=14pt{
S\ar_{s_1}[d]\ar^s[r] & X\ar^{d_1}[d]\ar^f[r] & S\ar^{s_1}[d] \\
\AA^1_S\ar[r] & D_S(X)\ar^-{\tilde f}[r] & \AA^1_S \\
S\ar^-{s_0}[u]\ar^-\sigma[r] & V\ar_-{d_0}[u]\ar^-p[r] & S\ar_-{s_0}[u]
}
$$
where the two left columns are made by the deformation diagram
 \eqref{eq:deformation} associated with the closed pair $(X,Z)$,
 the morphism $p$ is the canonical projection of $V/S$
 and $\tilde f$ is given by the composite map
 $D_S(X) \rightarrow D_S(S)\simeq \AA^1_S$.
 An easy check shows that $\tilde f$ is smooth.

According to $\AA^1$-homotopy invariance, the pullbacks
 $s_0^*,s_1^*:\E^{**}(\AA^1_S) \rightarrow \E^{**}(S)$ are
 the same isomorphism.
 The orientations of the form $\eta_f$ for $f$ smooth are stable
 under pullbacks
 (Lemma \ref{lm:fdl_sm_pullback}) so that applying Remark
 \ref{rem:fdl_retraction&thom}, we are reduced to prove the relation
$$
\fdl_\sigma.\fdl_p=1.
$$
In other words, we can assume $X=V$ is a vector bundle over $S$,
 $s=\sigma$ is its zero section and $f=p$ its canonical projection.
 We have seen in Remark \ref{rem:fdl_retraction&thom} that $\eta_\sigma$
 coincides with
 the refined Thom class of $V/S$, via the canonical isomorphism
 \eqref{eq:compute_Thom_coh}.
 Similarly, from Remark \ref{rem:fdl_smooth}(3),
 the orientation $\fdl_f$ is induced by
 the Thom class of the $S$-vector bundle $T_p=p^{-1}E$ via
 the isomorphism \eqref{eq:f_purity_iso_smooth1_induced}.
 As we obviously have: $\rthom(V) \refp \rthom(-V)=1$,
 we are reduced to prove the following lemma:
\begin{lm}
Given a vector bundle $V/S$ with zero section $\sigma$
 and canonical projection $p$, the following diagram is commutative:
$$
\xymatrix@R=4pt@C=60pt{
\biv \E {**}(S/V) \otimes \biv \E {**}(V/S)\ar^-{\bar \mu}[rd]
\ar_{\alpha^*_\sigma \otimes \pur_p^{\prime*}}[dd] &  \\
& \E^{**}(S) \\
\E^{**}(\MTh(V)) \otimes \E^{**}(\MTh(-V))\ar_-\mu[ru] & 
}
$$
where $\mu$ is the product on cohomology, $\bar \mu$ is the
 product of bivariant theory,
 while $\alpha_\sigma^*$ (resp. $\pur_p^{\prime*}$) is induced
 by the isomorphism $\alpha_\sigma:\sigma_! \rightarrow \sigma_*$
 as $\sigma$ is proper (resp. the purity isomorphism
 \eqref{eq:f_purity_iso_smooth1}).
\end{lm}
In the proof of this lemma, we will restrict ourselves to classes of degree
 $(0,0)$. It will healthily simplify notations and
 the proof for other degrees is the same. Let us thus consider the following
 maps, corresponding to Borel-Moore homology classes:
$$
y:\sigma_!(\un_S) \rightarrow \E_V, 
v:p_!(\un_V) \rightarrow \E_S.
$$
The commutativity of the preceding diagram then amounts to prove
 the following relation:
$$
\bar \mu(y \otimes v)
 =\mu\big(\alpha_\sigma^*(y) \otimes \pur_p^{\prime*}(v)\big),
$$
so we put: $\tilde y=\alpha_\sigma^*(y)$ and $\tilde v=\pur_p^{\prime*}(v)$.
 Then the preceding relation can be translated into the commutativity
 of the following diagram of isomorphisms:
$$
\xymatrix@R=10pt@C=12pt{
&&& p_!\sigma_!(\un_S)\ar^{p_!(y)}[r]\ar@{}|{(2)}[rrrdddd]
 & p_!(\E_V)\ar^-\sim[r] & \E_S \otimes p_!(\un_S)\ar^-{1 \otimes v}[r]
 & \E_S \otimes \E_S\ar@{=}[dddd]\ar^{\mu_\E}[rdd] & \\
 &&& p_\sharp\big(\sigma_!(\un_S) \otimes \MTh(-T_p)\big)
      \ar^{Ex_\sharp^\otimes}[d]\ar_{\pur'_p}[u] \\
\un_S\ar@/^13pt/^-{\epsilon_{p,\sigma}}[rrruu]\ar@/_10pt/_-{can}[rrrdd]\ar@{}|/-13pt/{(1)}[rrr]
&&& p_\sharp \sigma_!(\un_S) \otimes \MTh(-V)\ar^{\alpha_\sigma}[d]
 &&&& E_S \\
 &&& p_\sharp \sigma_*(\un_S) \otimes \MTh(-V)\ar@{=}[d] \\
 &&& \MTh(V) \otimes \MTh(-V)\ar^{\tilde y \otimes \tilde v}[rrr]
 &&& \E_S \otimes \E_S,\ar_{\mu_\E}[ruu]
}
$$
where:
\begin{itemize}
\item the morphism labelled $\epsilon_{p,\sigma}$ stands for the inverse of the
 canonical isomorphism $Id=(p \circ \sigma)_! \rightarrow p_!\sigma_!$
 coming from the fact $f \mapsto f_!$ corresponds to a $2$-functor;
\item the map labelled $Ex_\sharp^\otimes$ is the exchange
 isomorphism corresponding to the projection formula for $f_\sharp$
 (cf. \cite[\textsection 1.1.24]{CD3}) ---
 using the identification $\MTh(-T_p)=p^*\MTh(-V)$;
\item the map labelled $can$ follows from the definition of the Thom space
 of the virtual bundle $(-V)$.
\end{itemize}
 The commutativity of part (2) follows directly from the definition
 of $\tilde y$ and $\tilde v$.
 Thus we only need to show the commutativity of part (1) of the above
 diagram.

After taking tensor product with $\MTh(V)$, the diagram (1) can be simplified as follows:
$$
\xymatrix@R=16pt@C=44pt{
& p_!\sigma_!(\un_S) \otimes \MTh(V)\ar^{Ex_!^\otimes}[r]
 & p_!(\sigma_!(\un) \otimes p^*\MTh(V)) \\
\MTh(V)\ar^-{\epsilon_{p,\sigma}}[ru]\ar@{=}[rrd]\ar@{}|/38pt/{(1')}[rr]
 & & p_\sharp \sigma_!(\un_S)
  \ar_{\pur'_p}[u]\ar^{\alpha_\sigma}[d] \\
& & p_\sharp \sigma_*(\un_S),
}
$$
where the arrow $Ex_!^\otimes$ stands for the exchange isomorphism
 of the projection formula for $p_!$ (see \cite[2.2.12]{CD3}).

Let us first summarize the geometric situation
 in the following commutative diagram of schemes:
$$
\xymatrix@R=26pt@C=34pt{
S\ar^\sigma[r]
 & V\ar@/^10pt/^-\delta[r]\ar_p[d]\ar|{\sigma'}[r]\ar@{}|\Theta[rd]
 & W\ar^-{p'}[r]\ar^{p''}[d]\ar@{}|\Delta[rd]
 & V\ar^p[d] \\
& S\ar_\sigma[r] & V\ar_p[r] & S,
}
$$
where each square is cartesian and $\delta$ denotes the obvious
 diagonal embedding. The map $p'':W=V \times_S V \rightarrow V$
 is the projection on the second factor and in particular,
 we get: $\sigma'(v)=(v,p(v))$.
 Note finally that $\sigma$ is an equalizer of $(\delta,\sigma')$. \\
Then coming back to the definition of 
 the purity isomorphism $\pur'_p$
 (cf. Paragraph \ref{num:basic_fdl_smooth}),
 diagram $(1')$ can be divided as follows:
$$
\xymatrix@R=8pt@C=4pt{
&& p_!\sigma_!(\un) \otimes \MTh(V)\ar^-{Ex_!^\otimes}[rr]
 && p_!(\sigma_!(\un) \otimes p^*\MTh(V)) \\
&&& p_!p'_\sharp\sigma'_!\sigma_!(\un)\ar[ru] & \\
\MTh(V)\ar^-{\epsilon_{p,\sigma}}[rruu]\ar@{=}[dddd] &&&& \\
& p_\sharp p''_!\sigma'_!\sigma_!(\un)\ar^{Ex_{\sharp!}}[rruu]
  \ar@{}|/14pt/{(1'')}[rr]
 && p_!p'_\sharp\delta_!\sigma_!(\un)\ar_{\pur_\delta}[uu] & \\
 &&&& \\
&&& p_\sharp p''_!\delta_!\sigma_!(\un)
      \ar_{Ex_{\sharp!}}[uu]\ar^{\epsilon}[lluu] & \\
p_\sharp \sigma_*(\un) &&&& p_\sharp \sigma_!(\un_S),
  \ar_{\pur'_p}[uuuuuu]\ar^{\alpha_\sigma}[llll]
	\ar_{\epsilon_{\delta,\sigma}}[lu]
}
$$
where $\pur_\delta$ is (induced by) the purity isomorphism
 associated with the closed immersion $\delta$
 (Definition \ref{df:purity_closed_pairs}), the
 map labelled $Ex_{\sharp!}$ stands for the obvious exchange isomorphism
 associated with the cartesian square $\Delta$, 
 and $\epsilon$ is the isomorphism coming from the fact
 $f \mapsto f_!$ corresponds to a $2$-functor and the relation
 $\delta \circ \sigma=\sigma' \circ \sigma$.

In this diagram, the commutativity of the right hand side
 follows by definition of $\pur'_p$ and the commutativity of
 the left hand side follows from the definition of the
 involved exchange isomorphisms and the fact $\alpha_f:f_! \rightarrow f_*$
 corresponds to a morphism of $2$-functors.
 The geometry is hidden in the commutativity
 of the diagram labelled $(1'')$.

We can divide again $(1'')$ as follows:
$$
\xymatrix@R=26pt@C=34pt{
p_\sharp p''_!\sigma'_!\sigma_!(\un)\ar^-{Ex_{\sharp!}}[r]
 & p_!p'_\sharp\sigma'_!\sigma_!(\un)\ar@{=}[r]
  \ar@{}|{(1''')}[rd]
 & p_!p'_\sharp\sigma'_!\sigma_!(\un) \\
p_\sharp p''_!\delta_!\sigma_!(\un)\ar^-{Ex_{\sharp!}}[r]\ar^{\epsilon}[u]
 & p_!p'_\sharp\delta_!\sigma_!(\un)\ar_{\epsilon}[u]\ar@{=}[r]
 & p_!p'_\sharp\delta_!\sigma_!(\un)\ar_{\pur_\delta}[u]
}
$$
The commutativity of the left hand square follows from the
 naturality of the exchange isomorphism,
 it remains part $(1''')$. We can certainly erase the functor $p_!$
 in each edge of this diagram. Then the right-most vertical map
 can be expressed as follows:
$$
M(W/W-\delta(V)) \otimes \sigma_!(\un)
 \xrightarrow{\pur_\delta} \MTh(N_\delta/V) \otimes \sigma_!(\un)
$$
and the commutativity of diagram $(1''')$ means that this map
 is the identity. Using once again the projection
 formula for $\sigma$, this map
 can be expressed as follows:
$$
\sigma_!\sigma^*\big(M(W/W-\delta(V))\big)
 \xrightarrow{\ \sigma_!\sigma^*(\pur_\delta)\ }
 \sigma_!\sigma^*\MTh(N_\delta/V).
$$
So we are reduced to show that $\sigma^*(\pur_\delta)$
 is the identity map through the obvious identifications:
$$
\sigma^*M(W/W-\delta(V))=M(V/V-S)=\sigma^*\MTh(N_\delta/V).
$$
This is an easy geometric fact:
 let us consider the deformation
 diagram for the closed immersion $\delta:V \rightarrow W$:
$$
\xymatrix@=10pt{
V\ar[r]\ar[d] & \AA^1_V\ar[d] & V\ar[d]\ar[l] \\
W\ar[r] & D_\delta & N_\delta.\ar[l]
}
$$
Note that the closed immersion $\delta$ and $\sigma'$ are transversal
 --- \emph{i.e.} the square $\Theta$ is transversal. In other words,
 $\sigma^*(N_\delta)=N_\sigma=V$ as a vector bundle over $S$.
 In particular, the pullback of the preceding diagram of $V$-schemes
 along the immersion $\sigma:S \rightarrow V$ is the following one:
$$
\xymatrix@=10pt{
S\ar[r]\ar[d] & \AA^1_S\ar[d] & S\ar[d]\ar[l] \\
V\ar[r] & \AA^1_V & V.\ar[l]
}
$$ 
where we have used the identifications 
 $\sigma^*(D_\delta)=D_\sigma=\AA^1_V$; the last identification
 is justified by the fact $\sigma$ is the zero section of a vector
 bundle.
 In this diagram, the vertical maps are the unit and zero sections
 of the affine lines involved. Therefore, by homotopy invariance,
 we get that $\sigma^*(\pur_\delta)$ is identified to the identity map
 as required.
\end{proof}

We are now ready to state our main theorem.
\begin{thm} \label{thm:BM_fundamental_classes}
Let $(\E,c)$ be an absolute oriented ring $\T$-spectrum.

There exists a unique system of fundamental classes
 $\fdl_f \in \biv \E {**}(X/S)$ for gci morphisms $f$
 with coefficients in $\E$ such that,
 in addition to the associativity property, one has:
\begin{enumerate}
\item If $f$ is a smooth morphism,
 $\fdl_f$ coincides with the fundamental class
 defined in Proposition \ref{prop:fdl_smooth}.
\item If $i:Z \rightarrow X$ is a regular closed immersion,
 $\fdl_i$ coincides with the fundamental class 
 defined in \ref{prop:fdl_regular}.
\end{enumerate}
If $d$ is the relative dimension of $f$, seen as a Zariski
 local function on $X$, the class $\eta_f$ has dimension $d$
 (Definition \ref{df:fdl_orientation}).
\end{thm}
\begin{proof}
Because any gci morphism $p:X \rightarrow S$ admits
 a factorization
$$
X \xrightarrow i P \xrightarrow f S
$$
where $f$ is smooth and $i$
 is a regular closed immersion, we have to prove
 that the class $\fdl_i \refp \fdl_f$ is independent of the factorization.

To prove this, we are reduced by usual arguments 
 (see for example \cite[proof of 5.11]{Deg8})
 and the help of Lemma \ref{lm:fdl_sm_pullback} to show
 the associativity property: 
$$
\fdl_g \refp \fdl_f=\fdl_{fg}
$$
in the following three cases:
\begin{enumerate}
\item[(a)] $f$ and $g$ are smooth morphisms;
\item[(b)] $f$ and $g$ are regular closed immersions;
\item[(c)] $g$ is a smooth morphism
 and $f$ is a section of $g$.
\end{enumerate}
Case (a) follows from Proposition \ref{prop:fdl_smooth},
 case (b) from Proposition \ref{prop:fdl_regular}
 and case (c) from the preceding lemma.
 Then the associativity formula in the general case follows
 using standard arguments (see for example \cite[proof of 5.14]{Deg8})
 from (a), (b) and (c).

The last assertion follows as the degree of $\fdl_i$
 (resp. $\fdl_p$) is the opposite of the codimension $n$ of $i$
 (resp. the dimension $r$ of $p$) and we have: $d=r-n$.
\end{proof}

\begin{rem}\label{rem:formal_duality}
When $f$ is a smooth morphism, it follows from
 Proposition \ref{prop:fdl_smooth} that $\fdl_f$ is universally
 strong (Definition \ref{df:strong_orientation}).
When $f=s$ is the section of a smooth s-morphism,
 property (2) in the above theorem and the last assertion of Proposition
 \ref{prop:fdl_regular} shows that the class $\fdl_s$ constructed
 in the above theorem coincides with that of Proposition
 \ref{prop:fdl_retraction}. Therefore $\fdl_s$ is universally strong
  according to \emph{loc. cit.}
This remark will be amplified in Section \ref{sec:duality}.
\end{rem}

\begin{df}\label{df:fdl_classes}
Given the assumptions of the previous proposition,
 for any gci morphism $f:X \rightarrow S$,
 we call $\fdl_f$ the \emph{fundamental class} of $f$ 
 associated with the orientation $c$ of the absolute ring spectrum
 $\E$.

In case $f=i:Z \rightarrow X$ is a regular closed immersion,
 we will also use the notation:
$$
\fdl_X(Z):=\fdl_i
$$
seen as an element of $\E^{2c,c}_Z(X)$ where $c$ is 
 the codimension of $i$ --- as a locally constant function
 on $Z$.
\end{df}
This fundamental class only depends upon the choice of the (global)
 orientation $c$ of $\E$. If there is a possible confusion about
 the chosen orientation, we write $\fdl_f^c$ instead of $\fdl_f$.

We often find another notion of fundamental class in the literature
 that we introduce now for completeness.
\begin{df}\label{df:fdl_coh}
Let $(\E,c)$ be an absolute oriented ring spectrum
 and $f:Y \rightarrow X$ be a proper gci morphism.
 We define the fundamental class of $f$ in $\E$-cohomology,
 denoted by $\cfdl_f$,
 as the image of $\fdl_f$ by the map:
$$
\biv \E {**}(Y/X) \xrightarrow{\ f_!\ } \biv \E {**}(X/X)=\E^{**}(X).
$$
\end{df}
We will give more details on these classes in section \ref{sec:Gysin}.

\subsection{The quasi-projective lci case}

We end this section by presenting an alternative method
 to build fundamental classes, when restricting to quasi-projective
 lci morphism. This is based on the following uniqueness result.
\begin{thm}\label{thm:uniqueness}
Let $(\E,c)$ be an absolute oriented $\T$-spectrum.

Then a family of orientations
 $\fdl_f \in \biv \E {**}(X/S)$ attached to quasi-projective
 lci morphisms $f:X \rightarrow S$ is uniquely characterized by
 the following properties:
\begin{enumerate}
\item If $j:U \rightarrow X$ is an open immersion,
 $\fdl_j$ is equal to the following composite:
$$
\un_X \xrightarrow \eta \E_X \simeq j^*(\E_U)=j^!(\E_U)
$$
where $\eta$ is the unit of the ring spectrum $\E_X$.
\item If $s$ is the zero section of a line bundle $L/S$,
 one has: $\fdl_s=c_1^Z(L)$
 (notation of Definition \ref{df:Chern_support}).
\item If $p$ is the projection map of $\PP^n_S/S$,
 $\fdl_p$ is given by the construction \eqref{eq:f_purity_iso_smooth1}.
\item Given any composable pair of morphism $(f,g)$,
 one has the relation $\fdl_f.\fdl_g=\fdl_{g \circ f}$
 whenever $f$ and $g$ are immersions, or $f$ is the projection
 of $\PP^n_S/S$ and $g$ is an immersion.
\item Let $i:Z \rightarrow X$ be a regular closed immersion
 and $f:X' \rightarrow X$ a morphism transversal to
 $i$. Put $k=f^{-1}(i)$. Then the following relation holds:
 $f^*\fdl_i=\fdl_k$.
\item Consider the blow-up square of a closed immersion $i$
 of codimension $n$,
$$
\xymatrix@=10pt{
E\ar^k[r]\ar[d] & B\ar^p[d] \\
Z\ar^i[r] & X
}
$$
the following relation holds:
$p^*\fdl_i=c_{n-1}(N_ZX).\fdl_k$.
\end{enumerate}
\end{thm}
\begin{proof}
Let us prove the uniqueness of $\fdl_f$ for a quasi-projective
 lci morphism $f$.
 The map $f$ admits a factorization $f=p j i$
where $p$ is the projection of $\PP^n_S$ for a suitable integer
 $n\geq 0$, $j$ is an open immersion and $i$ a closed immersion.
 According to property (4), we reduce the case of $f$ to
 that of $p$, $j$ or $i$.
The case of $j$ follows from (1), and that of $p$ follows from (3).

So we are reduced to the case of a regular closed immersion
 $i:Z \rightarrow X$.
 According to property (6) and its notations, we are reduced to the case of
 the regular closed immersion $k$.
 In other words, we can assume $i$ has codimension $1$.

Then $Z$ corresponds to an effective Cartier divisor in $X$,
 and therefore to a line bundle $L/X$ with a canonical section
 $s:X \rightarrow L$ such that the following diagram is
 cartesian:
$$
\xymatrix@=10pt{
Z\ar^i[r]\ar[d] & X\ar^s[d] \\
X\ar^{s_0}[r] & L
}
$$
where $s_0$ is the zero-section of the line bundle $L/X$.
 According to relation
 (5), we get: $s^*\fdl_{s_0}=\fdl_i$. This uniquely characterize
 $\fdl_i$ because $\fdl_{s_0}$ is prescribed by relation (2).
\end{proof}

\begin{num}
In fact, it is possible to show the existence of fundamental
 classes in the quasi-projective lci case by using the constructions
 in the preceding proof and the techniques of \cite{Deg8}
 (see more specifically \cite[Sec. 5]{Deg8}).
 This gives an alternate method where the construction of Ayoub
 \ref{num:basic_fdl_smooth} is avoided.

The interest of this method is that, instead of using the axiomatic
 of triangulated motivic categories, one can directly work
 with a given bivariant theory satisfying suitable axioms:
 in fact, the properties stated in Proposition
 \ref{prop:basic_bivariant}. Then one can recover the construction
 of a system of fundamental classes for quasi-projective lci
 morphisms and proves the properties that we will see in
 the forthcoming section.
\end{num}

%% file: intersection.tex
\subsection{Base change formula}

In all this section, we will fix once
 and for all an absolute oriented ring $\T$-spectrum $(\E,c)$.
 We first state the following extension of the classical excess
 intersection formula.
\begin{prop}[Excess of intersection formula]\label{prop:excess_generalized}
Consider a cartesian square
$$
\xymatrix@=14pt{
Y\ar^{g}[r]\ar_q[d]\ar@{}|\Delta[rd] & T\ar^p[d] \\
X\ar_f[r] & S
}
$$
of schemes such that $f$, $g$ are gci.
 Let $\tau_f \in K_0(X)$ (resp. $\tau_{g} \in K_0(Y)$)
 be the virtual tangent bundle of $f$ (resp. $g$).
 We put $\xi=p^*(\tau_f)-\tau_g$ as an element of
 $K_0(Y)$,\footnote{The element $\xi$ is called
 the \emph{excess intersection bundle} associated with the square $\Delta$;}
 and let $e(\xi)$ be the top Chern class of $\xi$ in $\E^{**}(Y)$.
 Then the following formula holds in $\biv \E {**} (Y/T)$:
$$
p^*(\fdl_f)=e(\xi) \refp \fdl_g.
$$
\end{prop}
In fact, we can consider a factorization of $f$ into a regular closed 
 immersion $i$ and a smooth morphism $p$. Because of property (3)
 of Theorem \ref{thm:BM_fundamental_classes}, we are reduced to the
 case $f=i$, regular closed immersion, or $f=p$, smooth
 morphism. The first case is \cite[Cor. 2.12]{Nav}
 while the second case was proved in Lemma \ref{lm:fdl_sm_pullback}.

\begin{ex}\label{ex:transversal}
Of course, an interesting case is obtained when the square  $\Delta$
 is transversal,
 \emph{i.e.} $\xi=0$: it shows, as expected,
 that fundamental classes are stable by pullback along
 transversal morphisms.\footnote{To be clear: a morphism of
 schemes $p:T \rightarrow S$ is transversal to a gci morphism
 $f:X \rightarrow S$ if $g=f \times_S T$ is gci and
 $p^*(\tau_f)=\tau_g$ as elements of $K_0(X \times_S T)$.}
\end{ex}

\begin{num}
Fundamental classes in the case of closed immersions
 give an incarnation of intersection theory. Let us consider
 the enlightening case of divisors.
 In fact one can extend slightly the notion of fundamental
 classes from effective Cartier divisors to
 that of \emph{pseudo-divisors} as defined by Fulton
 \cite[2.2.1]{Ful}
\end{num}
\begin{df}
Let $D=(\cL,Z,s)$ be a pseudo-divisor on a scheme $X$.
We define the fundamental class of $D$ in $X$ with coefficients
 in $(\E,c)$ as:
$$
\fdl_X(D):=c_1^Z(\cL,s) \in \E^{2,1}_Z(X).
$$
\end{df}
The following properties of these extended fundamental classes
 immediately follows from Proposition \ref{prop:chern_supp}.
\begin{prop} Let $X$ be a scheme.
\begin{enumerate}
\item For any pseudo-divisor $D$ on $X$ with support $Z$,
 the class $\fdl(D)$ is nilpotent in the ring $\E^{**}_Z(X)$.
\item Let $(D_1,...,D_r)$ be pseudo-divisors on $X$
 with support in a subscheme $Z \subset X$ and $(n_1,...,n_r) \in \ZZ^r$
 an $r$-uple/ One has the following relations in the ring $\E^{**}_Z(X)$:
$$
\fdl(n_1.D_1+\hdots+n_r.D_r)
=[n_1]_F.\fdl(D_1)+_F \hdots +_F [n_r]_F.\fdl(D_r)
$$
where $+_F$ (resp. $[n]_F$ for an integer $n \in \ZZ)$
 means the addition (resp. $n$-th self addition)
 for the formal group law with coefficients in $\E^{**}(X)$ associated 
 with the orientation $c$ (Proposition \ref{prop:chern_classes}).
\item Let $f:Y \rightarrow X$ be any morphism of schemes.
 Then for any pseudo-divisor $D$ with support $Z$, $T=f^{-1}(Z)$,
 on has in $\E^{**}_T(Y)$:
$$
f^*(\fdl_X(D))=\fdl_Y(f^*(D))
$$
where $f^*$ on the right hand side is the pullback of pseudo-divisors
 as defined in \cite[2.2.2]{Ful}.
\end{enumerate}
\end{prop}

In particular, it is worth to derive the following corollary
 which describes more precisely the pullback operation on fundamental
 classes associated with divisors.
\begin{cor}\label{cor:ramification}
Let $X$ be a normal scheme.
\begin{enumerate}
\item For any Cartier divisor $D$ on $X$, one has the relation
$$
\fdl_X(D)=[n_1]_F.\fdl_X(D_1)+_F \hdots +_F [n_r]_F.\fdl_X(D_r)
$$
in $\E^{2,1}_Z(X)$,where $Z$ in the support of $D$,
 $(D_i)_i$ the family of irreducible components of
 $Z$ and $n_i$ is the multiplicity of $D$ at $D_i$.\footnote{In
 case $D$ is effective, this is the geometric
 multiplicity of $D$, seen as a regular closed subscheme of $X$,
 at the generic point of $D_i$.}
\item Let $f:Y \rightarrow X$ be a dominant morphism of normal
 schemes. Then the pullback divisor $E=f^{-1}(D)$ is defined,
 as a Cartier divisor,
 and if one denotes $(E_j)_{j=1,...,r}$ the family of irreducible
 components of the support $T$ of $E$, and $m_j$ the intersection
 multiplicity of $E_j$ in the pullback of $D$ along $f$
 (\emph{i.e.} the multiplicity of $E_j$ in the Cartier divisor 
 $E$), one has the relation in $\E^{2,1}_T(Y)$:
$$
f^*(\fdl_X(D))=[m_1]_F.\fdl_Y(E_1)+_F \hdots +_F [m_r]_F.\fdl_Y(E_r).
$$
\end{enumerate}
\end{cor}


\subsection{Riemann-Roch formulas}

\begin{num} \label{num:todd}
We now show that we can derive from our theory many
 generalized Riemann-Roch formulas in the sense of 
 Fulton and MacPherson's bivariant theories (\cite[I.1.4]{FMP}).
 This is based on the construction of Todd classes.
 Let us fix a morphism of absolute ring spectra
$$
(\varphi,\phi):(\T,\E) \rightarrow (\T',\F)
$$
as in Definition \ref{df:morph_abs_sp}.

Suppose $c$ (resp. $d$) is an orientation of
 the ring spectrum $\E$ (resp. $\F)$.
 Given a base scheme $S$, we obtain following
 Paragraph \ref{num:functoriality2} a morphism of graded rings:
$$
\phi^{\PP^\infty_S}_*:\E^{**}(\PP^\infty_S)
 \rightarrow \F^{**}(\PP^\infty_S)
$$
--- induced by the Grothendieck transformation of \emph{loc. cit.}
 According to the projective bundle theorem satisfied
 by the oriented ring spectra $(\E_S,c_S)$ and $(\F_S,d_S)$,
 this corresponds to a morphism of rings:
$$
\E^{**}(S)[[u]] \rightarrow \F^{**}(S)[[t]]
$$
and we denote by $\Psi_\phi(t)$ the image of $u$ by this map.
In other words,
 the formal power series $\Psi_\phi(t)$ is characterized by the relation:
\begin{equation} \label{eq:ass_power_series}
\phi^{\PP^\infty_S}_*\left(c\right)=\Psi_\phi(d).
\end{equation}
Note that the restriction of $\phi^{\PP^\infty_S}_*(c)$
 to $\PP^0_S$ (resp. $\PP^1_S$) is $0$ (resp. $1$) because
 $c$ is an orientation and $\varphi$ is a morphism of
 ring spectra.
 Thus we can write $\Psi_\phi(t)$ as:
$$
\Psi_\phi(t)=t+\sum_{i>1} \alpha_i^S.t^i
$$
where $\alpha_i^S \in \F^{2-2i,1-i}(S)$.
In  particular, the power series $\Psi_\phi(t)/t$ is invertible.

We next consider the commutative monoid $\mathcal M(S)$
 generated by the isomorphism classes of vector bundles over $S$ 
 modulo the relations $[E]=[E']+[E'']$ coming from exact sequences
$$
0 \rightarrow E' \rightarrow E \rightarrow E'' \rightarrow 0.
$$
Then $\mathcal M$ is a presheaf of monoids on the category
 $\base$ whose associated presheaf of abelian groups is the functor $K_0$.

Note that $\F^{00}(S)$, equipped with the cup-product,
 is a commutative monoid. We will denote by
 $\F^{00\times}(S)$ the group made by its invertible elements.
\end{num}
\begin{prop} \label{prop:todd}
There exists a unique natural transformation of presheaves
 of monoids over the category $\base$
$$
\td_\phi:\mathcal M \rightarrow \F^{00}
$$
such that for any line bundle $L$ over a scheme $S$,
\begin{equation} \label{eq:td_line_bdl}
\td_\phi(L)=\frac t {\Psi_\phi(t)}.d_1(L).
\end{equation}
Moreover, it induces a natural transformation 
 of presheaves of abelian groups:
$$
\td_\phi:K_0 \rightarrow \F^{00\times}.
$$
\end{prop}
The proof is straightforward using the splitting principles
 (see \cite[4.1.2]{Deg12}).

\begin{rem}
According to the construction of the first Chern classes
 for the oriented ring spectra $(\E,c)$ and $(\F,d)$
 together with Relations \eqref{eq:ass_power_series} and \eqref{eq:td_line_bdl},
 we get for any line bundle $L/S$ the following identity in $\F^{2,1}(S)$:
\begin{equation} \label{eq:todd&chern}
\varphi_S\big(c_1(L)\big)
 =\td_\varphi(-L) \cupp d_1(L).
\end{equation}
\end{rem}

\begin{df} \label{df:todd}
Consider the context and notations of the previous
 proposition.

Given any virtual vector bundle $v$ over a scheme $S$,
 the element $\td_\phi(v) \in \F^{00}(S)$
 is called the \emph{Todd class} of $v$ over $S$ associated
 with the morphism of ring spectra $(\varphi,\phi)$.
\end{df}

The main property of Todd classes is the following formula.
\begin{lm}
Consider the above notations and assumptions.

Then for any smooth $S$-scheme $X$ and any virtual vector bundle $v$
 over $X$, the following relation holds in $\F^{**}(\Th(v))$:
$$
\phi_*\left(\rthom^\E(v)\right)=\td_\phi(-v).\rthom^\F(v)
$$
where $\rthom(v)$ denotes the Thom class associated
 with $v$ (see Definition \ref{df:Thom_class_vb}).
\end{lm}
\begin{proof}
Recall that for any virtual bundles $v$ and $v'$ over $X$
 (see Paragraph \ref{num:Thom_class_vb}),
 the tensor product in $\T(S)$ gives a pairing
$$
\otimes_X:\E^{**}(\Th(v)) \otimes_X \E^{**}(\Th(v'))
 \rightarrow \E^{**}(\Th(v+v'))
$$
and similarly for $\F^{**}$. It follows from definitions
 that the 
 natural transformation of cohomology theories
 $\phi_*:\E^{**} \rightarrow \F^{**}$
 is compatible with this product.

Moreover, we have the relations:
\begin{align*}
\rthom(v+v')&=\rthom(v) \otimes_X \rthom(v'), \\
\td_\phi(v+v')&=\td_\phi(v)+\td_\phi(v').
\end{align*}
Therefore, by definition of the Thom
 class of a virtual bundle (see \ref{df:Thom_class_vb}),
 it is sufficient to check the relation
 of the proposition when $v=[E]$ is the class of a vector
 bundle over $X$.
Besides, using again the preceding relations
 and the splitting principle, one reduces to the case of
 a line bundle $L$. But then, in the cohomology of the 
 projective completion $\bar L$ of $L/X$,
 we have the following relation
$$
\thom(L)=c_1(\xi_L)
$$
 where $\xi_L$ is the universal quotient bundle. Thus the desired
 relation follows from relation \eqref{eq:todd&chern} and the
 fact $\td_\phi(\xi_L)=\td_\phi(L)$.
\end{proof}

We can now derive the generalized Riemann-Roch formula.
\begin{thm}\label{thm:RR}
Let $(\T,\E,c)$ and $(\T',\F,d)$ be absolute oriented ring spectra
 together with a morphism of ring spectra:
$$
(\varphi,\phi):(\T,\E) \rightarrow (\T',\F).
$$
Using the notations of the Definitions \ref{df:fdl_classes}
 and \ref{df:todd}, for any gci morphism $f:X \rightarrow S$
 with virtual tangent bundle $\tau_f$, one has the following relation:
$$
\phi_*(\fdl_f^\E)=\td_\phi(\tau_f).\fdl_f^\F.
$$
\end{thm}
\begin{proof}
As $f$ is gci and because of the associativity property of
 our system of fundamental classes, we are reduced to the cases
 where $f=i$ is a regular closed immersion and $f$ is a smooth morphism.

In the first case, we can use the deformation diagram \eqref{eq:deformation}
 and the fact fundamental classes are stable by transversal base change
 to reduce to the case of the zero section $s=f$ of a vector bundle $E/X$.
 Then we recall that the fundamental class $\eta_s$ coincides with
 the Thom class associated with $E$ so that the preceding lemma concludes.

In the second case, we come back to the construction
 of Paragraph \ref{num:fdl_smooth}, and more precisely Remark
 \ref{rem:fdl_smooth}(3). It is clear that the map $\phi_*$
 is compatible with the isomorphisms \eqref{eq:f_purity_iso_smooth1_induced},
 compute either in $\T$ or in $\T'$, as they are all build using exchange
 transformations. So we are reduced again to the case where the
 fundamental class is $\rthom(-T_f)$, which follows from the preceding
 lemma. 
\end{proof}

\begin{ex}\label{ex:generalizedRR}
Let us fix a gci morphism $f:X \rightarrow S$
\begin{enumerate}
\item Given an absolute ring spectrum $\E$,
 we have seen in Paragraph \ref{num:univ_MGL} that the data of an
 orientation $c$ on $\E$ is equivalent to that of a morphism of
 ring spectra
$$
\phi:\MGL \rightarrow \E
$$
such that $\phi_*(c^\MGL)=c$, where $c^\MGL$ is the canonical orientation
 of $\MGL$. In that case, the previous theorem gives us the relation:
$$
\phi_*(\fdl_f^\MGL)=\fdl_f^\E.
$$
In other words, the fundamental classes $\fdl_f$ are all induced by
 the one defined in algebraic cobordism.
\item Next we can apply the previous formula to the morphisms of
 absolute ring spectra of Example \ref{ex:morph_sp}(2).
 Let us fix a prime $\ell$ and consider the two following cases:
\begin{itemize}
\item $\base$ is the category of all schemes, $\Lambda=\QQ$,
 $\Lambda_\ell=\QQ_\ell$;
\item $\base$ is the category of $k$-schemes for a field $k$
 if characteristic $p\neq \ell$, $\Lambda=\ZZ$,
 $\Lambda_\ell=\ZZ_\ell$;
\end{itemize}
Then according to \emph{loc. cit.}, we get a morphism of
 absolute ring spectra:
$$
\rho_\ell:\HH \Lambda \rightarrow \HH_\et \Lambda_\ell,
$$
corresponding to the higher \'etale cycle class.
As the formal group laws associated with the canonical orientations
 on each spectra are additive, the morphism of formal group law associated
 with the induced morphism
 $\rho_\ell:\HH^{**}(-,\Lambda) \rightarrow \HH^{**}_{\et}(-,\Lambda)$
 is the identity. Therefore, the Todd class associated with 
 $\phi$ is constant equal to $1$ and we get:
$$
\rho_\ell(\fdl_f)=\fdl_f^\et.
$$
\item Consider the Chern character
$$
\ch:\KGL \longrightarrow \bigoplus_{i \in \ZZ} \HH \QQ(i)[2i], 
$$
of Example \ref{ex:morph_sp}(3). As explained in \cite[5.3.3]{Deg12},
 the formal group law associated with the canonical orientation of
 $\KGL$ is multiplicative: $F_\KGL(x,y)=x+y-\beta.xy$ where
 $\beta$ is the Bott element in algebraic K-theory
 and the formal group law on rational motivic cohomology is the
 additive formal group law. As $\beta$ is sent to $1$ by the Chern
 character, the morphism of formal group law associated with $\phi$
 is necessarily the exponential one, $t \mapsto 1-\exp(-t)$.
 Therefore, we get the Todd class associated with $\ch_t$ defined on a 
 line bundle $L/X$ as:
$$
\td(L)=\frac {c_1(L)} {1-\exp(-c_1(L))}.
$$
Recall this formula makes sense as $c_1(L)$ is nilpotent in 
 the motivic cohomology ring $\HB^{**}(X)$ (recall the notation
 of Example \ref{ex:spectra}(2)).
So in fact the transformation $\td$ is the usual Todd class in
 motivic cohomology,
 and we have obtained the following generalized Riemann-Roch formula:
$$
\ch(\fdl_f^\KGL)=\td(\tau_f).\fdl_f^{\HH\QQ}.
$$
This generalizes the original formula of Fulton and MacPherson
 (cf. \cite[II, 1.4]{FMP}).
\end{enumerate}
For more examples, we refer the reader to \cite[\textsection 5]{Deg12}.
\end{ex}

\subsection{Gysin morphisms}\label{sec:Gysin}

\begin{num}
Recall that a weakly oriented (Definition \ref{df:weak_orient})
absolute $\T$-spectrum $\E$
 is a $\tau$-module over the ring spectrum $\MGL$
 where $\tau^*:\SH \rightarrow \T$ is the
 premotivic adjunction fixed according to our convention.

Recall from Remark \ref{rem:modules} that we get in particular
 an action:
$$
\biv \MGL {n,m}(Y/X) \otimes \biv \E {s,t}(X/S)
 \rightarrow \biv \E {n+s,m+t}(Y/S).
$$
Then Theorem \ref{thm:BM_fundamental_classes} induces the following constructions.
\end{num}
\begin{df}\label{df:Gysin}
Consider the above notation.
 Then for any gci morphism $f:Y \rightarrow X$,
 with relative dimension $d$
 and fundamental class $\fdl_f \in \biv{\MGL}{2d,d}(Y/X)$,
 we define the following \emph{Gysin morphisms}:
\begin{itemize}
\item if $f$ is a morphism of s-schemes over a base $S$,
 one gets a pullback
$$
f^*:\biv \E {**}(X/S) \rightarrow \biv \E {**}(Y/S), x \mapsto \fdl_f.x
$$
homogeneous of degree $(2d,d)$;
\item if $f$ is proper, one gets a pushforward
$$
f_*:\E^{**}(Y)=\biv \E {**}(Y/Y) \xrightarrow{.\fdl_f} \biv \E {**}(Y/X)
 \xrightarrow{f_!} \biv \E {**}(X/X)=\E^{**}(X)
$$
homogeneous of degree $(-2d,-d)$.
\end{itemize}
\end{df}
According to this definition, and the fact we have in fact constructed
 in Theorem \ref{thm:BM_fundamental_classes}
 a system of fundamental classes, which includes in particular
 the compatibility with composition
 (Definition \ref{df:fdl_classes_abstract}),
 we immediately get that these Gysin morphisms are compatible with
 composition.

\begin{rem}
\begin{enumerate}
\item When we consider the stronger case of an absolute oriented ring spectrum
 $\E$ in the sense of Definition \ref{df:orientation},
 in the preceding definition,
 one can consider $\fdl_f$ as the fundamental class associated
 with $f$ in $\biv \E {2d,d}(X/S)$ and only use the product
 of the bivariant theory $\biv \E {**}$. In fact, the two definitions
 coincide because of the Grothendieck-Riemann-Roch formula below
 (Proposition \ref{prop:GRR}) and
 the fact the morphism $\phi:\MGL \rightarrow \E$ of ring spectra
 corresponding to the chosen orientation $c$ sends the 
 canonical orientation of $\MGL$ to $c$.
\item Gysin morphisms, in the case of the Borel-Moore homology
 associated with an absolute oriented ring $\T$-spectrum $(\E,c)$,
 extends the one already obtained with respect to
 \'etale morphism in Paragraph \ref{num:duality_etale}.
 This rightly follows from the construction of the fundamental
 class in the case of \'etale morphisms.
\end{enumerate}
\end{rem}

\begin{ex}\label{ex:cov_Gysin_etale&motivic}
We can use the construction of the preceding definition
 in the case of all the ring spectra of Example \ref{ex:spectra}
 (according to Example \ref{ex:orient_spectra}).
 This gives back the Gysin morphisms on representable cohomologies
 as constructed in \cite{Nav}, and notably
 covariant functoriality of Spitzweck integral motivic cohomology
 (cf. Ex. \ref{ex:spectra}(4)) with respect to any gci proper morphism
 of schemes.

The important new case we get out of our theory is given
 when $\Lambda$ is any ring (resp. $\Lambda=\ZZ_\ell, \QQ_\ell$)
 and $\HH_\et\Lambda$ is the \'etale motivic absolute $\Lambda$-spectrum
 (resp. $\ell$-completed \'etale motivic absolute spectrum,
 integral or rational) as in Example \ref{ex:spectra}(3).
 The Gysin morphisms obtained here, for the corresponding cohomology
 and any gci proper morphism of schemes,
 cannot be deduced from Navarro's result (as explained in the end
 of Remark \ref{rem:SH_initial&ring_sp}).

When $\Lambda$ is a torsion ring, this gives covariant 
 functoriality for the classical \'etale cohomology with
 $\Lambda$-coefficients, for any gci proper morphism.
 This was known for flat proper morphisms by \cite[XVII, 2.13]{SGA4}
 and for proper morphisms between regular schemes by \cite[6.2.1]{Deg12}.
\end{ex}

\begin{rem}
The only class of morphisms containing both flat morphisms
 and local complete morphisms is the class of morphisms
 of finite Tor dimension.\footnote{Recall: a morphism of schemes
 $f:Y \rightarrow X$ is of finite Tor dimension if
 $\mathcal O_Y$ is a module of finite Tor dimension over
 $f^{-1}(\mathcal O_X)$.}
 This seems to be the largest class of morphisms for which
 fundamental classes an exceptional functoriality can exist. Indeed,
  we have the example of Quillen's higher algebraic G-theory:
 it is contravariant with respect to such morphisms
 as follows from \cite[\textsection 7, 2.5]{Qui}.
 Unfortunately, our method is powerless to treat this generality.
\end{rem} 

\begin{ex}\label{ex:Gysin_BM_hlg}
Another set of examples is obtained in the case of Borel-Moore
 homology. So applying Example \ref{ex:biv_theories},
 we get contravariant functoriality with respect
 to gci morphisms of $S$-schemes, of the following theories:
\begin{itemize}
\item Bloch's higher Chow groups, when $S$ is the spectrum of a field.
 This was previously known only for morphisms of smooth schemes according
 to constructions of Bloch and Levine.
\item for Borel-Moore \'etale homology, both in the case $S$ is
 the spectrum of a field (classically considered) and in the
 case $S$ is an arbitrary scheme.
\item Note also that according to Example \ref{ex:biv_theories}(3),
 we get contravariance of Thomason's G-theory,
 or equivalently, Quillen K'-theory with respect to any
 gci morphism of s-schemes over a regular base. This contravariance
 coincides with the classical one but we will not check that
 here.\footnote{In the quasi projective,
 one can directly use the analogue of Theorem \ref{thm:uniqueness} for
 Gysin morphisms. The general case requires to identify Borel-Moore
 homology with coefficients in $\KGL$ with a suitable bivariant version
 of $G$-theory.}
\end{itemize}
\end{ex}

The properties of fundamental classes obtained in the
 beginning of this section immediately translate to properties
 of Gysin morphisms.
\begin{prop}\label{prop:excess_Gysin}
Let $\E$ be a  weakly oriented absolute $\T$-spectrum
 (Definition \ref{df:weak_orient}).
 Consider a cartesian square of $S$-schemes:
$$
\xymatrix@=14pt{
Y'\ar^{g}[r]\ar_q[d]\ar@{}|\Delta[rd] & X'\ar^p[d] \\
Y\ar_f[r] & X.
}
$$
such that $f$ is gci and let $\xi \in K_0(X')$
 be the excess intersection bundle (see \ref{prop:excess_generalized}),
 $e=\rk(\xi)$.
 Then the following formulas hold:
\begin{itemize}
\item If $p$ is proper, for any $x' \in \biv \E {**}(X'/S)$,
 one has: $f^*p_*(x')=q_*(c_e(\xi).f^*(x'))$ in $\biv \E {**}(Y/S)$.
\item If $f$ is proper,  for any $y \in \E^{**}(Y)$,
 one has: $p^*f_*(y)=g_*(c_e(\xi).q^*(y))$ in $\E^{**}(X')$.
\end{itemize}
Assume moreover that $\E$ is an absolute oriented ring $\T$-spectrum.
 Then, if $f:Y \rightarrow X$ is gci and proper, for any pair
 $(x,y) \in \E^{**}(X) \times \E^{**}(Y)$, one gets the classical
 projection formula:
$$
f_*(f^*(x).y)=x.f_*(y).
$$
\end{prop}
Given Definition \ref{df:Gysin},
 the first two assertions are mere consequences
 of \ref{prop:excess_generalized}
 as follows from the properties of bivariant theories together
 with the commutativity property of the product on $\MGL$
 (see Proposition \ref{prop:orient_comm}).
 The last assertion easily follows from the second projection formula
 of the axioms of bivariant theories (as recalled in Paragraph
 \ref{num:product}).

\begin{rem}
Other projection formulas can be obtained,
 for products with respect to bivariant theories and for 
 modules over absolute oriented ring spectra. In each case,
 the formulation is straightforward, as well as their proof so
 we left them to the reader.
\end{rem}

\begin{num}
Consider an absolute oriented ring $\T$-spectrum $(\E,c)$
 and a proper gci morphism $f:X \rightarrow S$.
 Applying the above definition, we get a Gysin morphism
 $f_*:\E^{**}(X) \rightarrow \E^{**}(S)$.
 The fundamental class in cohomology defined in
 \ref{df:fdl_coh} is simply: $\cfdl_f=f_*(1)$
 where $1$ is the unit of the ring $\E^{**}(X)$.

This is the classical definition,
 and we can derive from the properties of $f_*$ several
 properties of fundamental classes.
 As an illustration, we note that the preceding projection formula
 (and the graded commutativity of cup-product) immediately gives the
 following abstract degree formula:
$$
f_*f^*(x)=\cfdl_f.x.
$$
Moreover, one can compute $\cfdl_f$ in many cases
 (see \cite[2.4.6, 3.2.12, 5.2.7]{Deg12}). Let us give
 an interesting example when $f$ is finite.
\end{num}
\begin{prop}
Let $f:X \rightarrow S$ be a finite lci morphism such that
 there exists a factorization:
$$
X \xrightarrow i \PP^1_S \xrightarrow p S
$$
where $p$ is the projection of the projective line over $S$ and
 $i$ is a closed immersion. Let $d$ be the Euler characteristic
 of the perfect complex $\mathrm R f_*(\mathcal O_X)$ over $S$,
 seen as a locally constant function on $S$, and $L$ be the line bundle
 on $\PP^1_S$ corresponding to the immersion $i$.

Then the invertible sheaf $L(d)$ can be written $L(d)=p^*(L_0)$
 where $L_0$ is a line bundle on $S$ and 
 the following formula holds in $\E^{00}(S)$:
$$
\cfdl_f=
d+(d-1).a_{11}.c_1(L_0)+a_{12}.c_1(L_0)^2+a_{13}.c_1(L_0)^3+\hdots
$$
where $a_{ij}$ are the coefficients of the formal group law
 associated with the orientation $c$ of $\E$ over $S$.
\end{prop}
In particular, if the formal group law of $\E$ is additive,
 or if $L_0/S$ is trivial, we get the usual degree formula:
$$
f_*f^*(x)=d.x.
$$
\begin{proof}
Let $\lambda=\mathcal O(-1)$ be the canonical line bundle on $\PP^1_S$.
 According to our assumptions on $f$, we get an isomorphism:
$$
L=p^{-1}(L_0)(-d)=\lambda^{\otimes,d} \otimes p^{-1}(L_0)
$$
where $L_0/S$ is the line bundle expected in the first assertion
 of the previous statement. In particular, if we denote by $F(x,y)$
 the formal group law associated with $(\E,c)$ over $S$,
 and put $x=c_1(\lambda)$, $y=p^*c_1(L_0)$, one gets:
$$
i_*(1)=c_1(L)=[d]_F.x+_F y
=(d.x)+_F y=(d.x)+y+(d.x).\sum_{i>0} a_{1i}.y^i,
$$
using the fact $x^i=0$ if $i>1$.

Because $p$ is the projection of a projective line
 one obtains the following explicit formulas (see \cite[5.31]{Deg8}): 
$$
p_*(x)=1, p_*(1)=-a_{11}.
$$
Therefore, as $y=p^*(y_0)$ where $y_0=c_1(L_0)$,
 one obtains:
$$
\cfdl_f=p_*(i_*(1))=y_0.\underset{=-a_{11}}{\underbrace{p_*(1)}}+d.\sum_{i \geq 0} a_{1i} y_0^i
$$
\end{proof}

Similarly one gets the following Grothendieck-Riemann-Roch formulas
 from the generalized Riemann-Roch formula of Theorem \ref{thm:RR}.
\begin{prop}\label{prop:GRR}
Consider the assumptions of Theorem \ref{thm:RR}.

Then for any gci morphism $f:Y \rightarrow X$ of s-schemes over $S$
 with tangent bundle $\tau_f$,
 the following diagrams are commutative:
$$
\xymatrix@=24pt@C=54pt{
\biv \E {**}(X/S)\ar^{f^*}[r]\ar_{\phi_*}[d] & \biv \E {**}(Y/S)\ar^{\phi_*}[d]
 & \E^{**}(Y)\ar^{f_*}[r]\ar_{\phi_*}[d] & \E^{**}(X)\ar^{\phi_*}[d] \\
\biv \F {**}(X/S)\ar^{\td_\phi(\tau_f).f^*}[r] & \biv \F {**}(Y/S)
 & \F^{**}(Y)\ar^{f_*[\td_\phi(\tau_f).]}[r] & \F^{**}(X) 
}
$$
where in the square on the right hand side
 we assume in addition $f$ is proper.
\end{prop}
Again this follows easily from Theorem \ref{thm:RR} and
 Definition \ref{df:Gysin} given the properties of bivariant theories
 together with Proposition \ref{prop:orient_comm} for the commutativity
 of the product on $\MGL$.

\begin{ex}\label{eq:GRR_gysin}
Our main examples are given by the morphisms of ring spectra
 of Example \ref{ex:morph_sp}, as already exploited in Example
 \ref{ex:generalizedRR}.
\begin{enumerate}
\item Assume we are in one of the following cases:
\begin{itemize}
\item $\base$ is the category of all schemes, $\Lambda=\QQ$,
 $\Lambda_\ell=\QQ_\ell$;
\item $\base$ is the category of $k$-schemes for a field $k$
 if characteristic $p\neq \ell$, $\Lambda=\ZZ$,
 $\Lambda_\ell=\ZZ_\ell$;
\end{itemize}
Then we obtain that the natural transformations induced
 by the $\ell$-adic realization functor gives natural transformations
 on cohomologies and Borel-Moore homologies
 that are compatible with Gysin morphisms.

Note in particular that in the second case, we get that
 the higher cycle class, from higher Chow groups to Borel-Moore \'etale
 homology of any $k$-scheme is compatible with Gysin functoriality
 (here, pullbacks).  
\item The Chern character as in \ref{ex:generalizedRR}(3) gives
 the usual Grothendieck-Riemann-Roch formula from homotopy invariant
 K-theory to motivic cohomology of \cite{Nav}. But we also get
 a Riemann-Roch formula for bivariant theories relative to any base scheme.
 
Let us be more specific in the case where the base scheme is a field $k$.
 Then the Chern character of Example \ref{ex:morph_sp}, applied
 to bivariant theories with respect to the
 s-morphism $X \rightarrow \spec k$,
 gives an isomorphism:
$$
\ch:G_n(X) \rightarrow \bigoplus_{i \in \ZZ} CH_i(Y,n)_\QQ
$$
in view of point (1) and (3) of Example \ref{ex:biv_theories}.
Considering the Todd class functor $\td$ as defined in Example
\ref{ex:generalizedRR}(3), with coefficients in rational motivic
 cohomology, we get for any gci morphism
 $f:Y \rightarrow X$ of separated $k$-schemes of finite type
 the following commutative diagram:
$$
\xymatrix@=24pt@C=54pt{
G_n(X)\ar^{f^*}[r]\ar_{\ch_X}[d]
 & G_n(Y)\ar^{\ch_Y}[d] \\
\bigoplus_{i \in \ZZ} CH_i(X,n)_\QQ\ar^{\td_\phi(\tau_f).f^*}[r]
 & \bigoplus_{i \in \ZZ} CH_i(Y,n)_\QQ.
}
$$
\end{enumerate}
\end{ex}

\begin{num}\label{num:Gysin_support}
In the case of an absolute oriented ring $\T$-spectrum $(\E,c)$,
 the Gysin morphisms can also be obtained very easily using
 the six functors formalism.
 Indeed, consider a commutative diagram:
$$
\xymatrix@C=20pt@R=8pt{
Y\ar_q[rd]\ar^f[rr] && X\ar^p[ld] \\
& S &
}
$$
where $f$ is gci of relative dimension $d$
 and $p$, $q$ are s-morphisms.
 By adjunction, one obtains from the map \eqref{eq:associated_pur_iso}
 associated with the fundamental class of $f$ with coefficients in $\E$
 (Definition \ref{df:fdl_classes}) the following map:
$$
\tfdl'_f:f_!(\E_Y)(d)[2d] \rightarrow \E_X
$$
When $f$ is proper, we deduce the following \emph{trace map},
 well known in the case of \'etale coefficients:
$$
\mathrm{tr}_f:f_*f^*(\E_X)(d)[2d] \simeq f_*(\E_Y) \simeq f_!(\E_Y)
 \xrightarrow{\tfdl'_f} \E_X.
$$
 It is clear that this trace map corresponds
 to the Gysin morphism associated with $f$ in cohomology.

Let us go back to the case where $f$ is an arbitrary gci
 morphism fitting into the above commutative diagram.
 Then, by applying the functor $p_!$ to $\tfdl'_f$, we get
 a canonical map:
$$
q_!(\E_Y)(d)[2d]=p_!f_!(\E_Y)(d)[2d] \xrightarrow{p_!(\tfdl'_f)} p_!(\E_X)
$$
which induces a \emph{covariant functoriality} on cohomology
 with compact supports:
$$
f_*:\E^{**}_c(Y/S) \rightarrow \E^{**}_c(X/S),
$$
morphisms of degree $(2d,d)$. This functoriality extends the one we had
 already seen with respect to \'etale morphisms
 in Paragraph \ref{num:supp_bivariant}.

Finally, if we assume again that $f$ is proper
 we get the following construction that was found by Adeel Khan
 (see also \cite{EHKSY}). From the fundamental class of $f$,
 we get by adjunction a map:
$$
\fdl'_f:f_!(\un_Y)(d)[2d] \rightarrow \E_X.
$$
We deduce the following composite map:
\begin{align*}
f_!f^*p^!(\E_S)(d)[2d]
& \simeq f_!\big(\un_S \otimes f^*p^!(\E_S)\big)(d)[2d]
 \underset{(1)}{\xrightarrow{\ \sim\ }}
 f_!\big(\un_S) \otimes p^!(\E_S)(d)[2d] \\
& \xrightarrow{\fdl'_f \otimes p^!(\E_S)}
 \E_X \otimes p^!(\E_S) \simeq p^*(\E_S) \otimes p^!(\E_S)
 \xrightarrow{Ex^{!*}_\otimes} p^!(\E_S \otimes \E_S)
 \xrightarrow{\mu} p^!(\E_S),
\end{align*}
where (1) is given by the projection formula,
 $Ex^{!*}_\otimes$ by the pairing \eqref{eq:pairing!*}
 and $\mu$ is the product of the ring spectrum $\E_S$.
 Using the adjunctions $(f_!,f^!)$ and $(f^*,f_*)$ and 
 applying the functor $p_!$, we get:
$$
p_!p^!(\E_S)(d)[2d] \rightarrow p_!f_*f^!p^!(\E_S)
 \simeq p_!f_!f^!p^!(\E_S)=q_!q^!(\E_S),
$$
where we have used the fact $f$ is proper.
 This immediately gives the expect contravariant functoriality:
$$
f^*:\E_{**}(X/S) \rightarrow \E_{**}(Y/S)
$$
which is a morphism of degree $(-2d,-d)$. This functoriality
 extends the one already mentioned in Paragraph \ref{num:supp_bivariant}
 in the case where $f$ is a finite morphism.
\end{num}

\begin{rem}
Therefore one has obtained exceptional functorialities
 for all the four theories associated with an absolute oriented ring 
 spectrum $(\E,c)$.

Besides, it is clear that the excess intersection formula 
 (Prop. \ref{prop:excess_Gysin}) and the Riemann-Roch formula 
 (Prop. \ref{prop:GRR}) extends to formulas involving 
 cohomology with compact support and homology. We leave
 the formulation to the reader not to overburden this paper.
\end{rem}

\begin{ex}
Again, one deduces notable examples from \ref{ex:coh_support}.
This gives covariance with respect to gci morphisms of $k$-schemes
 of all the classical cohomology with compact supports
 which corresponds to a Mixed Weil theory.

We also obtain the contravariance with respect to proper gci morphisms
 of complex schemes for the integral Betti homology. Surprisingly,
 this result seems new.
\end{ex}

%% file: purity.tex
\subsection{Purity for closed pairs}

\begin{num}
We will say that a closed pair $(X,Z)$ is \emph{regular}
 if the corresponding immersion $Z \rightarrow X$ is regular.
 Then, in  the deformation diagram \eqref{eq:deformation}
\begin{equation}
\begin{split}\label{eq:deformation_diag}
\xymatrix@R=10pt@C=20pt{
Z\ar[r]\ar_i[d] & \AA^1_Z\ar^\nu[d] & Z\ar^s[d]\ar[l] \\
X\ar^-{d_1}[r] & D_ZX & N_ZX,\ar_-{d_0}[l]
}
\end{split}
\end{equation}
the closed immersion $\nu$ is also regular.

The next definition is an obvious extension of \cite[1.3.2]{Deg12}.
\end{num}
\begin{df} \label{df:abs_pur_closed}
Let $\E$ be an absolute $\T$-spectrum
 and $(X,Z)$ be a regular closed pair.
\begin{enumerate}
\item We say that $(X,Z)$ is \emph{$\E$-pure} if the morphisms 
$$
\E^{**}_Z(X)=\E^{**}(X,Z) \xleftarrow{d_1^*} \E^{**}(D_ZX,\AA^1_Z)
 \xrightarrow{d_0^*} \E^{**}(N_ZX,Z)=\E^{**}(\Th(N_ZX))
$$
induced by the deformation diagram \eqref{eq:deformation} are isomorphisms.
\item We say that $(X,Z)$ is \emph{universally $\E$-pure} if
 for all smooth morphism $Y \rightarrow X$, the closed pair
 $(Y,Y \times_X Z)$ is $\E$-pure.
\end{enumerate}
\end{df}

\begin{ex}\label{ex:purityMV}
It follows directly from Morel-Voevodsky's purity theorem
 (see Th. \ref{thm:purity_MV}) that any closed pair $(X,Z)$
 of smooth schemes over some base $S$ is universally $\E$-pure.
\end{ex}

We can link this definition with Fulton-MacPherson's theory
 of \emph{strong orientations} (Def. \ref{df:strong_orientation})
 as follows.
\begin{prop}\label{prop:abs_pur&strong}
Let $(\E,c)$ be an absolute oriented ring $\T$-spectrum
 and $(X,Z)$ be a regular closed pair. Consider
 the notations of diagram \eqref{eq:deformation_diag}.
 Then the following conditions are equivalent:
\begin{enumerate}
\item[(i)] The closed pair $(X,Z)$ is $\E$-pure.
\item[(ii)] The orientations $\fdl_i$ and $\fdl_\nu$, associated
 with the orientation  $c$ in Definition \ref{df:fdl_classes},
 are strong.
\end{enumerate}
\end{prop}
\begin{proof}
The proof can be summarized in the commutativity of the following
 diagram:
$$
\xymatrix@R=14pt@C=30pt{
\E^{**}(Z)\ar_{.\fdl_i}[d]
 & \E^{**}(\AA^1_Z)\ar_{.\fdl_\nu}[d]\ar_-{s_1^*}^\sim[l]\ar^-{s_0^*}_\sim[r]
 & \E^{**}(Z)\ar^{.\rthom(N_ZX)}_\sim[d] \\
\E^{**}_Z(X) & \E^{**}_{\AA^1_Z}(D_ZX)\ar_-{d_1^*}[l]\ar^-{d_0^*}[r]
 & \E^{**}(\Th(N_ZX))
}
$$
The diagram is commutative according to the stability of fundamental
 classes with respect to transversal pullbacks
 (Example \ref{ex:transversal}) and the fact
 $\fdl_s=\rthom(N_ZX)$ (Remark \ref{rem:fdl_retraction&thom}).
 In this diagram, all arrows indicated
 with a symbol $\sim$ are obviously isomorphisms.
 Condition (i) (resp. (ii)) says that the maps $d_0^*$ and $d_1^*$
 (resp. $(.\fdl_i)$ and $(.\fdl_\nu)$) are isomorphisms.
 Thus the equivalence stated in this proposition obviously follows . 
\end{proof}

\begin{num}
To formulate stronger purity results,
 we now fix a full sub-category $\base_0$
 of $\base$ stable under the following operations:
\begin{itemize}\label{item:ass_base0}
\item For any scheme $S$ in $\base_0$, any smooth $S$-scheme
belongs to $\base_0$.
\item For any regular closed immersion $Z \rightarrow S$ in $\base_0$,
 the schemes $N_ZX$ and $B_ZX$ belong to $\base_0$.
\end{itemize}
 The main examples we have in mind 
 are the category $\reg$ of regular schemes (in $\base$)
 and the category $\sm_S$ of smooth $S$-schemes for a scheme $S$ in $\base$.

Following again \cite[1.3.2]{Deg12}, we introduce the following
 useful definition.
\end{num}
\begin{df}\label{df:abs_pur}
Let $\E$ be an absolute $\T$-spectrum.
\begin{enumerate}
\item We say that $\E$ is \emph{$\base_0$-pure}
 if for any regular closed pair $(X,Z)$ such that
 $X$ and $Z$ belongs to $\base_0$, $(X,Z)$ is $\E$-pure.
\item We say that $\T$ is \emph{$\base_0$-pure}
 if the unit cartesian section $\un$ of the fibred category $\T$
 is \emph{$\base_0$-pure}.
\end{enumerate}
Finally, we will simply say \emph{absolutely pure} for $\reg$-pure.
\end{df}

\begin{rem}
The last definition already appears in \cite[A.2.9]{CD4}
 --- in \emph{loc. cit.} one says $\T$ satisfies the absolute purity property.
\end{rem}

\begin{ex}\label{ex:S_0-pure}
\begin{enumerate}
\item From Example \ref{ex:purityMV}, 
 all absolute $\T$-spectra $\E$,
 as well as all motivic triangulated categories $\T$,
 are $\sm_S$-pure for any scheme $S$
 (even a singular one).
\item Let $k$ be a perfect field whose spectrum is in $\base$.
 Assume that $\E$-cohomology with support is compatible with projective limit
 in the following sense:
 for any essentially affine projective system of closed $k$-pairs 
 $(X_\alpha,Z_\alpha)_{\alpha \in A}$ whose projective
 limit $(X,Z)$ is still in $\base$,
 the canonical map:
$$
\ilim{\alpha \in A^{op}} \big(\E^{n,m}_{Z_\alpha}(X_\alpha)\big)
 \rightarrow \E^{n,m}_{Z}(X)
$$
is an isomorphism.
 This happens in particular if $\T$ is continuous in the sense
 of \cite[4.3.2]{CD3}.

Then one can deduce from Popescu's theorem that any regular $k$-pair
 $(X,Z)$ in $\reg$ is $\E$-pure. In other words, $\E$
 is $(\reg/k)$-pure (see \cite[1.3.4(2)]{Deg12}). 
 This fact concerns in particular the spectra
 of points (5) and (6) in Example \ref{ex:spectra}. 
 
Moreover, one deduces that any continuous motivic triangulated category $\T$
 is $(\reg/k)$-pure. This includes modules over a mixed Weil theory, 
 $\DM_\cdh(-,\ZZ[1/p])$ where $p$ is the characteristic of $k$
 (see \cite[ex. 5.11]{CD5} for the continuity statement).
\end{enumerate}
\end{ex}

\begin{ex}
The following absolute spectra (Example \ref{ex:morph_sp})
 are absolutely pure:
\begin{enumerate}
\item The homotopy invariant K-theory spectrum $\KGL$
 (see \cite[13.6.3]{CD2});
\item given any $\QQ$-algebra $\Lambda$,
 the motivic ring spectrum $\HH\Lambda$
 (see \cite[5.6.2 and 5.2.2]{CD4});
\item the rationalization $\MGL \otimes \QQ$ of algebraic cobordism 
 (this follows from the preceding example and \cite[10.5]{NOS});
\item Given any ring $\Lambda$, the \'etale motivic ring spectrum
 $\HH_\et\Lambda$ (see \cite[5.6.2]{CD4}).\footnote{Recall the case where
 $\Lambda$ is a torsion ring, or $\Lambda=\ZZ_\ell, \QQ_\ell$, follows
 directly from Thomason's purity theorem (\cite{Tho}).}
\end{enumerate}
\end{ex}

\begin{ex}\label{ex:abs_pure_cat}
The ring spectra of the previous examples all corresponds
 to the following (non-exhaustive) list
 of absolutely pure motivic triangulated categories:
\begin{enumerate}
\item the category $\smod{\KGL}$ of $\KGL$-modules (see \cite[\textsection 13.3]{CD2});
\item $\DMB$, $\DM_h(-,\Lambda)$ where $\Lambda$ is a $\QQ$-algebra
   (see \cite{CD2, CD5});
\item $\mathrm D^b_c((-)_\et,\Lambda)$ where $\Lambda=\ZZ/\ell^n$, $\ZZ_\ell$, $\QQ_\ell$
 and $\base$ is the category of $\ZZ[1/\ell]$-schemes;
 \item the category of $(\MGL \otimes \QQ)$-modules (see \cite[\textsection 13.3]{CD2});
\item $\DM_h(-,\Lambda)$ where $\Lambda$ is any ring (see \cite[\textsection 5]{CD4});
\item for any prime $\ell$, the $\ell$-completed category
 $\DM_h(-,\Lambda)$ where $\Lambda=\ZZ_\ell, \QQ_\ell$ (see \cite[\textsection 7.2]{CD4}). 
\end{enumerate}
\end{ex}

\subsection{Dualities}

\begin{num}
Recall that given an oriented ring $\T$-spectrum
 $(\E,c)$, we have associated to a gci morphism $f:X \rightarrow S$
  with relative dimension $d$
  the fundamental class $\fdl_f$ (Definition \ref{df:fdl_classes})
  and equivalently --- equation \eqref{eq:associated_pur_iso} --- a morphism:
$$
\tfdl_f:\E_S(d)[2d] \rightarrow f^!(\E_X).
$$
Recall we say the orientation $\fdl_f$ is universally strong when $\tfdl_f$
 is an isomorphism (Definition \ref{df:strong_orientation}).
\end{num}
\begin{prop}
Consider the preceding notations and a subcategory $\base_0 \subset \base$
 as in Definition \ref{df:abs_pur}. Then the following conditions are equivalent:
\begin{enumerate}
\item[(i)] $\E$ is $\base_0$-pure;
\item[(ii)] for any regular closed immersion $i$ in $\base_0$,
 $\fdl_i$ is a strong orientation;
\item[(ii')] for any gci morphism $f$ in $\base_0$,
 $\fdl_f$ is a strong orientation.
\end{enumerate}
\end{prop}
This is obvious given definitions and Proposition \ref{prop:abs_pur&strong}.

As a corollary, we get the following various formulations of duality
 statements.
\begin{cor}\label{cor:duality}
Let $(E,c)$ be an absolute oriented ring $\T$-spectrum
 which is $\base_0$-pure, following the notations of the previous
 proposition. Let $f:X \rightarrow S$ be a gci s-morphism
 of relative dimension $d$
 and $Y/X$ be an arbitrary s-scheme. Then the following
 maps are isomorphisms:
\begin{align}
\label{eq:duality1}
\delta_f:\E^{n,i}(X) & \rightarrow \biv \E {2d-n,d-i} (X/S),
 x \mapsto x.\fdl_f, \\
\label{eq:duality2}
\delta_f:\biv \E {n,i} (Y/X) & \rightarrow \biv \E {2d+n,d+i} (Y/S),
 y \mapsto y.\fdl_f, \\
\label{eq:duality3}
\delta_f^c:\E^{n,i}_c(X/S) & \rightarrow \E_{2d-n,d-i}(X/S),
 x \mapsto x \cap \fdl_f
\end{align}
where the first two maps are defined using the product of
 Borel-Moore homology and the last one using the cap-product
 \eqref{eq:cap}.
\end{cor}
Recall that these duality isomorphisms occur in particular
 whenever $X/S$ is a smooth s-scheme (Example \ref{ex:S_0-pure}(1)).

\begin{ex}
\begin{enumerate}
\item As a notable particular case of isomorphism
 \eqref{eq:duality2}, we get the formulation of duality with
 support due to Bloch and Ogus (\cite{BO}). Assume $Z=Y \rightarrow X$
 is a closed immersion, $S=\spec k$ and $X$ is smooth over $k$. Then the
 duality isomorphism \eqref{eq:duality2} has the form:
$$
\E^{n,i}_Z(X) \simeq \biv \E {2d-n,d-i}(Z/k).
$$
As an example, we get the following identification, when $p$ is the
 characteristic exponent of $k$:
$$
H^n_Z(X,\ZZ[1/p]) \simeq CH_{d-i}(Z,2i-n)[1/p]
$$
where the left hand side is Voevodsky motivic cohomology
 of $X$ with support in $Z$ and the right hand side is Bloch's higher
 Chow group (see Example \ref{ex:biv_theories}(1)).
\item In the extension of the preceding case,
 assume $X$ and $S$ are regular schemes and $Z=Y \subset X$
 is a closed subscheme. Then the duality isomorphism
 in the case of the absolutely pure spectrum $\KGL$ gives
 the classical duality with support isomorphism
 (see \cite{Sou}):
$$
K_{2i-n}^Z(X) \simeq \KGL^{n,i}_Z(X)
 \simeq \biv \KGL {2d-n,d-i} (Z/S) \simeq K'_{2i-n}(Z).
$$
\item Let $X/S$ be a smooth proper scheme. Then, the four
 theories defined in this paper coincide through isomorphisms
 pictured as follows:
$$
\xymatrix@R=14pt@C=26pt{
\E^{n,i}(X)\ar^-{\delta_f}_-\sim[r]\ar_\sim[d]
 & \biv \E {2d-n,d-i} (X/S)\ar^\sim[d] \\
\E^{n,i}_c(X/S)\ar^-{\delta_f^c}_-\sim[r] & \E_{2d-n,d-i}(X/S)
}
$$
\item An interesting application of the duality isomorphisms
 obtained above is the following identification,
 for a regular s-scheme $X$ over a field $k$,
 of dimension $d$:
$$
H_n^{sing}(X)[1/p]
 \simeq \HH_{n,0}(X/k,\ZZ[1/p])
 \xrightarrow{(\delta_{X/k}^c)^{-1}} H^{2d-n,d-n}_c(X,\ZZ[1/p])
$$
where the left hand side is Suslin homology
 and the right hand side is motivic cohomology
 with compact support. This isomorphism was only known for smooth
 $k$-schemes when $k$ is a perfect field and under
 the resolution of singularities assumption
 (see \cite[chap. 5, Th. 4.3.7]{FSV}).
\end{enumerate}
\end{ex}

\begin{rem}\label{rem:duality&Gysin}
An immediate corollary of the duality isomorphisms
 \eqref{eq:duality1} and \eqref{eq:duality3} is the existence
 of certain Gysin morphisms for the four theories.
 More precisely, under the assumptions of the previous corollary,
 given a scheme $S$ in $\base_0$, we obtain Gysin maps
 for all $S$-morphisms $f:Y \rightarrow X$ in $\base_0$
 such that in addition $Y/S$ and $X/S$ are gci.

In case $f$ is gci, it follows from the definitions that
 the Gysin morphisms obtained as in Definition \ref{df:Gysin}
 and Paragraph \ref{num:Gysin_support} coincides with
 the Gysin morphisms obtained respectively from the isomorphisms
 \eqref{eq:duality1} and \eqref{eq:duality3}.

On the other hand, the morphism $f$ can also simply be a \emph{local}
 complete intersection morphism, so the Gysin morphisms obtained
 in this way are slightly more general.
\end{rem}

We end-up this paper with the following Riemann-Roch-like
 statement, involving the previous duality isomorphisms
 and directly following from the general Riemann-Roch Theorem
 \ref{thm:RR}.
\begin{thm}
Consider a subcategory $\base_0 \subset \base$ as in
 Definition \ref{df:abs_pur}, $(\E,c)$ and $(\F,d)$
 absolute oriented $\base_0$-pure ring spectra.
 We adopt the notations of the previous Corollary.
Consider in addition a morphism of ring spectra:
$$
(\varphi,\phi):(\T,\E) \rightarrow (\T',\F)
$$
 and $\td_\phi:K_0 \rightarrow \F^{00\times}$ the associated
 Todd class transformation (Definition \ref{df:todd}).

Then, given an arbitrary gci morphism $f:X \rightarrow S$
 with virtual tangent bundle $\tau_f$ and relative dimension $d$,
 for any s-scheme $Y/X$, the following diagrams are commutative:
$$
\xymatrix@=58pt@R=20pt{
\E^{n,i}(Y/X)\ar_-\sim^-{\delta_f}[r]\ar_{\phi_*}[d]
 & \biv \E {2d-n,d-i} (Y/S)\ar^{\phi_*}[d]
 &
\E^{n,i}_c(X)\ar^-{\delta_f^c}_-\sim[r]\ar_{\phi_*}[d]
 & \biv \E {2d-n,d-i} (X/S)\ar^{\phi_*}[d] \\
\E^{n,i}(Y/X)\ar^-{\delta_f(-.\td_\phi(\tau_f))}_-\sim[r]
 & \biv \E {2d-n,d-i} (Y/S),
 &
\E^{n,i}_c(X)\ar^-{\td_\phi(\tau_f).\delta_f^c}_-\sim[r]
 & \biv \E {2d-n,d-i} (X/S).
}
$$
\end{thm}
The proof is obvious from the formulas in Corollary \ref{cor:duality}
 and Theorem \ref{thm:RR}.

\begin{rem}
\begin{enumerate}
\item This theorem has to be compared with \cite[I.7.2.2]{FMP}.
\item As indicated to us by Henri Gillet, one immediately deduces
 from this theorem the Grothendieck-Riemann-Roch formulas
 for the Gysin morphisms obtained using duality as in
 Remark \ref{rem:duality&Gysin}.
\end{enumerate}
\end{rem}